\documentclass[preprint,11pt]{elsarticle}%
\usepackage{eurosym}
\usepackage{makeidx}
\usepackage{amsfonts}
\usepackage{amsmath}
\usepackage{amssymb}
\usepackage{graphicx}
\usepackage{bbold}
\usepackage[colorlinks]{hyperref}
\usepackage{palatino}
\usepackage{upgreek}%
\providecommand{\U}[1]{\protect\rule{.1in}{.1in}}
\hypersetup{colorlinks=true, linkcolor=cyan, citecolor=green,
filecolor=black, urlcolor=black }
\providecommand{\U}[1]{\protect\rule{.1in}{.1in}}
\journal{\ldots}
\newtheorem{theorem}{Theorem}
\newtheorem{acknowledgement}[theorem]{Acknowledgement}

\newtheorem{definition}[theorem]{Definition}

\newtheorem{lemma}[theorem]{Lemma}

\newtheorem{proposition}[theorem]{Proposition}
\newtheorem{remark}[theorem]{Remark}

\newenvironment{proof}[1][Proof]{\noindent\textbf{#1.} }{\ \rule{0.5em}{0.5em}}

\begin{document}
\begin{frontmatter}
\title{On the continuity of the probabilistic representation of a semilinear Neumann--Dirichlet problem}
\author{Lucian Maticiuc$^{a,b}$, Aurel R\u{a}\c{s}canu$^{a,c}$}
\address{$^{a}$ Faculty of Mathematics, \textquotedblleft Alexandru
Ioan Cuza\textquotedblright\ University \\ Carol 1 Blvd., no. 11,
Ia\c{s}i, Romania \medskip \\$^{b}$~Department of Mathematics,
\textquotedblleft Gheorghe Asachi\textquotedblright\ Technical University of Ia\c{s}i\\
Carol I Blvd., no. 11, 700506, Romania \medskip \\$^{c}$ \textquotedblleft Octav Mayer\textquotedblright\ Mathematics Institute of the Romanian Academy, Iasi branch \\ Carol I Blvd., no. 8, Iasi, 700506, Romania}
\begin{abstract}
In this article we prove the continuity of the deterministic function $u:%
\left[ 0,T\right] \times \mathcal{\bar{D}}\rightarrow \mathbb{R}$, defined
by $u\left( t,x\right) :=Y_{t}^{t,x}$, where the process $%
(Y_{s}^{t,x})_{s\in \left[ t,T\right] }$ is given by the generalized
multivalued backward stochastic differential equation:%
\begin{equation*}
\left\{
\begin{array}{r}
-dY_{s}^{t,x}+\partial \varphi(Y_{s}^{t,x})ds+\partial\psi(Y_{s}^{t,x})dA_{s}^{t,x}\ni f(s,X_{s}^{t,x},Y_{s}^{t,x})ds\medskip \\
+g(s,X_{s}^{t,x},Y_{s}^{t,x})dA_{s}^{t,x}-Z_{s}^{t,x}dW_{s}~,\;t\leq
s < T,\medskip  \\
\multicolumn{1}{l}{Y_{T}=h(X_{T}^{t,x}).}%
\end{array}%
\right.
\end{equation*}%
The process $(X_{s}^{t,x},A_{s}^{t,x})_{s\geq t}$ is the
solution of a stochastic differential equation with reflecting boundary
conditions.
\end{abstract}
\end{frontmatter}


\textbf{Keywords or phrases:} Feynman--Kac formula; Reflected diffusion
processes; Continuity w.r.t. initial data; Neumann--Dirichlet boundary conditions.

\renewcommand{\thefootnote}{\fnsymbol{footnote}}
\footnotetext{\textit{{\scriptsize E-mail addresses:}}
{\scriptsize lucian.maticiuc@uaic.ro (Lucian Maticiuc), aurel.rascanu@uaic.ro
(Aurel R\u{a}\c{s}canu).}}

\section{Introduction}

It is well known that the probability methods represent often a tool to infer
the results from the deterministic analysis by solving some stochastic
equations. We refer here to the Feynman--Kac formula which allows to represent
the solution of the parabolic equation $\dfrac{\partial u}{\partial t}\left(
t,x\right)  -\dfrac{1}{2}\Delta u\left(  tx\right)  +c\,u\left(  t,x\right)
=h\left(  t,x\right)  $, $\left(  t,x\right)  \in(0,T]\times\mathbb{R}^{d}$
with the initial condition $u\left(  0,x\right)  =\kappa\left(  x\right)  $,
$x\in\mathbb{R}^{d}$. A similar formula occurs in the case of a problem with
boundary conditions (see, e.g., \cite{hsu/85}). With the advent of the
backward stochastic differential equations (BSDEs for short) it has become
possible to extend the Feynman--Kac formula to semilinear parabolic or
elliptic partial differential equations (PDEs for short) with various type of
boundary conditions (see \cite{pe/91}, \cite{pa-pe/92}, \cite{hu/93},
\cite{pa-pr-ra/97} and \cite{da/97}). In 1998 Pardoux \& Zhang proved in
\cite{pa-zh/98} a probabilistic formula for the viscosity solution of a system
of semilinear PDEs with Neumann boundary condition%
\[
\left\{
\begin{array}
[c]{l}%
\dfrac{\partial u_{i}}{\partial t}(t,x)+\mathcal{L}_{t}u_{i}\left(
t,x\right)  +f\big(t,x,u(t,x),\left(  \nabla u_{i}\sigma\right)  \left(
t,x\right)  \big)=0,~t\in(0,T],\;x\in\mathcal{D},\medskip\\
\dfrac{\partial u_{i}}{\partial n}(t,x)=g\big(t,x,u(t,x)\big),~t\in
\lbrack0,T],\;x\in\mathrm{Bd}\left(  \mathcal{D}\right)  ,\vspace*{2mm}\\
u(T,x)=h(x),\;x\in\overline{\mathcal{D}},\;i=\overline{1,k}\,,
\end{array}
\right.
\]
where $\mathcal{L}_{t}$ is a second--order differential operator defined by%
\[
\mathcal{L}_{t}v(x)=\dfrac{1}{2}\mathrm{Tr}\big[\sigma(t,x)\sigma^{\ast
}(t,x)D^{2}v(x)\big]+\big\langle b(t,x),\nabla v(x)\big\rangle,\;\text{for
}v\in C^{2}(\mathbb{R}^{d}).
\]
and $\mathcal{D}$ is an open connected bounded subset of $\mathbb{R}^{d}$ of
the form $\mathcal{D}=\left\{  x\in\mathbb{R}^{d}:\ell\left(  x\right)
<0\right\}  $ with the boundary $\mathrm{Bd}\left(  \mathcal{D}\right)
=\left\{  x\in\mathbb{R}^{d}:\ell\left(  x\right)  =0\right\}  $, where
$\ell\in C_{b}^{3}\left(  \mathbb{R}^{d}\right)  $ and $\left\vert \nabla
\ell\left(  x\right)  \right\vert =1,\;$for all $x\in\mathrm{Bd}\left(
\mathcal{D}\right)  $. The outward normal derivative of a function $v\in
C^{1}\left(  \mathrm{Bd}\left(  \mathcal{D}\right)  \right)  $ is given by
$\dfrac{\partial v\left(  x\right)  }{\partial n}=\left\langle \nabla
\ell\left(  x\right)  ,\nabla v\left(  x\right)  \right\rangle $ for all
$x\in\mathrm{Bd}\left(  \mathcal{D}\right)  .$

Afterwards, in \cite{ma-ra/10}, the authors proved a generalized Feynman--Kac
formula in order to represent the viscosity solution of the following
parabolic variational inequality with a mixed nonlinear multivalued
Neumann--Dirichlet boundary condition driven by subdifferential operators
${\partial\varphi}$ and ${\partial\psi}$ (associated to the convex lower
semicontinuous (l.s.c. for short) functions $\varphi,\psi:\mathbb{R\rightarrow
(-}\infty,+\infty]$):
\begin{equation}
\left\{
\begin{array}
[c]{l}%
\dfrac{\partial u(t,x)}{\partial t}-\mathcal{L}_{t}u\left(  t,x\right)
+{\partial\varphi}\big(u(t,x)\big)\ni f\big(t,x,u(t,x)\big),~t>0,\;x\in
\mathcal{D},\medskip\\
\dfrac{\partial u(t,x)}{\partial n}+{\partial\psi}\big(u(t,x)\big)\ni
g\big(t,x,u(t,x)\big),~t>0,\;x\in\mathrm{Bd}\left(  \mathcal{D}\right)
,\vspace*{2mm}\\
u(0,x)=h(x),\;x\in\overline{\mathcal{D}}.
\end{array}
\right.  \label{multivalued PDE}%
\end{equation}
More precisely, it was proved that the deterministic function $u:[0,T]\times
\overline{\mathcal{D}}\rightarrow\mathbb{R}$, given by the probabilistic
representation formula%
\begin{equation}
u\left(  t,x\right)  :=Y_{t}^{t,x},\;\;(t,x)\in\lbrack0,T]\times
\overline{\mathcal{D}}, \label{repres. formula}%
\end{equation}
where $(Y_{s}^{t,x})_{s\in\left[  t,T\right]  }$ is the unique solution of a
proper backward stochastic variational inequality, is the unique viscosity
solution of the above multivalued problem.

A link between backward stochastic variational inequalities with oblique
subgradients and the viscosity solution for a semilinear parabolic variational
inequality of type (\ref{multivalued PDE}), without boundary conditions but
featuring an oblique reflection, was constructed in \cite{ga-ra-ro/14}.
Another generalization was recently made in \cite{ni 1/14, ni 2/14} by
considering a fully coupled forward--backward stochastic variational
inequality and its associate generalized quasilinear parabolic variational
inequality of type (\ref{multivalued PDE}) on the whole space.$\smallskip$

The aim of this paper is to provide a proof for the continuity of the function
$\left(  t,x\right)  \mapsto u\left(  t,x\right)  =Y_{t}^{t,x}$. We mention
that the proof of the continuity given in \cite[Corollary 14--(c)]{ma-ra/10}
is not correct, since inequality (40) from \cite[Proposition 13]{ma-ra/10} has
a missing term (for the correct statement of \cite[Proposition 13]{ma-ra/10}
see the last section of this paper). Our main result constitutes the correct
proof of point $\left(  c\right)  $ of Corollary 14 from \cite{ma-ra/10}.

In order to obtain the principal result we should assume the additional
condition (\ref{Assumpt. 7}) (see the next section). Moreover, we restrict
ourselves to the case where coefficient $f$ does not depend on $Z.$ Based on
the remark that $\left(  t,x\right)  \mapsto Y_{t}^{t,x}$ is a deterministic
function, the idea used in this paper is to prove that, for any $\left(
t_{n},x_{n}\right)  \rightarrow\left(  t,x\right)  ,$ the sequence $\left(
Y^{n}\right)  _{n\in\mathbb{N}}:=\left(  Y^{t_{n},x_{n}}\right)
_{n\in\mathbb{N}}$ is tight with respect to the \textrm{S}--topology
(Jakubowski's topology \cite{ja/97}) on the space $\mathbb{D}\left(  \left[
0,T\right]  ,\mathbb{R}^{m}\right)  $ of c\`{a}dl\`{a}g functions and changing
the probability space to have the convergence almost sure on a subsequence.
Similar ideas can be found in \cite[Section 4]{ba-ma-za/13}. We work with the
\textrm{S}--topology because we need the continuity of the application
$\mathbb{D}\ni y\longmapsto\int_{0}^{s}g\left(  r,y\left(  r\right)  \right)
dA_{r}$, where $h$ is a continuous function and $A$ is a continuous
non--decreasing function. This property is not true in Meyer--Zheng topology
(unless the measure induced by $A$ is absolutely continuous with respect to
the Lebesgue measure).$\medskip$

We emphasize that, in our opinion, the techniques presented in our paper are
very useful in various cases presented in many other papers; starting with
\cite{pa-zh/98} the viscosity solution of various types of parabolic PDEs with
Neumann boundary condition, via probabilistic methods, represent the subject
of: \cite{bo-ca/04, re-xi/06, bo-ca-mr/07, ri/09, di-ou/10, ra-zh/10,
re-ot/10, am-mr/13}, and all of them used the continuity of the function
$\left(  t,x\right)  \mapsto u\left(  t,x\right)  $ defined through the
solution of a suitable backward equation.\medskip

The article is organized as it follows: In Section 2 we recall the notations,
assumptions and the existence results for the forward--backward stochastic
system envisaged by our work. Section 3 presents the main result of the paper,
while Section 4, Annexes, deals with some auxiliary results which concern
bounded variation functions in the c\`{a}dl\`{a}g case as well as passing to
the limit theorems. The last section, Erratum, presents the new statement of
Proposition 13 from \cite{ma-ra/10}.

\section{Preliminaries}

We adopt the notations and assumptions used in \cite{ma-ra/10}.$\smallskip$

Throughout this paper, $\left(  W_{t}\right)  _{t\geq0}$ denotes a
$d$--dimensional standard Brownian motion defined on a complete probability
space $(\Omega,\mathcal{F},\mathbb{P})$. For $s,t\geq0$, $\mathcal{F}_{s}^{t}$
denotes the $\sigma$--algebra $\sigma(\mathbb{1}_{N},W_{r}-W_{t};t\leq r\leq
s\vee t,N\in\mathcal{N)}$, where $\mathcal{N}$ is the set of $\mathbb{P}%
$--null events of $\mathcal{F}$.

Let $\mathcal{D}$ be a open connected bounded subset of $\mathbb{R}^{d}$ of
the form%
\[
\mathcal{D}=\{x\in\mathbb{R}^{d}:\ell\left(  x\right)  <0\},\quad
\mathrm{Bd}\left(  \mathcal{D}\right)  =\{x\in\mathbb{R}^{d}:\ell\left(
x\right)  =0\},
\]
where $\ell\in C_{b}^{3}(\mathbb{R}^{d})$, $\left\vert \nabla\ell\left(
x\right)  \right\vert =1,\;$for all $x\in\mathrm{Bd}\left(  \mathcal{D}%
\right)  $.

The stochastic process $(Y_{s}^{t,x})_{s\in\left[  0,T\right]  }$ from the
representation formula (\ref{repres. formula}) is defined through the
following stochastic problem.

We fix $T>0.$ For each $\left(  t,x\right)  \in\left[  0,T\right]
\times\mathcal{\bar{D}}$ arbitrary fixed, let%
\[
(X_{s}^{t,x},A_{s}^{t,x},Y_{s}^{t,x},Z_{s}^{t,x},U_{s}^{t,x},V_{s}%
^{t,x}):\Omega\rightarrow\mathbb{R}^{d}\times\mathbb{R\times R}^{m}%
\times\mathbb{R}^{m\times d}\times\mathbb{R}^{m}\times\mathbb{R}^{m},\quad
s\in\left[  0,T\right]
\]
be a sextuple of $\mathcal{F}_{s}^{t}$--progressively measurable stochastic
processes (p.m.s.p. for short) such that:

\begin{itemize}
\item $X^{t,x}:\Omega\times\left[  0,T\right]  \rightarrow\mathcal{\bar{D}}$
and $Y^{t,x}:\Omega\times\left[  0,T\right]  \rightarrow\mathbb{R}^{m}$ are
continuous stochastic processes,

\item $A^{t,x}:\Omega\times\left[  0,T\right]  \rightarrow\mathbb{R}_{+}$\ is
an increasing continuous stochastic process

\item $%
{\displaystyle\int_{0}^{T}}
\left(  |U_{r}^{t,x}|dr+|V_{r}^{t,x}|dA_{r}^{t,x}+|Z_{r}^{t,x}|^{2}dr\right)
<\infty$, $\mathbb{P}$--a.s. and

$(X_{s}^{t,x},A_{s}^{t,x},Y_{s}^{t,x},Z_{s}^{t,x},U_{s}^{t,x},V_{s}%
^{t,x})_{s\in\left[  0,T\right]  }$ satisfies $\mathbb{P}$--a.s. the following
decoupled forward--backward stochastic differential system:%
\begin{equation}%
\begin{array}
[c]{ll}%
\left(  a\right)  & \displaystyle X_{s}^{t,x}=x+\int_{t}^{s\vee t}%
b(r,X_{r}^{t,x})dr+\int_{t}^{s\vee t}\sigma(r,X_{r}^{t,x})dW_{r}-\int
_{t}^{s\vee t}\nabla\ell(X_{r}^{t,x})dA_{r}^{t,x},\medskip\\
\left(  b\right)  & \displaystyle A_{s}^{t,x}=\int_{t}^{s\vee t}%
\mathbb{1}_{\{X_{r}^{t,x}\in\mathrm{Bd}\left(  \mathcal{D}\right)  \}}%
dA_{r}^{t,x}\,,\medskip\\
\left(  c\right)  & \displaystyle Y_{s}^{t,x}+\int_{s\vee t}^{T}U_{r}%
^{t,x}dr+\int_{s\vee t}^{T}V_{r}^{t,x}\,dA_{r}^{t,x}=h(X_{T}^{t,x}%
)+\int_{s\vee t}^{T}\,f\left(  r,X_{r}^{t,x},Y_{r}^{t,x}\right)  dr\medskip\\
& \displaystyle\quad\quad\quad\quad\quad\quad\quad\quad\quad\quad+\int_{s\vee
t}^{T}\,g\left(  r,X_{r}^{t,x},Y_{r}^{t,x}\right)  \,dA_{r}^{t,x}%
-({M_{T}^{t,x}-M_{s}^{t,x})},\medskip\\
\left(  d\right)  & \displaystyle%
{\displaystyle\int_{s_{1}}^{s_{2}}}
\langle v-Y_{r}^{t,x},U_{r}^{t,x}\rangle dr+%
{\displaystyle\int_{s_{1}}^{s_{2}}}
\varphi\left(  Y_{r}^{t,x}\right)  dr\leq%
{\displaystyle\int_{s_{1}}^{s_{2}}}
{\varphi}\left(  v\right)  dr,\medskip\\
\left(  e\right)  & \displaystyle%
{\displaystyle\int_{s_{1}}^{s_{2}}}
\langle v-Y_{r}^{t,x},V_{r}^{t,x}\rangle dA_{r}^{t,x}+%
{\displaystyle\int_{s_{1}}^{s_{2}}}
\psi\left(  Y_{r}^{t,x}\right)  dA_{r}^{t,x}\leq%
{\displaystyle\int_{s_{1}}^{s_{2}}}
{\psi}\left(  v\right)  dA_{r}^{t,x},
\end{array}
\label{FBSDE}%
\end{equation}
for all $s,s_{1},s_{2}\in\left[  0,T\right]  ,$ such that $0\leq t\leq
s_{1}\leq s_{2}$ and any $v\in\mathbb{R}^{m}$, where%
\[
{M_{s}^{t,x}:={\int_{t}^{s\vee t}}Z_{r}^{t,x}{dW_{r}}=\int_{t}^{s\vee t}}%
\hat{Z}_{r}^{t,x}d{M_{r}^{X^{t,x}},}%
\]
with ${M_{s}^{X^{t,x}}=\int_{t}^{s\vee t}\sigma({X_{r}^{t,x})}dW_{r}}$ (the
martingale part of the reflected diffusion process) and $Z_{r}^{t,x}%
=({\sigma({X_{r}^{t,x}))}}^{\ast}\hat{Z}_{r}^{t,x}$.
\end{itemize}

Here above we consider the extension $U_{s}^{t,x}=V_{s}^{t,x}=0$ for $0\leq
s<t.$ We notice that, for $0\leq s\leq t$, $X_{s}^{t,x}=x$, $A_{s}^{t,x}=0$,
$Y_{s}^{t,x}=Y_{t}^{t,x}$, $Z_{s}^{t,x}=0$ and the last two conditions from
(\ref{FBSDE}) mean that $U_{s}^{t,x}\left(  \omega\right)  \in\partial
\varphi(Y_{s}^{t,x}\left(  \omega\right)  )$, $ds$--a.e. on $\left[
t,T\right]  $ and $V_{s}^{t,x}\left(  \omega\right)  \in\partial\psi
(Y_{s}^{t,x}\left(  \omega\right)  )$, $dA_{s}$--a.e. on $\left[  t,T\right]
$ , $\mathbb{P}$--a.s.

If we denote%
\[
K_{s}^{1,t,x}:=\int_{0}^{s}U_{r}^{t,x}dr\quad\text{and}\quad K_{s}%
^{2,t,x}:=\int_{0}^{s}V_{r}^{t,x}\,dA_{r}^{t,x},
\]
then, from conditions (\ref{FBSDE}--$d,e$) and using Proposition
\ref{equiv-subdiff}, we obtain that, as measure on $\left[  0,T\right]  ,$%
\begin{equation}
dK_{s}^{1,t,x}\in\partial\varphi(Y_{s}^{t,x})ds,~\mathbb{P}\text{--a.s.}%
\quad\text{and}\quad dK_{s}^{2,t,x}\in\partial\psi(Y_{s}^{t,x})dA_{s}%
^{t,x},\;\mathbb{P}\text{--a.s.} \label{subdiff apartenence 7}%
\end{equation}
Our aim is to prove the continuity of the deterministic function%
\[
u:\left[  0,T\right]  \times\mathcal{\bar{D}}\rightarrow\mathbb{R}^{m},\quad
u\left(  t,x\right)  :=Y_{t}^{t,x}.
\]
The assumptions required along the paper are:

\begin{itemize}
\item The functions%
\begin{equation}%
\begin{array}
[c]{l}%
b:\left[  0,\infty\right)  \times\mathbb{R}^{d}\rightarrow\mathbb{R}^{d}%
,\quad\sigma:\left[  0,\infty\right)  \times\mathbb{R}^{d}\rightarrow
\mathbb{R}^{d\times d},\medskip\\
f:\left[  0,\infty\right)  \times\overline{\mathcal{D}}\times\mathbb{R}%
^{m}\rightarrow\mathbb{R}^{m},\quad g:\left[  0,\infty\right)  \times
\mathrm{Bd}\left(  \mathcal{D}\right)  \times\mathbb{R}^{m}\rightarrow
\mathbb{R}^{m},\medskip\\
h:\overline{\mathcal{D}}\rightarrow\mathbb{R}^{m}\text{ are continuous.}%
\end{array}
\label{Assumpt. 1}%
\end{equation}

\item There exist $\beta\in\mathbb{R}$ and $L,\gamma\in\mathbb{R}_{+}$ such
that for all $t\in\left[  0,T\right]  ,\;x,\tilde{x}\in\mathbb{R}^{d}:$%
\begin{equation}
|b\left(  t,x\right)  -b\left(  t,\tilde{x}\right)  |+|\sigma\left(
t,x\right)  -\sigma\left(  t,\tilde{x}\right)  |\leq L\left\vert x-\tilde
{x}\right\vert , \label{Assumpt. 2}%
\end{equation}
and for all $t\in\left[  0,T\right]  ,\;x\in\overline{\mathcal{D}}%
,\;u\in\mathrm{Bd}\left(  \mathcal{D}\right)  ,\;y,\tilde{y}\in\mathbb{R}%
^{m}:$%
\begin{equation}%
\begin{array}
[c]{rl}%
\left(  i\right)  & \langle y-\tilde{y},f(t,x,y)-f(t,x,\tilde{y})\rangle
\leq\beta|y-\tilde{y}|^{2},\vspace*{2mm}\\
\left(  ii\right)  & \big|f(t,x,y)\big|\leq\gamma\big (1+|y|\big ),\vspace
*{2mm}\\
\left(  iii\right)  & \langle y-\tilde{y},g(t,u,y)-g(t,u,\tilde{y})\rangle
\leq\beta|y-\tilde{y}|^{2},\vspace*{2mm}\\
\left(  iv\right)  & \big|g(t,u,y)\big|\leq\gamma\big(1+|y|\big).
\end{array}
\label{Assumpt. 3}%
\end{equation}

\item The functions%
\begin{equation}%
\begin{array}
[c]{rl}%
\left(  i\right)  & \varphi,\psi:\mathbb{R}^{m}\rightarrow(-\infty
,+\infty]\text{ are proper convex l.s.c. such that\medskip}\\
\left(  ii\right)  & \varphi\left(  y\right)  \geq\varphi\left(  0\right)
=0\text{ and }\psi\left(  y\right)  \geq\psi\left(  0\right)  =0,\;\text{for
all }y\in\mathbb{R}^{m},
\end{array}
\label{Assumpt. 4}%
\end{equation}
and there exists a positive constant $M$ such that%
\begin{equation}%
\begin{array}
[c]{rl}%
\left(  i\right)  & \big|\varphi(h(x))\big|\leq M,\quad\text{for all }{x}%
\in\overline{\mathcal{D}},\medskip\\
\left(  ii\right)  & \big|\psi(h(x))\big|\leq M,\quad\text{for all }{x}%
\in\mathrm{Bd}\left(  \mathcal{D}\right)  .
\end{array}
\label{Assumpt. 5}%
\end{equation}

\end{itemize}

For the definitions of the domain of $\varphi$ and for the subdifferential
operator $\partial\varphi$ see Annex \ref{Convex functions}.

\begin{itemize}
\item The \textit{compatibility assumptions: }for all \ $\varepsilon>0$,
$t\geq0$, $x\in\mathrm{Bd}\left(  \mathcal{D}\right)  $, $\tilde{x}%
\in\overline{\mathcal{D}}$ and $y\in\mathbb{R}^{m},$%
\begin{equation}%
\begin{array}
[c]{rl}%
\left(  i\right)  & \langle\nabla\varphi_{\varepsilon}\left(  y\right)
,\nabla\psi_{\varepsilon}\left(  y\right)  \rangle\geq0,\medskip\\
\left(  ii\right)  & \langle\nabla\varphi_{\varepsilon}\left(  y\right)
,g\left(  t,x,y\right)  \rangle\leq\langle\nabla\psi_{\varepsilon}\left(
y\right)  ,g\left(  t,x,y\right)  \rangle^{+},\medskip\\
\left(  iii\right)  & \langle\nabla\psi_{\varepsilon}\left(  y\right)
,f\left(  t,\tilde{x},y\right)  \rangle\leq\langle\nabla\varphi_{\varepsilon
}\left(  y\right)  ,f\left(  t,\tilde{x},y\right)  \rangle^{+},
\end{array}
\label{Assumpt. 6}%
\end{equation}
where $a^{+}=\max\left\{  0,a\right\}  $ and $\nabla\varphi_{\varepsilon
}\left(  y\right)  $, $\nabla\psi_{\varepsilon}\left(  y\right)  $ are the
unique solutions $U$ and $V$, respectively, of the equations%
\[
{\partial\varphi}(y-\varepsilon U)\ni U\quad\text{and}\quad{\partial\psi
}(y-\varepsilon V)\ni V.
\]

\item In addition to \cite{ma-ra/10} we impose: for all $t,\tilde{t}\in\left[
0,T\right]  $, $x\in\mathrm{Bd}\left(  \mathcal{D}\right)  $, $y,\tilde{y}%
\in\mathbb{R}^{m}$%
\begin{equation}
\big|g(t,x,y)-g(\tilde{t},\tilde{x},\tilde{y})\big|\leq\beta\left(
|t-\tilde{t}|+|x-\tilde{x}|+|y-\tilde{y}|\right)  . \label{Assumpt. 7}%
\end{equation}
The last assumption is necessary in this stronger version (with respect to
\cite{ma-ra/10}) in order to obtain the convergence (\ref{Lipschitz for g}).
\end{itemize}

It follows from \cite[Theorem 3.1]{li-sz/84} that, under the assumptions
(\ref{Assumpt. 1}--\ref{Assumpt. 2}), for each $\left(  t,x\right)  \in\left[
0,T\right]  \times\overline{\mathcal{D}}$, there exists a unique pair of
continuous $\mathcal{F}_{s}^{t}$--p.m.s.p. $(X_{s}^{t,x},A_{s}^{t,x})_{s\geq
t}$, with values in $\overline{\mathcal{D}}\times\mathbb{R}_{+}$, the solution
of the reflected stochastic differential equation (\ref{FBSDE}$-a,b$) with
$A^{t,x}$ being an increasing stochastic process.

Since $\overline{\mathcal{D}}$ is a bounded set,%
\begin{equation}
\sup_{s\geq0}|X_{s}^{t,x}|\leq M. \label{mg X}%
\end{equation}

\begin{proposition}
Under the assumptions (\ref{Assumpt. 1}--\ref{Assumpt. 2}), for all $\kappa
>0$, $p\geq1$, there exists a positive constant $C$ such that for all
$s,t,t^{\prime}\in\left[  0,T\right]  $, $x,x^{\prime}\in\overline
{\mathcal{D}}:$%
\begin{equation}%
\begin{array}
[c]{l}%
\left(  a\right)  \quad\displaystyle\mathbb{E}\underset{s\in\left[
0,T\right]  }{\sup}|X_{s}^{t,x}-X_{s}^{t^{\prime},x^{\prime}}|^{p}%
+\mathbb{E}\underset{s\in\left[  0,T\right]  }{\sup}|A_{s}^{t,x}%
-A_{s}^{t^{\prime},x^{\prime}}|^{p}\leq C\big (|x-x^{\prime}|^{p}%
+|t-t^{\prime}|^{\frac{p}{2}}\big ),\medskip\\
\left(  b\right)  \quad\mathbb{E}|A_{s}^{t,x}|^{p}\leq C\big (1+|\left(  s\vee
t\right)  -t|^{p}\big ),\medskip\\
\left(  c\right)  \quad\mathbb{E}[e^{\kappa A_{T}^{t,x}}]\leq C\text{,
and}\medskip\\
\left(  d\right)  \quad\left(  t,x\right)  \mapsto\mathbb{E}{%
{\displaystyle\int_{t}^{T}}
h_{1}(s,X_{s}^{t,x})ds+\mathbb{E}%
{\displaystyle\int_{t}^{T}}
h_{2}(s,X_{s}^{t,x})dA_{s}^{t,x}}:\left[  0,T\right]  \times\overline
{D}\rightarrow\mathbb{R}\text{ is continuous,}\smallskip\\
\quad\quad\quad\quad\quad\quad\quad\quad\quad\text{for every continuous
functions }h_{1},h_{2}:\left[  0,T\right]  \times\overline{D}\rightarrow
\mathbb{R}\text{.}%
\end{array}
\label{cont X}%
\end{equation}

\end{proposition}

\begin{proof}
The proof follows the techniques from \cite{pa-ra/12}, Proposition 4.55
associated with Proposition 3.22, and Corollary 4.56. Roughly speaking, the
main idea is to use It\^{o}'s formula in order to compute $d\Big[\exp
\big[\delta\big(\ell(X_{r}^{t,x})+\ell(X_{r}^{t^{\prime},x^{\prime}%
})\big)\big]\big(X_{r}^{t,x}-X_{r}^{t^{\prime},x^{\prime}}\big)\Big]$, where
$\delta$ is a strictly positive constant (which exists due to \cite[Theorem
4.47]{pa-ra/12}) such that%
\[%
\begin{array}
[c]{l}%
\displaystyle-\langle X_{s}^{t,x}-X_{s}^{t^{\prime},x^{\prime}},\nabla
\ell(X_{r}^{t,x})dA_{r}^{t,x}-\nabla\ell(X_{r}^{t^{\prime},x^{\prime}}%
)dA_{r}^{t^{\prime},x^{\prime}}\rangle\medskip\\
\displaystyle\leq\delta|X_{s}^{t,x}-X_{s}^{t^{\prime},x^{\prime}}%
|^{2}\big(dA_{r}^{t,x}+dA_{r}^{t^{\prime},x^{\prime}}\big),\;\text{a.s.}%
\end{array}
\]
Moreover, we use the form of the process $A^{t,x}$ due to It\^{o}'s formula:%
\[
A_{s}^{t,x}=%
{\displaystyle\int\nolimits_{t}^{s\vee t}}
\mathcal{L}_{r}\ell(X_{r}^{t,x})dr+%
{\displaystyle\int\nolimits_{t}^{s\vee t}}
\left\langle \nabla\ell(X_{r}^{t,x}),\sigma(r,X_{r}^{t,x})dW_{r}\right\rangle
-\left[  \ell(X_{s}^{t,x})-\ell\left(  x\right)  \right]  .
\]
\hfill
\end{proof}

Under the assumptions (\ref{Assumpt. 1})--(\ref{Assumpt. 6}), it follows from
\cite[Theorem 9]{ma-ra/10} (with $k=1$ and $\tau$ replaced by $T$) that for
each $\left(  t,x\right)  \in\left[  0,T\right]  \times\overline{\mathcal{D}}$
there exists a unique 4--tuple $(Y^{t,x},Z^{t,x},U^{t,x},V^{t,x})\;$of
p.m.s.p. such that $Y^{t,x}$ has continuous trajectories, and for any $\mu
\geq0$ there exists a positive constant $C$ independent of $\left(
t,x\right)  $ such that
\[%
\begin{array}
[c]{l}%
\displaystyle\mathbb{E}\sup_{s\in\left[  0,T\right]  }e^{\mu A_{s}^{t,x}%
}|Y_{s}^{t,x}|^{2}+\mathbb{E}\int_{0}^{T}e^{\mu A_{s}^{t,x}}|Y_{s}^{t,x}%
|^{2}(ds+dA_{s}^{t,x})\leq C,\vspace*{2mm}\\
\displaystyle\mathbb{E}\int_{0}^{T}e^{\mu A_{s}^{t,x}}|Z_{s}^{t,x}|^{2}ds\leq
C,\medskip\\
\displaystyle\mathbb{E}\int_{0}^{T}e^{\mu A_{s}^{t,x}}|U_{s}^{t,x}%
|^{2}ds+\mathbb{E}\int_{0}^{T}e^{\mu A_{s}^{t,x}}|V_{s}^{t,x}|^{2}dA_{s}%
^{t,x}\leq C,
\end{array}
\]
and BSDE (\ref{FBSDE}$-c,d,e$) is satisfied.

We remark in addition that the functions $f$, $g\;$depend on $\omega$ only
through the process $X^{t,x}$.

\section{Main result}

We define%
\begin{equation}
u(t,x)=Y_{t}^{t,x},\quad(t,x)\in\lbrack0,T]\times\overline{\mathcal{D}},
\label{def u}%
\end{equation}
which is a determinist quantity since $Y_{t}^{t,x}$ is $\mathcal{F}_{t}%
^{t}=\sigma\left(  \mathcal{N}\right)  $--measurable.

From the Markov property we have%
\begin{equation}
u(s,X_{s}^{t,x})=Y_{s}^{t,x}. \label{prop u}%
\end{equation}

We highlight that the continuity of application $\left(  t,x\right)  \mapsto
u\left(  t,x\right)  $ does not follow anymore directly from inequality
$\left(  40\right)  $ from \cite[Proposition 13]{ma-ra/10} as it was declared
in \cite[Corollary 14--$\left(  c\right)  $]{ma-ra/10} (see the last section
for the correct statement of Proposition 13). Our article involves new
arguments. Since the function $u$ is defined through $Y^{t,x}$, the problem of
continuity of $u$ is a consequence of the continuity of the stochastic process
$(Y_{s}^{t,x})_{s\in\left[  0,T\right]  }$ with respect to the initial data
$\left(  t,x\right)  $. We will give first a generalization of
\cite[Proposition 15]{ba-ma-za/13} to our backward stochastic equation; more
precisely, we will show that $\left(  Y^{n}\right)  _{n\in\mathbb{N}}:=\left(
Y^{t_{n},x_{n}}\right)  _{n\in\mathbb{N}}$ is tight in a suitable topological
space and we will use the techniques presented in \cite[Section 4]%
{ba-ma-za/13} and \cite[Section 3]{bo-ca/04}. This approach forces us to
restrict to the case where coefficient $f$ does not depend on $Z$ (for a more
detailed explanation see the comments from \cite[Section 6, page 535]{pa/99}).

Let us consider the Skorohod space $\mathbb{D=D}\left(  \left[  0,T\right]
,\mathbb{R}^{m}\right)  $ of c\`{a}dl\`{a}g functions $y:\left[  0,T\right]
\rightarrow\mathbb{R}^{m}$ (i.e. right continuous and with left--hand limits)
endowed with \textrm{S}--topology (introduced by Jakubowski in \cite{ja/97}).
The spaces $\mathcal{C}\left(  \left[  0,T\right]  ,\mathbb{R}^{d}\right)  $
of continuous functions are equipped with the topology provided by the
supremum norm $\left\Vert u\right\Vert _{T}:=\sup_{s\in\left[  0,T\right]
}\left\vert u\left(  s\right)  \right\vert .$

For the convenience of the reader, we summarize in the Annexes the definitions
and remarks concerning \textrm{S}--topology (see also \cite{ja/97}), as well
as Helly--Bray type results, corresponding to the \textrm{S}--convergence
case.$\smallskip$

The main result of this article is the following:

\begin{theorem}
\label{cont of Y and u}Under assumptions (\ref{Assumpt. 1})--(\ref{Assumpt. 7}%
), function%
\[
\left(  t,x\right)  \mapsto u\left(  t,x\right)  =Y_{t}^{t,x}:[0,T]\times
\overline{\mathcal{D}}\rightarrow\mathbb{R}^{m}%
\]
is continuous.
\end{theorem}

\begin{remark}
Using the continuity of $u$ it was proved in \cite{ma-ra/10} that in the case
$m=1$ this function is the unique viscosity solution of the parabolic
variational inequality with mixed nonlinear multivalued Neumann--Dirichlet
boundary condition (\ref{multivalued PDE}).
\end{remark}

\begin{proof}
[Proof of Theorem \ref{cont of Y and u}]Let $\left(  t_{n},x_{n}\right)
\rightarrow\left(  t,x\right)  $, as $n\rightarrow\mathbb{\infty}$. To prove
that $u\left(  t_{n},x_{n}\right)  \rightarrow u\left(  t,x\right)  $ is
equivalent with proving that any subsequence has a further subsequence which
converges to $u\left(  t,x\right)  $. Let $\left(  t_{n_{k}},x_{n_{k}}\right)
$ be an arbitrary subsequence still denoted in the sequel by $\left(
t_{n},x_{n}\right)  .$\medskip

Using the definitions%
\[
f_{n}\left(  r,x,y\right)  :={\mathbb{1}_{[t_{n},T]}}\left(  r\right)
f\left(  r,x,y\right)  \quad\text{and}\quad g_{n}\left(  r,x,y\right)
:=g\left(  r\vee t_{n},x,y\right)
\]
it is clear that the processes%
\[%
\begin{array}
[c]{l}%
X^{n}:=X^{t_{n},x_{n}}\,,~A^{n}:=A^{t_{n},x_{n}}\quad\text{and}\medskip\\
Y^{n}:=Y^{t_{n},x_{n}}\,,~Z^{n}:=Z^{t_{n},x_{n}}\,,~M^{n}:=M^{t_{n},x_{n}%
}\,,~U^{n}:=U^{t_{n},x_{n}}\,,~V^{n}:=V^{t_{n},x_{n}}%
\end{array}
\]
satisfy equation (\ref{FBSDE}$-a,b$)%
\begin{equation}
\left\{
\begin{array}
[c]{l}%
\displaystyle X_{s}^{n}=x_{n}+\int_{t_{n}}^{s\vee t_{n}}b(r,X_{r}^{n}%
)dr+\int_{t_{n}}^{s\vee t_{n}}\sigma(r,X_{r}^{n})dW_{r}-\int_{t_{n}}^{s\vee
t_{n}}\nabla\ell(X_{r}^{n})dA_{r}^{n},\;\medskip\\
s\longmapsto A_{s}^{n}\text{\ \ is increasing,}\medskip\\
\displaystyle A_{s}^{n}=\int_{t_{n}}^{s\vee t_{n}}\mathbf{1}_{\{X_{r}^{n}%
\in\mathrm{Bd}\left(  \mathcal{D}\right)  \}}dA_{r}^{n}\,,\text{ for all }%
s\in\left[  0,T\right]
\end{array}
\right.  \label{def (X,A) 2}%
\end{equation}
and the backward equation%
\begin{equation}%
\begin{array}
[c]{r}%
Y_{s}^{n}+\displaystyle\int_{s}^{T}U_{r}^{n}dr+\int_{s}^{T}V_{r}^{n}%
\,dA_{r}^{n}{=h(X}_{T}^{n}){+\int_{s}^{T}}f_{n}{(r,{X_{r}^{n},Y{_{r}^{n}})}%
dr}\medskip\\
\displaystyle{+\int_{s}^{T}}g_{n}{(r,{X_{r}^{n},Y{_{r}^{n}})dA_{r}^{n}}}-%
{\displaystyle\int_{s}^{T}}
{\left\langle Z_{r}^{n},{dW_{r}}\right\rangle }\,,\;s\in\lbrack0,T]
\end{array}
\label{generalized BSDE 2}%
\end{equation}
such that (\ref{FBSDE}$-d,e$) is satisfied (we recall that ${{X_{r}^{n}}}%
={{x}}_{n}$, ${{A_{r}^{n}}}=U_{r}^{n}=V_{r}^{n}\,=0$ and $Z_{r}^{n}=0$ if
$r\in\left[  0,t_{n}\right]  ;$ consequently $Y_{s}^{n}=Y_{t_{n}}^{n}$ if
$s\in\left[  0,t_{n}\right]  $).

The first part of the proof (the arguments for $\mathrm{S}$--tightness) is
adapted from \cite[Proposition 15]{ba-ma-za/13} and \cite[Theorem
3.1]{bo-ca/04} to the case of backward stochastic variational inequalities.

Since we have the conclusion of the Existence Theorem 9 from \cite{ma-ra/10},
we easily see that:%
\[
\sup_{n\in\mathbb{N}}{\Big[}\mathbb{E}{\big(}\sup_{s\in\left[  0,T\right]
}|Y_{s}^{n}|^{2}{\big){+}\mathbb{E}{\int_{0}^{T}}|Y_{r}^{n}|^{2}{dA_{r}^{n}}%
+}\mathbb{E}{\int_{0}^{T}}|Z_{r}^{n}|^{2}dr+{\mathbb{E}{\int_{0}^{T}}%
|U_{r}^{n}|^{2}{dr+}\mathbb{E}{\int_{0}^{T}}|V_{r}^{n}|^{2}{dA_{r}^{n}}%
\Big]}<\mathbb{\infty}.
\]
In fact, it can be proved as in \cite[Theorem 9]{ma-ra/10} (also see
Proposition 5.46 from \cite{pa-ra/12}) that%
\begin{equation}
\sup_{n\in\mathbb{N}}\mathbb{E}{\big(}\sup_{s\in\left[  0,T\right]  }%
|Y_{s}^{n}|^{p}{\big)}<\infty,\;\text{for all }p>0. \label{mg Y 2}%
\end{equation}
The $\mathrm{S}$--tightness will be obtained by using the sufficient condition
given, e.g., in \cite[Appendix A]{le/02} (also see Theorem
\ref{Theorem 2_Annexes} in the Annexes).

Let%
\[
{M{_{s}^{n}}}:={\int_{t_{n}}^{s\vee t_{n}}Z_{r}^{n}{dW_{r}}}={\int_{t_{n}%
}^{s\vee t_{n}}}\hat{Z}_{r}^{t_{n},x_{n}}d{M_{r}^{X^{n}}},\quad{M_{s}^{X^{n}}%
}=\int_{t_{n}}^{s\vee t_{n}}\sigma(r,X_{r}^{n})dW_{r}~.
\]
and%
\[
K_{s}^{1,n}:=\int_{0}^{s}U_{r}^{n}dr\,,\quad K_{s}^{2,n}:=\int_{0}^{s}%
V_{r}^{n}dA_{r}^{n}\,,\text{ for all }s\in\left[  0,T\right]  .
\]
We clearly have ${M{_{s}^{n}}}=K_{s}^{1,n}=K_{s}^{2,n}=0$ for $s\in\left[
0,t\right]  $ and
\[
\sup_{n\in\mathbb{N}}\mathbb{E}{\Big[\left\updownarrow K^{1,n}%
\right\updownarrow _{T}^{2}+\left\updownarrow K^{2,n}\right\updownarrow
_{T}^{2}+\sup_{s\in\left[  0,T\right]  }}\left\vert {M{_{s}^{n}}}\right\vert
^{2}{\Big]}<\mathbb{\infty},
\]
where ${\left\updownarrow \cdot\right\updownarrow _{T}}$ denotes the total
variation on $\left[  0,T\right]  $ (see Section
\ref{Bounded variation functions} in the Annexes) and consequently
\begin{equation}
\sup_{n\in\mathbb{N}^{\ast}}\left[  \mathbb{E}\sup_{s\in\lbrack0,T]}\left\vert
Y_{s}^{n}\right\vert +\mathbb{E}\sup_{s\in\lbrack0,T]}|K_{s}^{1,n}%
|+\mathbb{E}\sup_{s\in\lbrack0,T]}|K_{s}^{2,n}|+\mathbb{E}{\sup_{s\in\left[
0,T\right]  }}\left\vert {M{_{s}^{n}}}\right\vert \right]  =C_{1}<\infty.
\label{bound_sup}%
\end{equation}
Moreover we easily deduce using the assumptions (\ref{Assumpt. 3}--$ii,iv$)%
\[
\sup_{n\in\mathbb{N}^{\ast}}\mathbb{E}\left[  {\int_{0}^{T}}\left\vert
f_{n}{(r,{X_{r}^{n},Y{_{r}^{n}})}}\right\vert {dr+\int_{0}^{T}}\left\vert
g_{n}{(r,{X_{r}^{n},Y{_{r}^{n}})}}\right\vert {{dA_{r}^{n}}}\right]
=C_{2}<\infty.
\]

Equation (\ref{generalized BSDE 2}) becomes%
\begin{equation}%
\begin{array}
[c]{r}%
\displaystyle Y_{s}^{n}+(K_{T}^{1,n}-K_{s}^{1,n})+(K_{T}^{2,n}-K_{s}%
^{2,n}){=h(X}_{T}^{n}){+\int_{s}^{T}}f_{n}{(r,{X_{r}^{n},Y{_{r}^{n}})}%
dr}\medskip\\
\displaystyle{+}\int_{s}^{T}g_{n}{(r,{X_{r}^{n},Y{_{r}^{n}})dA_{r}^{n}}%
}-(M_{T}^{n}-M_{s}^{n})\,,\;s\in\lbrack0,T]
\end{array}
\label{generalized BSDE 3}%
\end{equation}
where, as measure on $\left[  t_{n},T\right]  ,$ $\mathbb{P}$--a.s.%
\begin{equation}
dK_{s}^{1,n}=U_{s}^{n}ds\in\partial\varphi(Y_{s}^{n})ds\quad\text{and}\quad
dK_{s}^{2,n}=V_{s}^{n}dA_{s}^{n}\in\partial\psi(Y_{s}^{n})dA_{s}^{n}.
\label{elements of subdiff 5}%
\end{equation}
We recall now the definition of the conditional variation (see also
(\ref{def cond var}) in the Annexes):%
\[
\mathrm{CV}_{T}(L):=\sup_{\pi}{\sum_{i=0}^{N-1}{\mathbb{E}}}%
\Big[\big|\mathbb{E}^{{{\mathcal{F}}_{t_{i}}}}{[L_{t_{i+1}}-L_{t_{i}}%
]{\big|}\Big],}%
\]
with the supremum taken over all partitions $\pi:0=t_{0}<t_{1}<\cdots
<t_{N}=T,$ where $L$ is a c\`{a}dl\`{a}g stochastic process such that
$\mathbb{E}\left\vert L_{t}\right\vert <\infty$ for all $t\in\left[
0,T\right]  .$

It can be proved that there exists a positive constant $C$ independent of
$n\in\mathbb{N}^{\ast}$ such that%
\begin{equation}%
\begin{array}
[c]{l}%
\displaystyle\mathrm{CV}_{T}\left(  Y^{n}\right)  +\mathrm{CV}_{T}\left(
K^{1,n}\right)  +\mathrm{CV}_{T}\left(  K^{2,n}\right)  +{\mathrm{CV}%
_{T}\left(  M^{n}\right)  }\medskip\\
\displaystyle\leq2\mathbb{E}\int_{0}^{T}|U_{r}^{n}|dr+2\mathbb{E}\int_{0}%
^{T}|V_{r}^{n}|dA_{r}^{n}+\int_{0}^{T}\left\vert f_{n}{(r,{X_{r}^{n}%
,Y{_{r}^{n}})}}\right\vert {dr+}\int_{0}^{T}\left\vert g_{n}{(r,{X_{r}%
^{n},Y{_{r}^{n}})}}\right\vert {{dA_{r}^{n}}}\medskip\\
\displaystyle\leq C.
\end{array}
\label{bound_cv}%
\end{equation}
Using, for instance, \cite[Proposition 1.47]{pa-ra/12} (and the calculus from
the proof of \cite[Theorem 4.53]{pa-ra/12}) it can be show that the sequence
$\left(  X^{n},W,A^{n}\right)  $ is tight in $\left(  {\mathcal{C}%
}([0,T],{\mathbb{R}}^{d})\right)  ^{2}\times{\mathcal{C}}([0,T],{\mathbb{R}%
}).$

Now, by (\ref{bound_sup}) and (\ref{bound_cv}), the criterion presented in
Theorem \ref{Theorem 2_Annexes} (See the Annex) ensures tightness with respect
to the \textrm{S}--topology of the sequence $(Y^{n},M^{n},K^{1,n},K^{2,n})$
and therefore%
\[
\Gamma^{n}:=(X^{n},W,A^{n},Y^{n},M^{n},K^{1,n},K^{2,n})
\]
is tight in $\mathbb{X}:=\left(  {\mathcal{C}}([0,T],{\mathbb{R}}^{d})\right)
^{2}\times{\mathcal{C}}([0,T],{\mathbb{R}})\times\mathbb{D}^{4}$.

From Theorem \ref{Theorem 2_Annexes}, it follows that there exists a
subsequence (still denoted by $n$) and the following processes, defined on the
same probability space $\left(  \bar{\Omega},\mathcal{\bar{F}},\mathbb{\bar
{P}}\right)  :=\left(  \left[  0,1\right]  ,\mathcal{B}_{\left[  0,1\right]
},d\lambda\right)  $ (where $d\lambda$ denotes the Lebesgue measure),%
\[
\bar{\Gamma}^{n}:=(\bar{X}^{n},\bar{W}^{n},\bar{A}^{n},\bar{Y}^{n},\bar{M}%
^{n},\bar{K}^{1,n},\bar{K}^{2,n}):\bar{\Omega}\rightarrow\mathbb{X}%
\]
and%
\[
\bar{\Gamma}:=(\bar{X},\bar{W},\bar{A},\bar{Y},\bar{M},\bar{K}^{1},\bar{K}%
^{2}):\bar{\Omega}\rightarrow\mathbb{X}%
\]
such that%
\[
\Gamma^{n}\sim\bar{\Gamma}^{n}%
\]
and%
\[
\text{for all }\omega\in\bar{\Omega},\;\bar{\Gamma}^{n}\left(  \omega\right)
\xrightarrow[\mathrm{U}^{3}\times \mathrm{S}^{4}]{\;\;\;\;\;\;\;\;\;\;\;\;\;\;\;\;\;\;\;}\bar
{\Gamma}\left(  \omega\right)  \text{, as }n\rightarrow\infty,
\]
where $\sim$ denotes the equality in law of both stochastic processes, the
\textrm{U}$^{3}$--convergence means the uniform convergence of $(\bar{X}%
^{n},\bar{W}^{n},\bar{A}^{n})$ on the space of continuous
function$\mathcal{\ }$and the \textrm{S}$^{4}$--convergence of $(\bar{Y}%
^{n},\bar{M}^{n},\bar{K}^{1,n},\bar{K}^{2,n})$ is defined by Definition
\ref{Definition 2_Annexes} in the Annexes.

Morover, by a.s. Skorohod's representation theorem on $\left(  \left[
0,1\right]  ,\mathcal{B}_{\left[  0,1\right]  },d\lambda\right)  $ (see
\cite{ja/97}), we obtain%
\[%
\begin{array}
[c]{r}%
(\bar{X}^{n}\left(  \omega\right)  ,\bar{W}^{n}\left(  \omega\right)  ,\bar
{A}^{n}\left(  \omega\right)  )\longrightarrow(\bar{X}\left(  \omega\right)
,\bar{W}\left(  \omega\right)  ,\bar{A}\left(  \omega\right)  ),\;\text{as
}n\rightarrow\infty,\medskip\\
\text{in }\left(  {\mathcal{C}}([0,T],{\mathbb{R}}^{d})\right)  ^{2}%
\times{\mathcal{C}}([0,T],{\mathbb{R}})\text{, for all }\omega\in\bar{\Omega}%
\end{array}
\]
and there exist a countable set $Q\subset\lbrack0,T)$ such that, for any
$t\in\left[  0,T\right]  \setminus Q,$%
\begin{equation}
(\bar{Y}_{t}^{n},\bar{M}_{t}^{n},\bar{K}_{t}^{1,n},\bar{K}_{t}^{2,n}%
)\xrightarrow[\lambda-\mathrm{a.s.}]{\;\;\;\;\;\;\;\;\;}(\bar{Y}_{t},\bar
{M}_{t},\bar{K}_{t}^{1},\bar{K}_{t}^{2}),\quad\text{as }n\rightarrow
\infty.\label{technical ineq 10}%
\end{equation}
We remark that, in particular, we have $u\left(  t_{n},x_{n}\right)
=Y_{t_{n}}^{n}=\bar{Y}_{t_{n}}^{n}\,$, since are deterministic.

We are now able to pass to the limit in (\ref{def (X,A) 2}): since%
\[
(X^{n},W,A^{n})\sim(\bar{X}^{n},\bar{W}^{n},\bar{A}^{n}),
\]
we deduce, using standard arguments (see, e.g., \cite{bu-ra/03} or Proposition
2.15 and the proof of Theorem 3.54 from \cite{pa-ra/12}), that $(\bar{X}%
^{n},\bar{W}^{n},\bar{A}^{n})$ satisfies equation (\ref{def (X,A) 2}), the
limit process $(\bar{X},\bar{W},\bar{A})$ satisfies equation (\ref{FBSDE}%
$-a,b$) and $\bar{W}^{n}$ (respectively, $\bar{W}$), is a Brownian motion with
respect to the filtration $\big(\mathcal{F}_{s}^{\bar{X}^{n},\bar{W}^{n}%
}\big)$ (respectively, $\big(\mathcal{F}_{s}^{\bar{X},\bar{W}}\big)$).

Hence we have%
\[
(\bar{X},\bar{A})=(\bar{X}^{t,x},\bar{A}^{t,x}),
\]
where $({\bar{X}_{r}^{t,x}{,}}\bar{A}_{r}^{t,x})_{r\in\left[  0,T\right]  }$
is the solution of equation (\ref{FBSDE}$-a,b$), considered on the probability
space $\left(  \bar{\Omega},\mathcal{\bar{F}},\mathbb{\bar{P}}\right)  $ with
driving Brownian motion $\bar{W}.$

We highlight that the continuous process $\bar{A}$ is $(\mathcal{F}^{\bar
{X},\bar{W}})$--adapted. In fact, $(\bar{X},\bar{A})$ is even $(\mathcal{F}%
_{s}^{\bar{W}})$--adapted\footnote{We are thankful to the reviewer for this
suggestion.}. Indeed, under the assumptions on the coefficients, equation
(\ref{FBSDE}$-a,b$) has the property of pathwise uniqueness and, thus, due to
Yamada-Watanabe's Theorem any weak solution is a strong one, and so is
$(\bar{X},\bar{A})$ a strong solution with respect to the driving Brownian
motion $\bar{W}$. But this has as consequence that $(\bar{X},\bar{A})$ is a
non anticipative, measurable functional of $\bar{W}$ and, hence,
$(\mathcal{F}_{s}^{\bar{W}})$--adapted.\medskip

Concerning equation (\ref{generalized BSDE 3}), we state first the following
technical result (which proof is left to the reader):

\begin{lemma}
\label{technical lemma 1}Let%
\[
(X^{n},A^{n},Y^{n},M^{n},K^{1,n},K^{2,n})\sim(\bar{X}^{n},\bar{A}^{n},\bar
{Y}^{n},\bar{M}^{n},\bar{K}^{1,n},\bar{K}^{2,n}),
\]
$G,H:\left[  0,T\right]  \times\mathbb{R}^{d}\times\mathbb{R}\rightarrow
\mathbb{R}$ be two continuous functions and $\phi:{\mathcal{C}}%
([0,T],{\mathbb{R}}^{d})\times{\mathcal{C}}([0,T],{\mathbb{R}})\times
\mathbb{D}^{4}\rightarrow\mathbb{R}$ be a measurable function.

If%
\[
\phi(X^{n},A^{n},Y^{n},M^{n},K^{1,n},K^{2,n})=\int_{s_{1}}^{s_{2}}G\left(
r,X_{r}^{n},Y_{r}^{n}\right)  dA_{r}^{n}+\int_{s_{1}}^{s_{2}}H\left(
r,X_{r}^{n},Y_{r}^{n}\right)  dr\,,
\]
then%
\[
\phi(\bar{X}^{n},\bar{A}^{n},\bar{Y}^{n},\bar{M}^{n},\bar{K}^{1,n},\bar
{K}^{2,n})=\int_{s_{1}}^{s_{2}}G\left(  r,\bar{X}_{r}^{n},\bar{Y}_{r}%
^{n}\right)  d\bar{A}_{r}^{n}+\int_{s_{1}}^{s_{2}}H\left(  r,\bar{X}_{r}%
^{n},\bar{Y}_{r}^{n}\right)  dr\,.
\]
$\smallskip$
\end{lemma}

Hence we deduce that $\bar{X}^{n},\bar{A}^{n},\bar{Y}^{n},\bar{M}^{n},\bar
{K}^{1,n}$ and $\bar{K}^{2,n}$ are continuous and%
\begin{equation}%
\begin{array}
[c]{r}%
\displaystyle\bar{Y}_{s}^{n}+(\bar{K}_{T}^{1,n}-\bar{K}_{s}^{1,n})+(\bar
{K}_{T}^{2,n}-\bar{K}_{s}^{2,n}){=h(\bar{X}}_{T}^{n}){+\int_{s}^{T}}%
f_{n}{(r,{\bar{X}_{r}^{n},\bar{Y}{_{r}^{n}})}dr}\medskip\\
\displaystyle{+}\int_{s}^{T}g_{n}{(r,{\bar{X}_{r}^{n},\bar{Y}{_{r}^{n}}%
)d\bar{A}_{r}^{n}}}-(\bar{M}_{T}^{n}-\bar{M}_{s}^{n})\,,\;s\in\lbrack0,T].
\end{array}
\label{generalized BSDE 4}%
\end{equation}
In addition, we have%
\[
{{\bar{X}_{s}^{n}=x}}_{n}~,\quad\bar{Y}_{s}^{n}=\bar{Y}_{t_{n}}^{n}%
\,,\quad{{\bar{A}_{s}^{n}=}}\bar{K}_{s}^{1,n}=\bar{K}_{s}^{2,n}=0\,,\quad
\bar{M}_{s}^{n}=0\,,\quad\text{for all }s\in\left[  0,t_{n}\right]  .
\]
In order to show that (\ref{elements of subdiff 5}) holds true for $(\bar
{Y}^{n},\bar{K}^{1,n},\bar{K}^{2,n})$, the next Lemma can be proved (see,
e.g., the proof of \cite[Proposition 1.19]{pa-ra/12}).

\begin{lemma}
\label{technical lemma 2}Let $t\in\left[  0,T\right]  $ be fixed. Let
$\varphi:\mathbb{R}^{d}\rightarrow(-\infty,+\infty]$ be a proper convex l.s.c.
function such that $\varphi\left(  y\right)  \geq\varphi\left(  0\right)  =0,$
for all $y\in\mathbb{R}^{d}.$ Let $L,\bar{L}$ be $\mathbb{R}$-valued
continuous stochastic processes (with $L_{0}=0$ and $L$ is a non-decreasing
stochastic process) and $S,N,\bar{S},\bar{N}$ be $\mathbb{R}^{d}$-valued
continuous stochastic processes on $\left[  0,T\right]  $ defined on the
probability spaces $\left(  \Omega,\mathcal{F},\mathbb{P}\right)  $ and,
respectively, $\left(  \bar{\Omega},\mathcal{\bar{F}},\mathbb{\bar{P}}\right)
.$ If $\left\updownarrow N\right\updownarrow _{T}<\infty$, $\mathbb{P}$--a.s.,%
\[
(L,S,N)\sim(\bar{L},\bar{S},\bar{N})
\]
and $\mathbb{P}$--a.s.%
\[
\int_{s_{1}}^{s_{2}}\varphi(S_{r})dL_{r}\leq\int_{s_{1}}^{s_{2}}\left\langle
S_{r}-v,dN_{r}\right\rangle +\int_{s_{1}}^{s_{2}}\varphi(v)dL_{r},\;\text{for
all }v\in{\mathbb{R}},0\leq t\leq s_{1}\leq s_{2}\leq T,
\]
then $\mathbb{\bar{P}}$--a.s.%
\[
\bar{L}\text{ is a non-decreasing stochastic process,}\quad\left\updownarrow
\bar{N}\right\updownarrow _{T}<\infty
\]
and $\mathbb{\bar{P}}$--a.s.%
\[
\int_{s_{1}}^{s_{2}}\varphi(\bar{S}_{r})d\bar{L}_{r}\leq\int_{s_{1}}^{s_{2}%
}\left\langle \bar{S}_{r}-v,d\bar{N}_{r}\right\rangle +\int_{s_{1}}^{s_{2}%
}\varphi(v)d\bar{L}_{r},\;\text{for all }v\in{\mathbb{R}},0\leq t\leq
s_{1}\leq s_{2}\leq T.
\]

\end{lemma}

Hence, we remark (see (\ref{subdiff apartenence 7}) and Proposition
\ref{equiv-subdiff}) that we have as measure on $\left[  t_{n},T\right]  ,$%
\begin{equation}
d\bar{K}_{s}^{1,n}\in\partial\varphi(\bar{Y}_{s}^{n})ds,~\mathbb{\bar{P}%
}\text{--a.s.}\quad\text{and}\quad d\bar{K}_{s}^{2,n}\in\partial\psi(\bar
{Y}_{s}^{n})d\bar{A}_{s}^{n},~\mathbb{\bar{P}}\text{--a.s..}
\label{elements of subdiff 6}%
\end{equation}
Now we can pass to the limit in (\ref{generalized BSDE 4}).

First, applying Remark \ref{Remark 2_Annexes}, there exists a countable set
$Q\subset(0,T)$ such that for all $r\in\left[  0,T\right]  \setminus\left(
Q\cup\left\{  t\right\}  \right)  $,%
\[%
\begin{array}
[c]{l}%
\big|f_{n}{(r,{\bar{X}_{r}^{n},\bar{Y}{_{r}^{n}})-}\mathbb{1}}_{\left[
t,T\right]  }\left(  r\right)  f{(r,{\bar{X}_{r}^{t,x},\bar{Y}}}_{r}{{)}%
}\big|\medskip\\
\leq\big|f{(r,{\bar{X}_{r}^{n},\bar{Y}{_{r}^{n}})-}}f{(r,{\bar{X}_{r}%
^{t,x},\bar{Y}}}_{r}{{)}}\big|+\big|{{{\mathbb{1}}_{\left[  t_{n},T\right]
}\left(  r\right)  -{\mathbb{1}}_{\left[  t,T\right]  }\left(  r\right)  }%
}\big|\big|f{(r,{\bar{X}_{r}^{t,x},\bar{Y}}}_{r}{{)}}\big|\rightarrow
0,\quad\text{as }n\rightarrow\infty,
\end{array}
\]
and%
\[
\mathbb{\bar{E}}\big|f_{n}{(r,{\bar{X}_{r}^{n},\bar{Y}{_{r}^{n}})-}\mathbb{1}%
}_{\left[  t,T\right]  }\left(  r\right)  f{(r,{\bar{X}_{r}^{t,x},\bar{Y}}%
}_{r}{{)}}\big|^{2}\leq C+C\,\mathbb{E}\sup_{s\in\left[  0,T\right]
}\left\vert {{\bar{Y}{_{s}^{n}}}}\right\vert ^{2}+C\,\mathbb{E}\sup
_{s\in\left[  0,T\right]  }\left\vert {{\bar{Y}{_{s}}}}\right\vert ^{2}\leq
C_{1}<\infty.
\]
Then, by the uniform integrability property on $\bar{\Omega},$
\[
\mathbb{\bar{E}}\big|f_{n}{(r,{\bar{X}_{r}^{n},\bar{Y}{_{r}^{n}})-}\mathbb{1}%
}_{\left[  t,T\right]  }\left(  r\right)  f{(r,{\bar{X}_{r}^{t,x},\bar{Y}}%
}_{r}{{)}}\big|\rightarrow0,\quad\text{as }n\rightarrow\infty,\quad\text{for
all }r\in\left[  0,T\right]  \setminus\left(  Q\cup\left\{  t\right\}
\right)  .
\]
Since%
\[%
\begin{array}
[c]{l}%
\displaystyle\int_{0}^{T}\left(  \mathbb{\bar{E}}\big|f_{n}{(r,{\bar{X}%
_{r}^{n},\bar{Y}{_{r}^{n}})-}\mathbb{1}}_{\left[  t,T\right]  }\left(
r\right)  f{(r,{\bar{X}_{r}^{t,x},\bar{Y}}}_{r}{{)}}\big|\right)
^{2}dr\medskip\\
\displaystyle\leq\int_{0}^{T}\mathbb{\bar{E}}\big|f_{n}{(r,{\bar{X}_{r}%
^{n},\bar{Y}{_{r}^{n}})-}\mathbb{1}}_{\left[  t,T\right]  }\left(  r\right)
f{(r,{\bar{X}_{r}^{t,x},\bar{Y}}}_{r}{{)}}\big|^{2}dr\leq C_{1}T,
\end{array}
\]
by the uniform integrability property on $\left[  0,T\right]  ,$ we get%
\[%
\begin{array}
[c]{l}%
\displaystyle\mathbb{\bar{E}}\int_{0}^{T}\big|f_{n}{(r,{\bar{X}_{r}^{n}%
,\bar{Y}{_{r}^{n}})-}\mathbb{1}}_{\left[  t,T\right]  }\left(  r\right)
f{(r,{\bar{X}_{r}^{t,x},\bar{Y}}}_{r}{{)}}\big|dr\medskip\\
\displaystyle=\int_{0}^{T}\mathbb{\bar{E}}\big|f_{n}{(r,{\bar{X}_{r}^{n}%
,\bar{Y}{_{r}^{n}})-}\mathbb{1}}_{\left[  t,T\right]  }\left(  r\right)
f{(r,{\bar{X}_{r}^{t,x},\bar{Y}}}_{r}{{)}}\big|dr\rightarrow0.
\end{array}
\]
In particular from $L^{1}$--convergence it follows that on a subsequence
(indexed also by $n$),
\[
\int_{0}^{T}\big|f_{n}{(r,{\bar{X}_{r}^{n},\bar{Y}{_{r}^{n}})-}\mathbb{1}%
}_{\left[  t,T\right]  }\left(  r\right)  f{(r,{\bar{X}_{r}^{t,x},\bar{Y}}%
}_{r}{{)}}\big|dr\rightarrow0,\quad\text{as }n\rightarrow\infty,\;\mathbb{\bar
{P}}\text{--a.s.}%
\]
Hence, for all $s\in\left[  0,T\right]  ,$%
\begin{equation}
\lim_{n\rightarrow\infty}{\int_{s}^{T}}f_{n}{(r,{\bar{X}_{r}^{n},\bar{Y}%
{_{r}^{n}})}dr=\int_{s}^{T}}f{(r,{\bar{X}_{r}^{t,x},\bar{Y}}}_{r}{{)}%
dr,}\;\mathbb{\bar{P}}-a.s.. \label{technical limit 1}%
\end{equation}
For the Riemann--Stieltjes integral we will apply part $\mathrm{(III)}$ of
Theorem \ref{Theorem 1_Annexes} as well as Proposition
\ref{Proposition 4_Annexes} in the Annexes. Hence, from
(\ref{technical ineq 10}), we infer that there exists a countable set
$Q\subset\left(  0,T\right)  $ such that for all $s\in\lbrack0,T]\setminus Q$%
\begin{equation}
\lim_{n\rightarrow\infty}\int_{s}^{T}g_{n}{(r,{\bar{X}_{r}^{n},\bar{Y}%
{_{r}^{n}})d\bar{A}_{r}^{n}}=}\int_{s}^{T}g{(r,{\bar{X}_{r}^{t,x},\bar{Y}%
{_{r}})d\bar{A}_{r}^{t,x}}}\,. \label{technical limit 2}%
\end{equation}
It follows that%
\begin{equation}%
\begin{array}
[c]{r}%
\displaystyle\bar{Y}_{s}+(\bar{K}_{T}^{1}-\bar{K}_{s}^{1})+(\bar{K}_{T}%
^{2}-\bar{K}_{s}^{2}){=h(\bar{X}}_{T}^{t,x}){+\int_{s}^{T}\mathbb{1}_{[t,T]}%
}\left(  r\right)  f{(r,{\bar{X}_{r}^{t,x},\bar{Y}{_{r}})}dr}\medskip\\
\displaystyle{+}\int_{s}^{T}g{(r,{\bar{X}_{r}^{t,x},\bar{Y}{_{r}})d\bar{A}%
_{r}^{t,x}}}-(\bar{M}_{T}-\bar{M}_{s})\,,\quad s\in\lbrack0,T]\setminus Q.
\end{array}
\label{generalized BSDE 5}%
\end{equation}
Since the processes ${\bar{Y}}$, ${\bar{M}}$, $\bar{K}^{1}$ and $\bar{K}^{2}$
are c\`{a}dl\`{a}g, the above equality holds for all $s\in\lbrack0,T].$

In addition, we have%
\begin{equation}
\bar{Y}_{s}=\bar{Y}_{t}\,,\quad\bar{K}_{s}^{1}=\bar{K}_{s}^{2}=0\,,\quad
\bar{M}_{s}=0\,,\quad\text{for all }s\in\lbrack0,t].
\label{extension of sol 1}%
\end{equation}
From the above equation, it is immediate that ${\bar{M}}$ is ${\mathcal{F}%
}_{s}^{\bar{X},\bar{W},{\bar{Y}},{\bar{M},}\bar{K}^{1},\bar{K}^{2}}%
\equiv{\mathcal{F}}_{s}^{\bar{W},{\bar{Y}},{\bar{M},}\bar{K}^{1},\bar{K}^{2}}%
$--adapted and it can be shown (see e.g. the proof \cite[Theorem 3.1 (step
3)]{bo-ca/04}) that both ${M^{\bar{X}}}$ and ${\bar{M}}$ are martingales with
respect to the same filtration $\mathcal{\bar{F}}_{s}:={\mathcal{F}}_{s}%
^{\bar{W},{\bar{Y}},{\bar{M},}\bar{K}^{1},\bar{K}^{2}},$ $s\in\lbrack0,T],$
(and this is the reason to work not with the filtration generated by the
Brownian motion).

We mention here that we can deduce, using Proposition
\ref{Proposition 2_Annexes} from the Annexes, that the processes $\bar{K}^{1}$
and $\bar{K}^{2}$ are with bounded variation, since $\mathbb{\bar{E}%
}\left\updownarrow \bar{K}^{i}\right\updownarrow _{T}\leq\liminf
_{n\rightarrow+\infty}\mathbb{\bar{E}}\left\updownarrow \bar{K}^{i,n}%
\right\updownarrow _{T}=\liminf_{n\rightarrow+\infty}\mathbb{E}%
\left\updownarrow K^{i,n}\right\updownarrow _{T}\,$, with $i=\overline{1,2}%
.$\medskip

On our new probability space $\left(  \bar{\Omega},\mathcal{\bar{F}%
},\mathbb{\bar{P}}\right)  $ we consider the solution $(\bar{Y}^{t,x},{\bar
{Z}{^{t,x},\bar{U}}}^{t,x},\bar{V}^{t,x})$ of BSDE (\ref{FBSDE}$-c,d,e$):%
\begin{equation}%
\begin{array}
[c]{r}%
\displaystyle\bar{Y}_{s}^{t,x}+(\bar{K}_{T}^{1,t,x}-\bar{K}_{s}^{1,t,x}%
)+(\bar{K}_{T}^{2,t,x}-\bar{K}_{s}^{2,t,x}){=h(}\bar{X}_{T}^{t,x}){+\int
_{s}^{T}\mathbb{1}_{[t,T]}}\left(  r\right)  f{(r,\bar{X}_{r}^{t,x}{,\bar
{Y}_{r}^{t,x})}dr}\medskip\\
\displaystyle{+\int_{s}^{T}}g{(r,\bar{X}_{r}^{t,x}{,\bar{Y}_{r}^{t,x})d}}%
\bar{A}_{r}^{t,x}-({\bar{M}_{T}^{t,x}-\bar{M}_{s}^{t,x}})\,,\;s\in\lbrack0,T],
\end{array}
\label{generalized BSDE 6}%
\end{equation}
with%
\[
\bar{K}_{s}^{1,t,x}=\int_{0}^{s}\bar{U}_{r}^{t,x}dr\,,\quad\bar{K}_{s}%
^{2,t,x}=\int_{0}^{s}\bar{V}_{r}^{t,x}d\bar{A}_{r}^{t,x}\,,\quad{\bar{M}%
_{s}^{t,x}=\int_{0}^{s}{\bar{Z}{_{r}^{t,x}d\bar{W}}}_{r}\,{,}}%
\]
where $\bar{U}_{r}^{t,x}=\bar{V}_{r}^{t,x}=0$ and ${\bar{Z}{_{r}^{t,x}=0}}$
for $r\in\left[  0,t\right]  $ and as measures on $\left[  t,T\right]  ,$%
\begin{equation}%
\begin{array}
[c]{l}%
\bar{U}_{s}^{t,x}ds\in\partial\varphi(\bar{Y}_{s}^{t,x})ds,\;\mathbb{\bar{P}%
}\text{--a.s.}\quad\text{and}\medskip\\
\bar{V}_{s}^{t,x}d\bar{A}_{s}^{t,x}\in\partial\psi(\bar{Y}_{s}^{t,x})d\bar
{A}_{s}^{t,x},\;\mathbb{\bar{P}}\text{--a.s.}%
\end{array}
\label{subdiff apartenence 1}%
\end{equation}
In addition, we have%
\begin{equation}
\bar{Y}_{s}^{t,x}=\bar{Y}_{t}^{t,x}\,,\quad\bar{K}_{s}^{1,t,x}=\bar{K}%
_{s}^{2,t,x}=0\,,\quad\bar{M}_{s}^{t,x}=0\,,\quad\text{for all }s\in\left[
0,t\right]  . \label{extension of sol 2}%
\end{equation}
The process $(\bar{Y}^{t,x},{\bar{Z}{^{t,x},\bar{U}}}^{t,x},\bar{V}^{t,x})$ is
${\mathcal{F}}_{s}^{\bar{W}}$--adapted, therefore is $\mathcal{\bar{F}}_{s}%
$--adapted. It can be shown that $\bar{W}$ is an $\mathcal{\bar{F}}_{s}%
$--Wiener process (for the proof, see \cite[Corollary 1.96]{pa-ra/12}).
Therefore, by the definition of the stochastic integral we deduce that
${\bar{M}^{t,x}}$ is an $\mathcal{\bar{F}}_{s}$--martingale.

From It\^{o}'s formula for semimartingales (see, e.g., \cite[Chapter II,
Theorem 32]{pr/04}) applied to (\ref{generalized BSDE 5}) and
(\ref{generalized BSDE 6}), and, since ${\bar{M}}$ and ${\bar{M}^{t,x}}$ are
martingale with respect to the same filtration $\left(  \mathcal{\bar{F}}%
_{s}\right)  _{s\in\left[  0,T\right]  }$, we obtain, for any two stopping
times $\sigma,\tau:\bar{\Omega}\rightarrow\left[  0,T\right]  $, such that
$\sigma\leq\tau$, $\mathbb{\bar{P}}$--a.s.%
\begin{equation}%
\begin{array}
[c]{l}%
\displaystyle|\bar{Y}_{\sigma}-\bar{Y}_{\sigma}^{t,x}|^{2}+{\int_{\sigma
}^{\tau}}d[{\bar{M}}+{\bar{K}}^{1}{+\bar{K}}^{2}{-\bar{M}^{t,x}}-\bar
{K}^{1,t,x}-\bar{K}^{2,t,x}]_{r}\medskip\\
\displaystyle\quad+2{\int_{\sigma}^{\tau}}\langle\bar{Y}_{r-}-\bar{Y}%
_{r}^{t,x},d\bar{K}_{r}^{1}-d\bar{K}_{r}^{1,t,x}+d\bar{K}_{r}^{2}-d\bar{K}%
_{r}^{2,t,x}\rangle\medskip\\
\displaystyle=|\bar{Y}_{\tau}-\bar{Y}_{\tau}^{t,x}|^{2}+2{\int_{\sigma}^{\tau
}}\langle\bar{Y}_{r}-\bar{Y}_{r}^{t,x},f{(r,\bar{X}{{_{r}^{t,x}},\bar{Y}%
_{r})-f{(r,\bar{X}_{r}^{t,x}{,\bar{Y}_{r}^{t,x})}}}\rangle dr}\medskip\\
\displaystyle\quad{+}2{\int_{\sigma}^{\tau}}\langle\bar{Y}_{r}-\bar{Y}%
_{r}^{t,x},{{g}(r,\bar{X}{{_{r}^{t,x}},\bar{Y}_{r})-g}(r,\bar{X}_{r}%
^{t,x}{,\bar{Y}_{r}^{t,x})\rangle d}}\bar{A}_{r}^{t,x}\medskip\\
\displaystyle\quad{-}2{\int_{\sigma}^{\tau}}\langle\bar{Y}_{r-}-\bar{Y}%
_{r}^{t,x},d({\bar{M}}_{r}{-\bar{M}_{r}^{t,x}})\rangle\,,
\end{array}
\label{Ito formula 1}%
\end{equation}
where $[{\bar{M}}+{\bar{K}}^{1}{+\bar{K}}^{2}{-\bar{M}^{t,x}}-\bar{K}%
^{1,t,x}-\bar{K}^{2,t,x}]$ is the quadratic variation process of ${\bar{M}%
}+{\bar{K}}^{1}{+\bar{K}}^{2}{-\bar{M}^{t,x}}-\bar{K}^{1,t,x}-\bar{K}^{2,t,x}$.

By taking%
\begin{equation}%
\begin{array}
[c]{l}%
(\bar{Y}_{r}^{t,x},{\bar{K}}_{r}^{1,t,x},{\bar{K}}_{r}^{2,t,x},{\bar{M}%
_{r}^{t,,x}):=}(\bar{Y}_{T}^{t,x},{\bar{K}}_{T}^{1,t,x},{\bar{K}}_{T}%
^{2,t,x},{\bar{M}_{T}^{t,x})}\text{ and}\medskip\\
(\bar{Y}_{r},{\bar{K}}_{r}^{1},{\bar{K}}_{r}^{2},{\bar{M}_{r}):=}(\bar{Y}%
_{T},{\bar{K}}_{T}^{1},{\bar{K}}_{T}^{2},{\bar{M}}_{T}{)}\text{, whenever
}r\geq T,
\end{array}
\label{extension 1}%
\end{equation}
we extend equality (\ref{Ito formula 1}) to any stopping times $\sigma
,\tau:\bar{\Omega}\rightarrow\lbrack0,\infty)$, such that $\sigma\leq\tau.$

Using the assumptions (\ref{Assumpt. 3}) and (\ref{Assumpt. 7}) on $f$ and $g$
and the auxiliary result below, namely inequality (\ref{subdiff ineq}) (see
the next Lemmas \ref{technical lemma 3} and \ref{monoton}), we see that, for
any stopping times $\sigma,\tau:\bar{\Omega}\rightarrow\lbrack0,\infty)$, such
that $\sigma\leq\tau$, $\mathbb{\bar{P}}$--a.s.%
\begin{equation}%
\begin{array}
[c]{l}%
\displaystyle\mathbb{\bar{E}}|\bar{Y}_{\sigma}-\bar{Y}_{\sigma}^{t,x}%
|^{2}+\mathbb{\bar{E}}{\int_{\sigma}^{\tau}}d[{\bar{M}-\bar{M}^{t,x}}]_{r}%
\leq\mathbb{\bar{E}}|\bar{Y}_{\tau}-\bar{Y}_{\tau}^{t,x}|^{2}+2\beta
\mathbb{\bar{E}}{\int_{\sigma}^{\tau}}|\bar{Y}_{r}-\bar{Y}_{r}^{t,x}|^{2}%
d\bar{Q}_{r}\medskip\\
\displaystyle=\mathbb{\bar{E}}|\bar{Y}_{\tau}-\bar{Y}_{\tau}^{t,x}|^{2}%
+2\beta\mathbb{\bar{E}}{\int_{\bar{Q}_{\sigma}}^{\bar{Q}_{\tau}}}|\bar
{Y}_{\bar{Q}_{r}^{-1}}-\bar{Y}_{\bar{Q}_{r}^{-1}}^{t,x}|^{2}dr,
\end{array}
\label{Gronwall}%
\end{equation}
where%
\begin{equation}
s\longmapsto\bar{Q}_{s}\left(  \omega\right)  :=s+\bar{A}_{s\wedge T}%
^{t,x}\left(  \omega\right)  :[0,\infty)\rightarrow\lbrack0,\infty)
\label{def Q}%
\end{equation}
is a continuous strictly increasing and bijective function and $\bar{Q}^{-1}$
denotes the inverse mapping.

Let us consider the stopping times $\sigma=\bar{Q}_{s_{1}}^{-1}$ and
$\tau=\bar{Q}_{s_{2}}^{-1}$, where $0\leq s_{1}\leq s_{2}$.

We obtain, for any $0\leq s_{1}\leq s_{2}\,,$%
\[
\displaystyle\mathbb{\bar{E}}|\bar{Y}_{\bar{Q}_{s_{1}}^{-1}}-\bar{Y}_{\bar
{Q}_{s_{1}}^{-1}}^{t,x}|^{2}\leq\mathbb{\bar{E}}|\bar{Y}_{\bar{Q}_{s_{2}}%
^{-1}}-\bar{Y}_{\bar{Q}_{s_{2}}^{-1}}^{t,x}|^{2}+2\beta{\int_{s_{1}}^{s_{2}}%
}\mathbb{\bar{E}}|\bar{Y}_{\bar{Q}_{r}^{-1}}-\bar{Y}_{\bar{Q}_{r}^{-1}}%
^{t,x}|^{2}dr
\]
and, using the Gronwall's lemma (see, e.g., \cite[Lemma 12]{ma-ra/07} or
\cite[Proposition 6.69]{pa-ra/12}), we deduce%
\[
\displaystyle\mathbb{\bar{E}}\Big(e^{2\beta s_{1}}|\bar{Y}_{\bar{Q}_{s_{1}%
}^{-1}}-\bar{Y}_{\bar{Q}_{s_{1}}^{-1}}^{t,x}|^{2}\Big)\leq2\beta
\mathbb{\bar{E}}\Big(e^{2\beta s_{2}}|\bar{Y}_{\bar{Q}_{s_{2}}^{-1}}-\bar
{Y}_{\bar{Q}_{s_{2}}^{-1}}^{t,x}|^{2}\Big).
\]
Since%
\[
\Big(e^{2\beta s}|\bar{Y}_{\bar{Q}_{s}^{-1}}-\bar{Y}_{\bar{Q}_{s}^{-1}}%
^{t,x}|^{2}\Big)=0,\;\text{for any }s\geq\bar{Q}_{T},\;\text{a.s.}%
\]
and%
\[
\sup_{s\geq0}\Big(e^{2\beta s}|\bar{Y}_{\bar{Q}_{s}^{-1}}-\bar{Y}_{\bar{Q}%
_{s}^{-1}}^{t,x}|^{2}\Big)\leq e^{2\beta\bar{Q}_{T}}\sup_{r\in\left[
0,T\right]  }|\bar{Y}_{r}-\bar{Y}_{r}^{t,x}|^{2},
\]
we deduce, passing to the limit as $s_{2}\rightarrow\infty$ and using the
Lebesgue dominated convergence theorem, that%
\[
\mathbb{\bar{E}}\Big(e^{2\beta s_{1}}\big|\bar{Y}_{\bar{Q}_{s_{1}}^{-1}}%
-\bar{Y}_{\bar{Q}_{s_{1}}^{-1}}^{t,x}\big|^{2}\Big)=0,
\]
for any $s_{1}\geq0,$ which yields the identification of the limit%
\[
\bar{Y}=\bar{Y}^{t,x}.
\]
From inequality (\ref{Gronwall}) we deduce that%
\[
{\bar{M}}={\bar{M}^{t,x}}%
\]
and from (\ref{generalized BSDE 5}) and (\ref{generalized BSDE 6}) we get%
\[
{\bar{K}}^{1}{+\bar{K}}^{2}={\bar{K}}^{1,t,x}+{\bar{K}}^{2,t,x}.
\]
Finally, from equality (\ref{generalized BSDE 4})%
\[
\bar{Y}_{t_{n}}^{t_{n},x_{n}}=\bar{Y}_{t_{n}}^{n}=-\bar{K}_{T}^{1,n}-\bar
{K}_{T}^{2,n}+{h(\bar{X}}_{T}^{n}){+\int_{t_{n}}^{T}}f{(r,{\bar{X}_{r}%
^{n},\bar{Y}{_{r}^{n}})}dr+}\int_{t_{n}}^{T}g{(r,{\bar{X}_{r}^{n},\bar{Y}%
{_{r}^{n}})d\bar{A}_{r}^{n}}}-\bar{M}_{T}^{n}%
\]
and the pointwise convergence outside the countable set $Q\subset\lbrack0,T)$
(see (\ref{technical ineq 10})), we deduce with the help of
(\ref{technical limit 1}) and (\ref{technical limit 2}) that%
\begin{align*}
\lim_{n\rightarrow\infty}\bar{Y}_{t_{n}}^{n}  &  =\bar{Y}_{t}^{t,x}=-\bar
{K}_{T}^{1,t,x}-\bar{K}_{T}^{2,t,x}+{h(\bar{X}}_{T}^{t,x}){+\int_{t}^{T}%
}f{(r,{\bar{X}_{r}^{t,x},\bar{Y}{_{r}^{t,x}})}dr}\\
&  \quad\quad\quad\;\quad{+}\int_{t}^{T}g{(r,{\bar{X}{_{r}^{t,x}},\bar{Y}%
{_{r}^{t,x}})d\bar{A}_{r}^{t,x}}}-\bar{M}_{T}^{t,x}.
\end{align*}
Since as deterministic processes $Y_{t_{n}}^{t_{n},x_{n}}=\bar{Y}_{t_{n}%
}^{t_{n},x_{n}}$ and $\bar{Y}_{t}^{t,x}=Y_{t}^{t,x},$ we have, along a
subsequence,%
\[
\lim_{n\rightarrow\infty}u\left(  t_{n},x_{n}\right)  =\lim_{n\rightarrow
\infty}Y_{t_{n}}^{t_{n},x_{n}}=\lim_{n\rightarrow\infty}\bar{Y}_{t}%
^{t_{n},x_{n}}=\bar{Y}_{t}^{t,x}=Y_{t}^{t,x}=u\left(  t,x\right)  .
\]
$\medskip$

The last part of the proof consists in showing the next two Lemmas.$\medskip$

\begin{lemma}
[c\`{a}dl\`{a}g subdifferential inequality]\label{technical lemma 3}The limit
process $\left(  \bar{K}^{1},\bar{K}^{2}\right)  $ \textit{satisfies:}%
\begin{equation}%
\begin{array}
[c]{l}%
\displaystyle\mathbb{\bar{E}}{\int_{\sigma}^{\tau}}\langle v\left(  r\right)
-\bar{Y}_{r-},d\left(  \bar{K}_{r}^{1}+\bar{K}_{r}^{2}\right)  \rangle
+\mathbb{\bar{E}}{\int_{\sigma}^{\tau}}\varphi\left(  \bar{Y}_{r}\right)
dr+\mathbb{\bar{E}}{\int_{\sigma}^{\tau}}\psi\left(  \bar{Y}_{r}\right)
d\bar{A}_{r}^{t,x}+\frac{1}{2}\mathbb{\bar{E}}{\int_{\sigma}^{\tau}}d[{\bar
{M}}]_{r}\medskip\\
\displaystyle\leq\frac{1}{2}\mathbb{\bar{E}}{\int_{\sigma}^{\tau}}d[{\bar
{M}+\bar{K}}^{1}{+\bar{K}}^{2}]_{r}+\mathbb{\bar{E}}{\int_{\sigma}^{\tau}%
}\varphi\left(  v\left(  r\right)  \right)  dr+\mathbb{\bar{E}}{\int_{\sigma
}^{\tau}}\psi\left(  v\left(  r\right)  \right)  d\bar{A}_{r}^{t,x}%
\end{array}
\label{subineq}%
\end{equation}
\textit{for any stopping times }$\sigma,\tau:\bar{\Omega}\rightarrow\left(
t,T\right)  $\textit{\ such that }$\sigma\leq\tau$\textit{\ and for any
c\`{a}dl\`{a}g stochastic process }$v$\textit{\ such that}
\[
\mathbb{\bar{E}}\sup_{r\in\left[  0,T\right]  }\left\vert v\left(  r\right)
\right\vert ^{2}<\infty.
\]

\end{lemma}

\noindent\textbf{Proof.} We know (see, e.g. \cite[Proposition 6.26]{pa-ra/12})
that if $\varphi$ is a l.s.c. function such that $\varphi\left(  x\right)
\geq0,$ for all $x\in\mathbb{R}^{d}$, then there exists a sequence of locally
Lipschitz functions $\varphi_{n}:\mathbb{R}^{d}\rightarrow\mathbb{R}$ such
that%
\[%
\begin{array}
[c]{l}%
0\leq\varphi_{1}\left(  x\right)  \leq\cdots\leq\varphi_{j}\left(  x\right)
\leq\cdots\leq\varphi\left(  x\right)  \quad\text{and}\medskip\\
\lim\limits_{j\rightarrow\infty}\varphi_{j}\left(  x\right)  =\varphi\left(
x\right)
\end{array}
\]
(the same conclusion holds true for $\psi$).

Let $\sigma,\tau:\bar{\Omega}\rightarrow\left(  t_{n},T\right)  $ be two
stopping times such that $\sigma\leq\tau$, $\mathbb{\bar{P}}$--a.s.

We fix $\omega\in\bar{\Omega}$ for which (\ref{technical ineq 1}) holds and
$t_{n}<\sigma\left(  \omega\right)  <\tau\left(  \omega\right)  <T$.

Let $\left(  v\left(  t\right)  \right)  _{t\in\mathbb{R}}$ be an arbitrary
c\`{a}dl\`{a}g stochastic process such that $v\left(  s\right)  =v\left(
0\right)  $ for all $s\leq0$, $v\left(  s\right)  =v\left(  T\right)  $, for
all $s\geq T$ and $v\left(  t\right)  \in\overline{\mathrm{Dom}\left(
\varphi\right)  }\cap\overline{\mathrm{Dom}\left(  \psi\right)  }$ for all
$t\in\mathbb{R}.$

Let us define, for $\delta\in\left(  0,1\right)  $ and $R>0$%
\begin{equation}
v_{R}^{\delta}\left(  t\right)  :=\frac{1}{\delta}\int_{t}^{\infty}%
e^{-\frac{r-t}{\delta}}\left[  \frac{R}{1+R}\rho_{R}\left(  r\right)  v\left(
r\right)  +\frac{1}{1+R}u_{0}\right]  dr, \label{def approx}%
\end{equation}
where $u_{0}\in\mathrm{int}\left(  \mathrm{Dom}\left(  \varphi\right)
\right)  \cap\mathrm{int}\left(  \mathrm{Dom}\left(  \psi\right)  \right)  $
and $\rho_{R}\left(  r\right)  :=\mathbb{1}_{\left[  0,R\right]  }\left(
\left\vert v\left(  r\right)  \right\vert +\varphi\left(  v\left(  r\right)
\right)  +\psi\left(  v\left(  r\right)  \right)  \right)  $.

Then $t\mapsto v_{R}^{\delta}\left(  t\right)  $ is a continuous function.

Since $v\left(  r\right)  \rho_{R}\left(  r\right)  \in\overline
{\mathrm{Dom}\left(  \varphi\right)  }$ and%
\[
\varepsilon\overline{\mathrm{Dom}\left(  \varphi\right)  }+\left(
1-\varepsilon\right)  ~\mathrm{int}\left(  \mathrm{Dom}\left(  \varphi\right)
\right)  \subset\mathrm{int}\left(  \mathrm{Dom}\left(  \varphi\right)
\right)  ,\quad\text{for all }\varepsilon\in\lbrack0,1),
\]
the set $\big\{v_{R}^{\delta}\left(  t\right)  :t\in\left[  0,T\right]
\big\}$ is a bounded subset of $\mathrm{int}\left(  \mathrm{Dom}\left(
\varphi\right)  \right)  $ and $\big|v_{R}^{\delta}\left(  t\right)  \big|\leq
R+\left\vert u_{0}\right\vert $, for all $t\in\left[  0,T\right]  .$ Moreover
$t\mapsto\varphi\big(v_{R}^{\delta}\left(  t\right)  \big):\left[  0,T\right]
\rightarrow\lbrack0,\infty)$ is continuous and, by Jensen inequality,%
\begin{equation}
0\leq\varphi\big(v_{R}^{\delta}\left(  t\right)  \big)\leq\frac{1}{\delta}%
\int_{t}^{\infty}e^{-\frac{r-t}{\delta}}\left[  \varphi\left(  \varepsilon
\rho_{R}\left(  r\right)  v\left(  r\right)  +\left(  1-\varepsilon\right)
u_{0}\right)  \right]  dr. \label{technical ineq 2}%
\end{equation}
But%
\begin{equation}%
\begin{array}
[c]{l}%
\displaystyle\varphi\left(  \frac{R}{1+R}\rho_{R}\left(  r\right)  v\left(
r\right)  +\frac{1}{1+R}u_{0}\right)  \leq\frac{R}{1+R}\varphi\left(  \rho
_{R}\left(  r\right)  v\left(  r\right)  +\left(  1-\rho_{R}\left(  r\right)
\right)  0\right)  +\frac{1}{1+R}\varphi\left(  u_{0}\right)  \medskip\\
\displaystyle\leq\frac{R}{1+R}\left[  \rho_{R}\left(  r\right)  \varphi\left(
v\left(  r\right)  \right)  +\left(  1-\rho_{R}\left(  r\right)  \right)
\varphi\left(  0\right)  \right]  +\frac{1}{1+R}\varphi\left(  u_{0}\right)
\medskip\\
\displaystyle=\frac{R}{1+R}\rho_{R}\left(  r\right)  \varphi\left(  v\left(
r\right)  \right)  +\frac{1}{1+R}\varphi\left(  u_{0}\right)  .
\end{array}
\label{technical ineq 3}%
\end{equation}
Hence%
\[
0\leq\varphi\big(v_{R}^{\delta}\left(  t\right)  \big)\leq R+\varphi\left(
u_{0}\right)  .
\]
The same conclusions we have for $t\mapsto\psi\big(v_{R}^{\delta}\left(
t\right)  \big).$

Let $n\in\mathbb{N}^{\ast}$. Using (\ref{elements of subdiff 6}) and
Proposition \ref{equiv-subdiff}, we deduce that, for all $j\in\mathbb{N}%
^{\ast}$ any stopping times $\sigma,\tau:\bar{\Omega}\rightarrow\left(
t_{n},T\right)  ,$with $\sigma\leq\tau$, $\mathbb{\bar{P}}$--a.s. $\omega
\in\bar{\Omega}:$%
\begin{equation}%
\begin{array}
[c]{l}%
\displaystyle\int_{\sigma}^{\tau}\varphi_{j}\left(  \bar{Y}_{r}^{n}\right)
dr+\int_{\sigma}^{\tau}\psi_{j}\left(  \bar{Y}_{r}^{n}\right)  d\bar{A}%
_{r}^{n}\leq\int_{\sigma}^{\tau}\varphi\left(  \bar{Y}_{r}^{n}\right)
dr+\int_{\sigma}^{\tau}\psi\left(  \bar{Y}_{r}^{n}\right)  d\bar{A}_{r}%
^{n}\medskip\\
\displaystyle\leq\int_{\sigma}^{\tau}\langle\bar{Y}_{r}^{n}-v_{R}^{\delta
}\left(  r\right)  ,d\left(  \bar{K}_{r}^{1,n}+\bar{K}_{r}^{2,n}\right)
\rangle+\int_{\sigma}^{\tau}\varphi\left(  v_{r}\right)  dr+\int_{\sigma
}^{\tau}\psi(v_{R}^{\delta}\left(  r\right)  )d\bar{A}_{r}^{n}%
\,,\;\mathbb{\bar{P}}\text{--a.s.}.
\end{array}
\label{technical ineq 1}%
\end{equation}
The proof will be split into several steps. First we extend to $\mathbb{R}$,
by continuity, the stochastic processes from (\ref{technical ineq 1}) as
follows: $\bar{Y}_{r}^{n}=\bar{Y}_{t_{n}}^{n}$, $\bar{A}_{r}^{n}=\bar{K}%
_{r}^{1,n}=\bar{K}_{r}^{2,n}=0$ for all $r\leq t_{n}$ and $\bar{Y}_{r}%
^{n}=\bar{Y}_{T}^{n}$, $\bar{A}_{r}^{n}=\bar{A}_{T}^{n}$, $\bar{K}_{r}%
^{1,n}=\bar{K}_{T}^{1,n}$, $\bar{K}_{r}^{2,n}=\bar{K}_{T}^{2,n}$ for all
$r\geq T.\medskip$

\noindent\textit{Step 1. Passing to the limit as }$n\rightarrow\infty
.\smallskip$

\noindent In the next two steps, let the c\`{a}dl\`{a}g stochastic process $v$
be such that $\mathbb{\bar{E}}\sup_{r\in\left[  0,T\right]  }\left\vert
v\left(  r\right)  \right\vert ^{2}<\infty.$

From (\ref{technical ineq 1}) we see that, for any stopping times $\sigma
,\tau:\bar{\Omega}\rightarrow\left(  t,T\right)  $ with $\sigma\leq\tau$,%
\begin{equation}%
\begin{array}
[c]{l}%
\displaystyle\mathbb{1}_{[0,\sigma)}\left(  t_{n}\right)  \Big(\int_{\sigma
}^{\tau}\varphi_{j}\left(  \bar{Y}_{r}^{n}\right)  dr+\int_{\sigma}^{\tau}%
\psi_{j}\left(  \bar{Y}_{r}^{n}\right)  d\bar{A}_{r}^{n}\Big)\medskip\\
\displaystyle\leq\mathbb{1}_{[0,\sigma)}\left(  t_{n}\right)  \Big(\int
_{\sigma}^{\tau}\varphi\left(  \bar{Y}_{r}^{n}\right)  dr+\int_{\sigma}^{\tau
}\psi\left(  \bar{Y}_{r}^{n}\right)  d\bar{A}_{r}^{n}\Big)\medskip\\
\displaystyle\leq\mathbb{1}_{[0,\sigma)}\left(  t_{n}\right)  \Big(\int
_{\sigma}^{\tau}\langle\bar{Y}_{r}^{n}-v_{R}^{\delta}\left(  r\right)
,d\left(  \bar{K}_{r}^{1,n}+\bar{K}_{r}^{2,n}\right)  \rangle\Big)\medskip\\
\displaystyle\quad+\mathbb{1}_{[0,\sigma)}\left(  t_{n}\right)  \Big(\int
_{\sigma}^{\tau}\varphi(v_{R}^{\delta}\left(  r\right)  )dr+\int_{\sigma
}^{\tau}\psi(v_{R}^{\delta}\left(  r\right)  )d\bar{A}_{r}^{n}\Big)\,.
\end{array}
\label{subdiff apartenence 3}%
\end{equation}
Obviously, $\mathbb{1}_{[0,\sigma)}\left(  t_{n}\right)  \rightarrow1$, as
$n\rightarrow\infty$ and $\mathbb{1}_{[0,\sigma)}\left(  t_{n}\right)  $ is a
$\mathcal{F}_{\sigma}$--random variable.

By It{\^{o}}'s formula applied to (\ref{generalized BSDE 4}) we have for any
two stopping times $\sigma,\tau:\bar{\Omega}\rightarrow\left(  t,T\right)  $
such that $\sigma\leq\tau$,%
\begin{equation}%
\begin{array}
[c]{l}%
\displaystyle\mathbb{\bar{E}}\Big[\mathbb{1}_{[0,\sigma)}\left(  t_{n}\right)
\int_{\sigma}^{\tau}\varphi_{j}\left(  \bar{Y}_{r}^{n}\right)  dr+\mathbb{1}%
_{[0,\sigma)}\left(  t_{n}\right)  \int_{\sigma}^{\tau}\psi_{j}\left(  \bar
{Y}_{r}^{n}\right)  d\bar{A}_{r}^{n}{\Big]}\medskip\\
\displaystyle\leq\mathbb{\bar{E}}\Big[\mathbb{1}_{[0,\sigma)}\left(
t_{n}\right)  \int_{\sigma}^{\tau}\varphi\left(  \bar{Y}_{r}^{n}\right)
dr+\mathbb{1}_{[0,\sigma)}\left(  t_{n}\right)  \int_{\sigma}^{\tau}%
\psi\left(  \bar{Y}_{r}^{n}\right)  d\bar{A}_{r}^{n}{\Big]}\medskip\\
\displaystyle\leq\mathbb{\bar{E}}\Big[{-}\mathbb{1}_{[0,\sigma)}\left(
t_{n}\right)  \int_{\sigma}^{\tau}\langle v_{R}^{\delta}\left(  r\right)
,d\bar{K}_{r}^{1,n}+d\bar{K}_{r}^{2,n}\rangle{\Big]}\medskip\\
\quad\displaystyle+\mathbb{\bar{E}}\Big[\mathbb{1}_{[0,\sigma)}\left(
t_{n}\right)  \int_{\sigma}^{\tau}\varphi(v_{R}^{\delta}\left(  r\right)
)dr+\mathbb{1}_{[0,\sigma)}\left(  t_{n}\right)  \int_{\sigma}^{\tau}%
\psi(v_{R}^{\delta}\left(  r\right)  )d\bar{A}_{r}^{n}{\Big]}\medskip\\
\quad\displaystyle+\frac{1}{2}\mathbb{\bar{E}}\big[\mathbb{1}_{[0,\sigma
)}\left(  t_{n}\right)  {\big(|\bar{Y}_{\tau}^{n}|^{2}-|\bar{Y}_{\sigma}%
^{n}|^{2}\big)}\big]+\mathbb{\bar{E}}\Big[\mathbb{1}_{[0,\sigma)}\left(
t_{n}\right)  \int_{\sigma}^{\tau}\langle\bar{Y}_{r}^{n},f{(r,{\bar{X}_{r}%
^{n},\bar{Y}{_{r}^{n}})\rangle dr}\Big]}\medskip\\
\quad\displaystyle+\mathbb{\bar{E}}\Big[\mathbb{1}_{[0,\sigma)}\left(
t_{n}\right)  {\int_{\sigma}^{\tau}\langle\bar{Y}_{r}^{n},g{(r,{\bar{X}%
_{r}^{n},\bar{Y}{_{r}^{n}})\rangle d{\bar{A}}}_{r}^{n}}\Big]{-}\frac{1}%
{2}\mathbb{\bar{E}}}\big[\mathbb{1}_{[0,\sigma)}{\left(  t_{n}\right)
\big({[{{\bar{M}}^{n}}]_{\tau}-{[{{\bar{M}}^{n}}]_{\sigma}}}\big)}\big]\,.
\end{array}
\label{technical ineq 7}%
\end{equation}
Now,%
\[%
\begin{array}
[c]{l}%
\displaystyle\liminf\limits_{n\rightarrow\infty}\mathbb{\bar{E}}%
\Big[\mathbb{1}_{[0,\sigma)}\left(  t_{n}\right)  \Big(\int_{\sigma}^{\tau
}\varphi_{j}\left(  \bar{Y}_{r}^{n}\right)  dr+\int_{\sigma}^{\tau}\psi
_{j}\left(  \bar{Y}_{r}^{n}\right)  d\bar{A}_{r}^{n}+\frac{1}{2}%
\big({{[{{\bar{M}}^{n}}]_{\tau}-{[{{\bar{M}}^{n}}]_{\sigma}}}\big)\Big)\Big]}%
\medskip\\
\displaystyle\geq\liminf\limits_{n\rightarrow\infty}\mathbb{\bar{E}%
}\Big[\mathbb{1}_{[0,\sigma)}\left(  t_{n}\right)  \int_{\sigma}^{\tau}%
\varphi_{j}\left(  \bar{Y}_{r}^{n}\right)  dr{\Big]}+\liminf
\limits_{n\rightarrow\infty}\mathbb{\bar{E}}\Big[\mathbb{1}_{[0,\sigma
)}\left(  t_{n}\right)  \int_{\sigma}^{\tau}\psi_{j}\left(  \bar{Y}_{r}%
^{n}\right)  dA_{r}^{n}{\Big]}\medskip\\
\quad\displaystyle{+}\frac{1}{2}\liminf\limits_{n\rightarrow\infty
}\mathbb{\bar{E}}\Big[\mathbb{1}_{[0,\sigma)}\left(  t_{n}\right)
{({[{{\bar{M}}^{n}}]_{\tau}-{[{{\bar{M}}^{n}}]_{\sigma}})}\Big]}.
\end{array}
\]

\begin{itemize}
\item By Fatou's lemma we have%
\[
\liminf\limits_{n\rightarrow\infty}\mathbb{\bar{E}}\Big(\mathbb{1}%
_{[0,\sigma)}\left(  t_{n}\right)  \int_{\sigma}^{\tau}\varphi_{j}\left(
\bar{Y}_{r}^{n}\right)  dr\Big)\geq\mathbb{\bar{E}}\Big(\int_{\sigma}^{\tau
}\varphi_{j}\left(  \bar{Y}_{r}\right)  dr\Big).
\]

\item From part $\mathrm{(III)}$ of Theorem \ref{Theorem 1_Annexes} and
Proposition \ref{Proposition 4_Annexes} from the Annexes and using
(\ref{technical ineq 10}), we have for any stopping times $\sigma,\tau
:\bar{\Omega}\rightarrow\left[  0,T\right]  \setminus Q$, $\sigma\leq\tau$,
(where $Q$ is the countable subset of $\left(  0,T\right)  $ defined in
(\ref{technical ineq 10})), that
\[
\lim_{n\rightarrow\infty}\Big(\mathbb{1}_{[0,\sigma)}\left(  t_{n}\right)
\int_{\sigma}^{\tau}\psi_{j}\left(  \bar{Y}_{r}^{n}\right)  d\bar{A}_{r}%
^{n}\Big)=\int_{\sigma}^{\tau}\psi_{j}\left(  \bar{Y}_{r}\right)  d\bar{A}%
_{r}^{t,x}\,,\quad\mathbb{\bar{P}}\text{--a.s.}%
\]
Hence, using Fatou's lemma,%
\begin{align*}
\liminf\limits_{n\rightarrow\infty}\mathbb{\bar{E}}\Big(\mathbb{1}%
_{[0,\sigma)}\left(  t_{n}\right)  \int_{\sigma}^{\tau}\psi_{j}\left(  \bar
{Y}_{r}^{n}\right)  d\bar{A}_{r}^{n}\Big)  &  \geq\mathbb{\bar{E}}%
\liminf\limits_{n\rightarrow\infty}\Big(\mathbb{1}_{[0,\sigma)}\left(
t_{n}\right)  \int_{\sigma}^{\tau}\psi_{j}\left(  \bar{Y}_{r}^{n}\right)
d\bar{A}_{r}^{n}\Big)\\
&  =\mathbb{\bar{E}}\int_{\sigma}^{\tau}\psi_{j}\left(  \bar{Y}_{r}\right)
d\bar{A}_{r}^{t,x}\,.
\end{align*}

\item Using the identity%
\[%
\begin{array}
[c]{l}%
\displaystyle\mathbb{\bar{E}}\big[\mathbb{1}_{[0,\sigma)}{\left(
t_{n}\right)  |{{{\bar{M}}_{\tau}^{n}}-{{{{\bar{M}}_{\sigma}^{n}|}}}^{2}%
}\big]}{=}\mathbb{\bar{E}}\big[\mathbb{1}_{[0,\sigma)}{\left(  t_{n}\right)
\big(|{{{\bar{M}}_{\tau}^{n}}|}}^{2}{{-|{{{{{{\bar{M}}_{\sigma}^{n}}}|}}}%
^{2}-2{{{\bar{M}}_{\sigma}^{n}}}}}\left(  {{{{{{{{{\bar{M}}_{\tau}^{n}-}}}%
\bar{M}}_{\sigma}^{n}}}}}\right)  \big){\big]}\medskip\\
\displaystyle=\mathbb{\bar{E}}\big[\mathbb{1}_{[0,\sigma)}{\left(
t_{n}\right)  }\big({|{{{\bar{M}}_{\tau}^{n}}|}}^{2}{{-|{{{{{{\bar{M}}%
_{\sigma}^{n}}}|}}}^{2}}\big)\big]={\mathbb{\bar{E}}}\big[\mathbb{1}%
_{[0,\sigma)}{\left(  t_{n}\right)  \big({[{{\bar{M}}^{n}}]_{\tau}-{[{{\bar
{M}}^{n}}]_{\sigma}}}\big)}\big],}%
\end{array}
\]
and from (\ref{technical ineq 10}) we deduce that%
\[%
\begin{array}
[c]{l}%
\displaystyle\liminf\limits_{n\rightarrow\infty}{\mathbb{\bar{E}}%
}\big[\mathbb{1}_{[0,\sigma)}{\left(  t_{n}\right)  \big({[{{\bar{M}}^{n}%
}]_{\tau}-{[{{\bar{M}}^{n}}]_{\sigma}}}\big)}\big]=\liminf
\limits_{n\rightarrow\infty}\mathbb{\bar{E}}\big(\mathbb{1}_{[0,\sigma
)}{\left(  t_{n}\right)  |{{{\bar{M}}_{\tau}^{n}}-{{{{\bar{M}}_{\sigma}^{n}|}%
}}^{2}}\big)}\medskip\\
\displaystyle\geq\mathbb{\bar{E}}{(|{{{\bar{M}}_{\tau}}-{{{{\bar{M}}_{\sigma
}|}}}^{2})}}={{\mathbb{\bar{E}}}(|{{{\bar{M}}_{\tau}|}}}^{2}{{-|{{{{\bar{M}%
}_{\sigma}|}}}^{2})}}={{\mathbb{\bar{E}}{({[{{\bar{M}}}]_{\tau}-{[{{\bar{M}}%
}]_{\sigma}}),}}}}%
\end{array}
\]
for any stopping times $\sigma,\tau:\bar{\Omega}\rightarrow\left[  0,T\right]
\setminus Q$ with $\sigma\leq\tau.$

\item Using again (\ref{technical ineq 10}) and part $\mathrm{(I)}$ of Theorem
\ref{Theorem 1_Annexes} from the Annexes, we deduce that%
\[
\lim_{n\rightarrow\infty}\Big(\mathbb{1}_{[0,\sigma)}\left(  t_{n}\right)
\int_{\sigma}^{\tau}\langle v_{R}^{\delta}\left(  r\right)  ,d\bar{K}%
_{r}^{1,n}+d\bar{K}_{r}^{2,n}\rangle\Big)=\int_{\sigma}^{\tau}\langle
v_{R}^{\delta}\left(  r\right)  ,d\bar{K}_{r}^{1}+d\bar{K}_{r}^{2}\rangle,
\]
for any stopping times $\sigma,\tau:\bar{\Omega}\rightarrow\left[  0,T\right]
\setminus Q$ with $\sigma\leq\tau.$

But, for all $1<p<2,$%
\[
\sup_{n\in\mathbb{N}^{\ast}}\mathbb{\bar{E}}\Big|\mathbb{1}_{[0,\sigma
)}\left(  t_{n}\right)  \int_{\sigma}^{\tau}\langle v_{R}^{\delta}\left(
r\right)  ,d\bar{K}_{r}^{1,n}\rangle\Big|^{p}\leq C\left(  R+\varphi\left(
u_{0}\right)  \right)  ^{p}\sup_{n\in\mathbb{N}^{\ast}}\Big[\mathbb{E}%
\int_{\sigma}^{\tau}|U_{r}^{n}|^{2}dr\Big]^{p/2}<\infty
\]
and%
\[%
\begin{array}
[c]{l}%
\displaystyle\sup_{n\in\mathbb{N}^{\ast}}\mathbb{\bar{E}}\Big|\mathbb{1}%
_{[0,\sigma)}\left(  t_{n}\right)  \int_{\sigma}^{\tau}\langle v_{R}^{\delta
}\left(  r\right)  ,d\bar{K}_{r}^{2,n}\rangle\Big|^{p}\\
\displaystyle\leq\left(  R+\psi\left(  u_{0}\right)  \right)  ^{p}\sup
_{n\in\mathbb{N}^{\ast}}\mathbb{E}\Big[\left(  A_{t}^{n}\right)
^{p/2}\big(\int_{\sigma}^{\tau}|V_{r}^{n}|^{2}dA_{r}^{n}\big)^{p/2}%
\Big]\medskip\\
\displaystyle\leq\left(  R+\psi\left(  u_{0}\right)  \right)  ^{p}\,\sup
_{n\in\mathbb{N}^{\ast}}\Big[\mathbb{E}\left(  A_{t}^{n}\right)
^{p/(2-p)}\Big]^{\left(  2-p\right)  /2}\,\sup_{n\in\mathbb{N}^{\ast}%
}\Big[\mathbb{E}\int_{\sigma}^{\tau}|V_{r}^{n}|^{2}dA_{r}^{n}\Big]^{p/2}%
<\infty.
\end{array}
\]
Hence, by the uniform integrability property (see, e.g. \cite[Proposition
1.23]{pa-ra/12}), we deduce%
\[
\lim_{n\rightarrow\infty}\mathbb{\bar{E}}\Big(\mathbb{1}_{[0,\sigma)}\left(
t_{n}\right)  \int_{\sigma}^{\tau}\langle v_{R}^{\delta}\left(  r\right)
,d\bar{K}_{r}^{1,n}+d\bar{K}_{r}^{2,n}\rangle\Big)=\mathbb{\bar{E}}%
\int_{\sigma}^{\tau}\langle v_{R}^{\delta}\left(  r\right)  ,d\bar{K}_{r}%
^{1}+d\bar{K}_{r}^{2}\rangle,
\]
for any stopping times $\sigma,\tau:\bar{\Omega}\rightarrow\left[  0,T\right]
\setminus Q$ with $\sigma\leq\tau.$

\item Since $r\mapsto\varphi(v_{R}^{\delta}\left(  r\right)  )$ and
$r\mapsto\psi(v_{R}^{\delta}\left(  r\right)  )$ are continuous functions on
$\left[  0,T\right]  $,%
\begin{align*}
&  \lim_{n\rightarrow\infty}\mathbb{\bar{E}}\Big[\mathbb{1}_{[0,\sigma
)}\left(  t_{n}\right)  \int_{\sigma}^{\tau}\varphi(v_{R}^{\delta}\left(
r\right)  )dr+\mathbb{1}_{[0,\sigma)}\left(  t_{n}\right)  \int_{\sigma}%
^{\tau}\psi(v_{R}^{\delta}\left(  r\right)  )d\bar{A}_{r}^{n}{\Big]}\\
&  =\mathbb{\bar{E}}\Big[\mathbb{1}_{[0,\sigma)}\left(  t_{n}\right)
\int_{\sigma}^{\tau}\varphi(v_{R}^{\delta}\left(  r\right)  )dr+\mathbb{1}%
_{[0,\sigma)}\left(  t_{n}\right)  \int_{\sigma}^{\tau}\psi(v_{R}^{\delta
}\left(  r\right)  )d\bar{A}_{r}{\Big]}%
\end{align*}
for any stopping times $\sigma,\tau:\bar{\Omega}\rightarrow\left[  0,T\right]
$ with $\sigma\leq\tau.$

\item From (\ref{technical ineq 10}) and the uniform square integrability of
${\bar{Y}^{n}}$ (the estimate (\ref{mg Y 2}) also holds ${\bar{Y}}^{n}$ with
$p>2$ since ${Y}^{n}{\sim\bar{Y}^{n}}$) we deduce that%
\[
\lim_{n\rightarrow\infty}\mathbb{\bar{E}}{\big[\mathbb{1}_{[0,\sigma)}\left(
t_{n}\right)  \big(|\bar{Y}_{\tau}^{n}|^{2}-|\bar{Y}_{\sigma}^{n}%
|^{2}\big)\big]=\mathbb{\bar{E}}\big[|\bar{Y}_{\tau}|^{2}-|\bar{Y}_{\sigma
}|^{2}\big]},
\]
for any stopping times $\sigma,\tau:\bar{\Omega}\rightarrow\left[  0,T\right]
\setminus Q$ with $\sigma\leq\tau.$

\item Next, by (\ref{technical ineq 10}), we have, for all $r\in\left[
0,T\right]  \setminus Q$
\[
\bar{F}_{r}^{n}:=\mathbb{1}_{[0,\sigma)}\left(  t_{n}\right)  \langle\bar
{Y}_{r}^{n},f{(r,{\bar{X}_{r}^{n},\bar{Y}{_{r}^{n}})\rangle}}\rightarrow
\langle\bar{Y}_{r},f{(r,{\bar{X}{_{r}^{t,x}},\bar{Y}{_{r}})}}\rangle
\text{,}\quad\text{as }n\rightarrow\infty,
\]
and by (\ref{mg Y 2}) and ${Y}^{n}{\sim\bar{Y}^{n}:}$
\[
\sup_{n\in\mathbb{N}^{\ast}}\mathbb{\bar{E}}\int_{\sigma}^{\tau}\left(
\bar{F}_{r}^{n}\right)  ^{2}dr\leq2\gamma^{2}{\,\sup_{n\in\mathbb{N}^{\ast}%
}\mathbb{\bar{E}}}\Big[\sup_{r\in\left[  0,T\right]  }\left\vert \bar{Y}%
_{r}^{n}\right\vert ^{2}+\sup_{r\in\left[  0,T\right]  }\left\vert \bar{Y}%
_{r}^{n}\right\vert ^{4}{\Big]}<\infty.
\]
Hence $\left(  \bar{F}^{n}\right)  _{n\in\mathbb{N}^{\ast}}$ is uniformly
integrable on $\bar{\Omega}\times\left[  0,T\right]  $ and
\[
\lim_{n\rightarrow\infty}\mathbb{\bar{E}}\Big[\int_{\sigma}^{\tau}%
\mathbb{1}_{[0,\sigma)}\left(  t_{n}\right)  \langle\bar{Y}_{r}^{n}%
,f{(r,{\bar{X}_{r}^{n},\bar{Y}{_{r}^{n}})\rangle dr}\Big]}={\mathbb{\bar{E}}%
}\int_{\sigma}^{\tau}\langle\bar{Y}_{r},f{(r,{\bar{X}{_{r}^{t,x}},\bar{Y}%
{_{r}})\rangle dr,}}%
\]
for any stopping times $\sigma,\tau:\bar{\Omega}\rightarrow\left[  0,T\right]
$ with $\sigma\leq\tau.$

\item Since $\left(  r,x,y\right)  {\mapsto y\cdot g}\left(  r,x,y\right)  $
is locally Lipschitz, by Proposition \ref{Proposition 4_Annexes} and part
$\mathrm{(III)}$ of Theorem \ref{Theorem 1_Annexes} we get for any stopping
times $\sigma,\tau:\bar{\Omega}\rightarrow\left[  0,T\right]  \setminus Q$
with $\sigma\leq\tau:$%
\begin{equation}
\bar{\xi}_{n}:=\mathbb{1}_{[0,\sigma)}\left(  t_{n}\right)  \int_{\sigma
}^{\tau}\langle\bar{Y}_{r}^{n},g{(r,{\bar{X}_{r}^{n},\bar{Y}{_{r}^{n}})\rangle
d\bar{A}}}_{r}^{n}\rightarrow\int_{\sigma}^{\tau}\langle\bar{Y}_{r}%
,g{(r,{\bar{X}_{r},\bar{Y}{_{r}})d\bar{A}}}_{r}\rangle~,\quad\text{as
}n\rightarrow\infty. \label{Lipschitz for g}%
\end{equation}
By (\ref{cont X}--$b$) and (\ref{mg Y 2}),%
\[
\sup_{n\in\mathbb{N}^{\ast}}\mathbb{\bar{E}}\left(  \bar{\xi}_{n}^{2}\right)
\leq2\gamma^{2}{\,\sup_{n\in\mathbb{N}^{\ast}}\mathbb{\bar{E}}}\Big[\big(\sup
_{r\in\left[  0,T\right]  }\left\vert \bar{Y}_{r}^{n}\right\vert ^{2}%
+\sup_{r\in\left[  0,T\right]  }\left\vert \bar{Y}_{r}^{n}\right\vert
^{4}\big)|{{\bar{A}}}_{T}^{n}|^{2}{\Big]}\leq C<\infty.
\]
It follows, using again the uniformly integrability criterion to pass to the
limit under the integrals, that%
\[
\lim_{n\rightarrow\infty}\mathbb{\bar{E}}\Big[\mathbb{1}_{[0,\sigma)}\left(
t_{n}\right)  \int_{\sigma}^{\tau}\langle\bar{Y}_{r}^{n},g{(r,{\bar{X}_{r}%
^{n},\bar{Y}{_{r}^{n}})\rangle d\bar{A}}}_{r}^{n}{\Big]=\mathbb{\bar{E}}}%
\int_{\sigma}^{\tau}\langle\bar{Y}_{r},g{(r,{\bar{X}_{r}^{t,x},\bar{Y}{_{r}%
})\rangle d\bar{A}}}_{r}^{t,x}\,{,}%
\]
for any stopping times $\sigma,\tau:\bar{\Omega}\rightarrow\left[  0,T\right]
\setminus Q$ with $\sigma\leq\tau.$
\end{itemize}

Passing to the $\liminf\limits_{n\rightarrow\infty}$ in
(\ref{technical ineq 7}) and summarizing, from here above, the convergences of
the all terms from (\ref{technical ineq 7}) we infer that%
\[%
\begin{array}
[c]{l}%
\displaystyle\mathbb{\bar{E}}{\Big[}\int_{\sigma}^{\tau}\varphi_{j}\left(
\bar{Y}_{r}\right)  dr+\int_{\sigma}^{\tau}\psi_{j}\left(  \bar{Y}_{r}\right)
d\bar{A}_{r}^{t,x}+\frac{1}{2}{{{({{{[{{\bar{M}}}]_{\tau}-{[{{\bar{M}}%
}]_{\sigma}}}})}}}\Big]}\medskip\\
\displaystyle\leq\mathbb{\bar{E}}{\Big[-}\int_{\sigma}^{\tau}\langle
v_{R}^{\delta}\left(  r\right)  ,d\bar{K}_{r}^{1}+d\bar{K}_{r}^{2}\rangle
+\int_{\sigma}^{\tau}\varphi(v_{R}^{\delta}\left(  r\right)  )dr+\int_{\sigma
}^{\tau}\psi(v_{R}^{\delta}\left(  r\right)  )d\bar{A}_{r}^{t,x}{\Big]}%
+\frac{1}{2}\mathbb{\bar{E}}{\big[|\bar{Y}_{\tau}|^{2}-|\bar{Y}_{\sigma}%
|^{2}\big]}\medskip\\
\quad\displaystyle+\mathbb{\bar{E}}{\Big[}\int_{\sigma}^{\tau}{\langle}\bar
{Y}_{r},f{(r,{\bar{X}{_{r}^{t,x}},\bar{Y}{_{r}})\rangle dr}+\int_{\sigma
}^{\tau}\langle\bar{Y}_{r},g{(r,{\bar{X}_{r}^{t,x},\bar{Y}{_{r}})\rangle d}%
}\bar{A}_{r}^{t,x}\Big],}%
\end{array}
\]
for any stopping times $\sigma,\tau:\bar{\Omega}\rightarrow\left(  t,T\right)
\setminus Q$ with $\sigma\leq\tau.$

Taking into account the right continuity of the above integrals (we use also
Proposition \ref{Proposition 3_Annexes} from the Annexes), we see that, in
fact, this inequality takes place for any stopping times $\sigma,\tau
:\bar{\Omega}\rightarrow\left(  t,T\right)  $ with $\sigma\leq\tau.$

After that, using Beppo Levi monotone convergence theorem as $j\nearrow\infty$
and It\^{o}'s formula for ${|\bar{Y}_{r}|^{2}}$ on $\left[  \sigma
,\tau\right]  $, we deduce that%
\begin{equation}%
\begin{array}
[c]{l}%
\displaystyle\mathbb{\bar{E}}{\Big[}\int_{\sigma}^{\tau}\varphi\left(  \bar
{Y}_{r}\right)  dr+\int_{\sigma}^{\tau}\psi\left(  \bar{Y}_{r}\right)
d\bar{A}_{r}^{t,x}+\frac{1}{2}{{{({{{{{[{{\bar{M}}}]_{\tau}-{[{{\bar{M}}%
}]_{\sigma}}}}}})}}}\Big]}\medskip\\
\leq\displaystyle\mathbb{\bar{E}}{\Big[}\int_{\sigma}^{\tau}\langle\bar
{Y}_{r-}-v_{R}^{\delta}\left(  r\right)  ,d\left(  \bar{K}_{r}^{1}+\bar{K}%
_{r}^{2}\right)  \rangle+\int_{\sigma}^{\tau}\varphi(v_{R}^{\delta}\left(
r\right)  )dr+\int_{\sigma}^{\tau}\psi(v_{R}^{\delta}\left(  r\right)
)d\bar{A}_{r}^{t,x}{\Big]}\medskip\\
\displaystyle\quad+\frac{1}{2}\mathbb{\bar{E}}\int_{\sigma}^{\tau}d[{\bar{M}%
}+{\bar{K}}^{1}{+\bar{K}}^{2}]_{r}\,{.}%
\end{array}
\label{subdiff apartenence 6}%
\end{equation}
\noindent\textit{Step 2. Passing to the limit as }$\delta\searrow0$\textit{
and }$R\nearrow\infty.$

Now, using inequalities (\ref{technical ineq 2}--\ref{technical ineq 3}) (and
similarly for $\psi$), we have from the last inequality
(\ref{subdiff apartenence 6})%
\[%
\begin{array}
[c]{l}%
\displaystyle\mathbb{\bar{E}}\int_{\sigma}^{\tau}\varphi\left(  \bar{Y}%
_{r}\right)  dr+\mathbb{\bar{E}}\int_{\sigma}^{\tau}\psi\left(  \bar{Y}%
_{r}\right)  d\bar{A}_{r}^{t,x}+\frac{1}{2}\mathbb{\bar{E}}\int_{\sigma}%
^{\tau}{{{d{{{{{[{{\bar{M}}}]_{r}}}}}}}}}\medskip\\
\displaystyle\leq\mathbb{\bar{E}}\int_{\sigma}^{\tau}\langle\bar{Y}_{r-}%
-v_{R}^{\delta}\left(  r\right)  ,d\left(  \bar{K}_{r}^{1}+\bar{K}_{r}%
^{2}\right)  \rangle+\frac{1}{2}\mathbb{\bar{E}}\int_{\sigma}^{\tau}d[{\bar
{M}}+{\bar{K}}^{1}{+\bar{K}}^{2}]_{r}\medskip\\
\displaystyle\quad+\mathbb{\bar{E}}\int_{\sigma}^{\tau}\Big(\frac{1}{\delta
}\int_{r}^{\infty}e^{-\frac{s-r}{\delta}}\Big[\frac{R}{1+R}\rho_{R}\left(
s\right)  \varphi\left(  v\left(  s\right)  \right)  +\frac{1}{1+R}%
\varphi\left(  u_{0}\right)  \Big]ds\Big)dr\medskip\\
\displaystyle\quad+\mathbb{\bar{E}}\int_{\sigma}^{\tau}\Big(\frac{1}{\delta
}\int_{r}^{\infty}e^{-\frac{s-r}{\delta}}\Big[\frac{R}{1+R}\rho_{R}\left(
s\right)  \psi\left(  v\left(  s\right)  \right)  +\frac{1}{1+R}\psi\left(
u_{0}\right)  \Big]ds\Big)d\bar{A}_{r}^{t,x}\,.
\end{array}
\]
Passing to the limit for $\delta\rightarrow0$, we obtain, by the Lebesgue
dominated convergence theorem, that%
\[%
\begin{array}
[c]{l}%
\displaystyle\mathbb{\bar{E}}\int_{\sigma}^{\tau}\varphi\left(  \bar{Y}%
_{r}\right)  dr+\mathbb{\bar{E}}\int_{\sigma}^{\tau}\psi\left(  \bar{Y}%
_{r}\right)  d\bar{A}_{r}^{t,x}+\frac{1}{2}\mathbb{\bar{E}}\int_{\sigma}%
^{\tau}{{{d{{{{{[{{\bar{M}}}]_{r}}}}}}}}}\medskip\\
\displaystyle\leq\mathbb{\bar{E}}\int_{\sigma}^{\tau}\langle\bar{Y}_{r-}%
-\frac{R}{1+R}\rho_{R}\left(  r\right)  v\left(  r\right)  -\frac{1}{1+R}%
u_{0},d\left(  \bar{K}_{r}^{1}+\bar{K}_{r}^{2}\right)  \rangle\medskip\\
\displaystyle\quad+\mathbb{\bar{E}}\int_{\sigma}^{\tau}\frac{R}{1+R}\rho
_{R}\left(  r\right)  \varphi\left(  v\left(  r\right)  \right)  dr+\frac
{1}{1+R}\varphi\left(  u_{0}\right)  T\medskip\\
\displaystyle\quad+\mathbb{\bar{E}}\int_{\sigma}^{\tau}\frac{R}{1+R}\rho
_{R}\left(  r\right)  \psi\left(  v\left(  r\right)  \right)  d\bar{A}%
_{r}^{t,x}+\frac{1}{1+R}\psi\left(  u_{0}\right)  \mathbb{\bar{E}}\bar{A}%
_{T}^{t,x}+\frac{1}{2}\mathbb{\bar{E}}\int_{\sigma}^{\tau}d[{\bar{M}}+{\bar
{K}}^{1}{+\bar{K}}^{2}]_{r}\,.
\end{array}
\]
We pass now to the limit as $R\nearrow\infty$. Using, in the second member of
this inequality, the Lebesgue theorem (for the first integral) and Beppo Levi
theorem (for the next two integrals), it follows that for any c\`{a}dl\`{a}g
process $v$ such that $\mathbb{\bar{E}}\sup_{r\in\left[  0,T\right]
}\left\vert v\left(  r\right)  \right\vert ^{2}<\infty$ and any stopping times
$\sigma,\tau:\bar{\Omega}\rightarrow\left(  t,T\right)  $ with $\sigma\leq
\tau,$%
\begin{equation}%
\begin{array}
[c]{l}%
\displaystyle\mathbb{\bar{E}}\int_{\sigma}^{\tau}\varphi\left(  \bar{Y}%
_{r}\right)  dr+\mathbb{\bar{E}}\int_{\sigma}^{\tau}\psi\left(  \bar{Y}%
_{r}\right)  d\bar{A}_{r}^{t,x}+\frac{1}{2}\mathbb{\bar{E}}\int_{\sigma}%
^{\tau}{{{d{{{{{[{{\bar{M}}}]_{r}}}}}}}}}\medskip\\
\displaystyle\leq\mathbb{\bar{E}}\int_{\sigma}^{\tau}\langle\bar{Y}%
_{r-}-v\left(  r\right)  ,d\left(  \bar{K}_{r}^{1}+\bar{K}_{r}^{2}\right)
\rangle+\frac{1}{2}\mathbb{\bar{E}}\int_{\sigma}^{\tau}d[{\bar{M}}+{\bar{K}%
}^{1}{+\bar{K}}^{2}]_{r}\medskip\\
\displaystyle\quad+\mathbb{\bar{E}}\int_{\sigma}^{\tau}\varphi\left(  v\left(
r\right)  \right)  dr+\mathbb{\bar{E}}\int_{\sigma}^{\tau}\psi\left(  v\left(
r\right)  \right)  d\bar{A}_{r}^{t,x},
\end{array}
\label{subdiff apartenence 8}%
\end{equation}
which represents the conclusion of Lemma \ref{technical lemma 3}.\hfill

\begin{lemma}
[c\`{a}dl\`{a}g monotonicity property]\label{monoton}The following inequality
holds true for any $\sigma,\tau$ stopping times $\sigma,\tau:\bar{\Omega
}\rightarrow\lbrack0,\infty)$, such that $\sigma\leq\tau$,%
\begin{equation}%
\begin{array}
[c]{l}%
\displaystyle\mathbb{\bar{E}}{\int_{\sigma}^{\tau}}\langle\bar{Y}_{r-}-\bar
{Y}_{r}^{t,x},d\bar{K}_{r}^{1}-d\bar{K}_{r}^{1,t,x}+d\bar{K}_{r}^{2}-d\bar
{K}_{r}^{2,t,x}\rangle\medskip\\
\displaystyle\quad+\frac{1}{2}\mathbb{\bar{E}}{\int_{\sigma}^{\tau}}d[{\bar
{M}}+{\bar{K}}^{1}{+\bar{K}}^{2}{-\bar{M}^{t,x}}-\bar{K}^{1,t,x}-\bar
{K}^{2,t,x}]_{r}\geq\frac{1}{2}\mathbb{\bar{E}}{\int_{\sigma}^{\tau}}%
d[{\bar{M}-\bar{M}^{t,x}}]_{r}\,.
\end{array}
\label{subdiff ineq}%
\end{equation}

\end{lemma}

\noindent\textbf{Proof.} Relation (\ref{subdiff apartenence 1}) becomes
$d(\bar{K}_{r}^{1,t,x}+\bar{K}_{r}^{2,t,x})\in\partial\varphi(\bar{Y}%
_{r}^{t,x})dr+\partial\psi(\bar{Y}_{r}^{t,x})d\bar{A}_{r}^{t,x}$. Using
Proposition \ref{equiv-subdiff} it follows, in a similar manner as the proof
of (\ref{subdiff apartenence 8}), that for any c\`{a}dl\`{a}g process
$v^{\prime}$ and any stopping times $\sigma,\tau:\bar{\Omega}\rightarrow
\left(  t,T\right)  $ with $\sigma\leq\tau,$%
\begin{equation}%
\begin{array}
[c]{l}%
\displaystyle\mathbb{\bar{E}}\int_{\sigma}^{\tau}\langle v^{\prime}\left(
r\right)  -\bar{Y}_{r}^{t,x},d\left(  \bar{K}_{r}^{1,t,x}+\bar{K}_{r}%
^{2,t,x}\right)  \rangle+\mathbb{\bar{E}}\int_{\sigma}^{\tau}\varphi\left(
\bar{Y}_{r}^{t,x}\right)  dr+\mathbb{\bar{E}}\int_{\sigma}^{\tau}\psi\left(
\bar{Y}_{r}^{t,x}\right)  d\bar{A}_{r}^{t,x}\medskip\\
\displaystyle\leq\mathbb{\bar{E}}\int_{\sigma}^{\tau}\varphi\left(  v^{\prime
}\left(  r\right)  \right)  dr+\mathbb{\bar{E}}\int_{\sigma}^{\tau}\psi\left(
v^{\prime}\left(  r\right)  \right)  d\bar{A}_{r}^{t,x}%
\end{array}
\label{subdiff apartenence 5}%
\end{equation}
(in this case, by the continuity of $\bar{K}^{1,t,x}$ and $\bar{K}^{2,t,x}$,
the quadratic variation and quadratic covariation are zero).

Taking $v=\bar{Y}^{t,x}$ in inequality (\ref{subdiff apartenence 8}) and
$v^{\prime}=\bar{Y}$ in (\ref{subdiff apartenence 5}) we deduce%
\begin{equation}%
\begin{array}
[c]{l}%
\displaystyle\mathbb{\bar{E}}\int_{\sigma}^{\tau}\langle\bar{Y}_{r-}-\bar
{Y}_{r}^{t,x},d\bar{K}_{r}^{1}+d\bar{K}_{r}^{2}-d\bar{K}_{r}^{1,t,x}-d\bar
{K}_{r}^{2,t,x}\rangle\medskip\\
\displaystyle\quad+\frac{1}{2}\mathbb{\bar{E}}\int_{\sigma}^{\tau}d[{\bar{M}%
}+{\bar{K}}^{1}{+\bar{K}}^{2}]_{r}\geq\frac{1}{2}\mathbb{\bar{E}}\int_{\sigma
}^{\tau}d[{\bar{M}}]_{r}\,.
\end{array}
\label{subdiff apartenence 9}%
\end{equation}
Of course, using extension (\ref{extension 1}), we have%
\[
d[{\bar{M}}+{\bar{K}}^{1}{+\bar{K}}^{2}]_{r}=0=d[{\bar{M}}]_{r}\,,\text{ as
measures on }\left(  T,\infty\right)  ,
\]
therefore the inequality (\ref{subdiff apartenence 9}) is true for any
stopping times $\sigma,\tau:\bar{\Omega}\rightarrow\left(  t,T\right)  $ with
$\sigma\leq\tau.$

Using the extensions (\ref{extension of sol 1}) and (\ref{extension of sol 2})
we conclude that the inequality (\ref{subdiff apartenence 9}) is true for any
two stopping times $\sigma,\tau:\bar{\Omega}\rightarrow\lbrack0,\infty)$ with
$\sigma\leq\tau$.

Using once again the continuity of ${\bar{M}^{t,x}}$ and of the bounded
variation processes $\bar{K}^{1,t,x}$ and $\bar{K}^{2,t,x}$, we obtain (see
\cite[Section 6, Chapter II]{pr/04})%
\[%
\begin{array}
[c]{l}%
\displaystyle[{\bar{M}}+{\bar{K}}^{1}{+\bar{K}}^{2}{-\bar{M}^{t,x}}-\bar
{K}^{1,t,x}-\bar{K}^{2,t,x}]_{r}-[{\bar{M}-\bar{M}^{t,x}}]_{r}\medskip\\
\displaystyle=[{\bar{K}}^{1}{+\bar{K}}^{2}-\bar{K}^{1,t,x}-\bar{K}%
^{2,t,x}]_{r}+2[{\bar{M}-\bar{M}^{t,x}},{\bar{K}}^{1}{+\bar{K}}^{2}-\bar
{K}^{1,t,x}-\bar{K}^{2,t,x}]_{r}\medskip\\
\displaystyle=[{\bar{K}}^{1}{+\bar{K}}^{2}]_{r}+2[{\bar{M}},{\bar{K}}%
^{1}{+\bar{K}}^{2}]_{r}\medskip\\
\displaystyle=[{\bar{M}+\bar{K}}^{1}{+\bar{K}}^{2}]_{r}-[{\bar{M}}]_{r}\,,
\end{array}
\]
where $[{\bar{M}},{\bar{K}}^{1}{+\bar{K}}^{2}]$ is the quadratic covariation
of ${\bar{M}}$ and ${\bar{K}}^{1}{+\bar{K}}^{2}$.

Therefore, the proof of Lemma \ref{monoton} is complete.\hfill
\end{proof}

\section{Annexes: C\`{a}dl\`{a}g bounded variation functions}

\subsection{Skorohod space}

We say that $x:\mathbb{R}_{+}\rightarrow\mathbb{R}^{d}$ is a c\`{a}dl\`{a}g
function if for every $t\in\mathbb{R}_{+}$ the left limit $x_{t-}%
:=\lim\limits_{s\nearrow t}x_{s}$ and the right limit $x_{t+}:=\lim
\limits_{s\searrow t}x_{s}$ exist in $\mathbb{R}^{d}$ and $x_{t+}=x_{t}$ for
all $t\geq0;$ by convention $x_{0-}=x_{0}.$

Denote by $\mathbb{D}\left(  \mathbb{R}_{+};\mathbb{R}^{d}\right)  $ the set
of c\`{a}dl\`{a}g functions $x:\mathbb{R}_{+}\rightarrow\mathbb{R}^{d}$ and
$\mathbb{D}\left(  \left[  0,T\right]  ;\mathbb{R}^{d}\right)  \subset
\mathbb{D}\left(  \mathbb{R}_{+};\mathbb{R}^{d}\right)  $ is the subspace of
paths $x$ that stop at the instant $T$ that is $x\in\mathbb{D}\left(  \left[
0,T\right]  ;\mathbb{R}^{d}\right)  $ if $x\in\mathbb{D}\left(  \mathbb{R}%
_{+},\mathbb{R}^{d}\right)  $ and $x_{t}=x_{t}^{T}:=x_{t\wedge T}$ for all
$t\geq0.$ If $y:\left[  0,T\right]  \rightarrow\mathbb{R}^{d}$ then by
convention we consider $y:\mathbb{R}_{+}\rightarrow\mathbb{R}^{d}$ with
$y_{s}=y_{T}$ for all $s\geq T.$ The spaces of continuous functions will be
denoted by $\mathcal{C}\left(  \mathbb{R}_{+};\mathbb{R}^{d}\right)  $ and
$\mathcal{C}\left(  \left[  0,T\right]  ;\mathbb{R}^{d}\right)  $,
respectively.\medskip

We say that $\pi=\left\{  t_{0},t_{1},t_{2},\ldots\right\}  $ is a partition
of $\mathbb{R}_{+}$ if $0=t_{0}<t_{1}<t_{2}<\,\ldots$ and $t_{n}%
\rightarrow+\infty.$ Let $\pi$ be a partition and $r\in\pi$. We denote by
$r^{\prime}$ the successor of $r$ in the partition $\pi$, i.e. if $r=t_{i}$
then $r^{\prime}:=t_{i+1}.$ We define $\left\Vert \pi\right\Vert
:=\sup\left\{  r^{\prime}-r:r\in\pi\right\}  .$The set of all partitions of
$\mathbb{R}_{+}$ will be denoted $\mathcal{P}_{\mathbb{R}_{+}}\,.$

Given a function $x:\mathbb{R}_{+}\rightarrow\mathbb{R}^{d}$ we
define$\smallskip$

\noindent$\bullet\quad$the norm sup by: $\left\Vert x\right\Vert _{T}%
=\sup_{t\in\left[  0,T\right]  }\left\vert x_{t}\right\vert \quad$%
and$\quad\left\Vert x\right\Vert _{\infty}=\sup_{t\geq0}\left\vert
x_{t}\right\vert ;\smallskip$

\noindent$\bullet\quad$the oscillation of $x$ the on a set $F\subset
\mathbb{R}_{+}$ by: $\mathcal{O}_{x}\left(  F\right)  =\sup_{t,s\in
F}\left\vert x_{t}-x_{s}\right\vert ;\smallskip$

\noindent$\bullet\quad$the modulus of continuity $\boldsymbol{\upmu}%
_{x}:\mathbb{R}_{+}\rightarrow\mathbb{R}_{+}$ by: $\boldsymbol{\upmu}%
_{x}\left(  \varepsilon\right)  =\sup_{t\in\mathbb{R}_{+}}\mathcal{O}%
_{x}\left(  \left[  t,t+\varepsilon\right]  \right)  .$

\begin{remark}
\label{Remark 1_Annexes}\noindent$1.$ If $x\in\mathbb{D}\left(  \left[
0,T\right]  ;\mathbb{R}^{d}\right)  $, then there exists a sequence of
partitions $\pi_{\varepsilon}\in\mathcal{P}_{\mathbb{R}_{+}}\,,$
$\varepsilon>0,$ with $\left\Vert \pi_{\varepsilon}\right\Vert \rightarrow0$,
as $\varepsilon\rightarrow0$, such that $\max_{r\in\pi_{\varepsilon}%
}\mathcal{O}_{x}\left(  [r,r^{\prime})\right)  <\varepsilon$. In particular
$x$ can be uniformly approximated by simple functions (constant on intervals):%
\[
x_{t}^{\varepsilon}=\sum_{r\in\pi_{\varepsilon}}x_{r}\mathbf{1}_{[r,r^{\prime
})}\left(  t\right)  ,\text{\ }t\geq0,
\]
such that $\left\Vert x^{\varepsilon}-x\right\Vert _{\infty}\leq\varepsilon$.

The function $x$ can also be pointwise approximated by $C^{1}$--functions%
\[
\tilde{x}_{t}^{\varepsilon}=\frac{1}{\varepsilon^{2}}%
{\displaystyle\int_{t}^{t+\varepsilon}}
\Big(%
{\displaystyle\int_{s}^{s+\varepsilon}}
x_{r}dr\Big)ds
\]
such that: $\lim_{\varepsilon\searrow0}\tilde{x}_{t}^{\varepsilon}=x_{t}$, for
all $t\geq0$ and $\left\Vert \tilde{x}^{\varepsilon}\right\Vert _{T}%
\leq\left\Vert x\right\Vert _{T}$, for all $T\geq0.\smallskip$

\noindent$2.$ If $x\in\mathbb{D}\left(  \left[  0,T\right]  ,\mathbb{R}%
^{d}\right)  ,$ then for each $\delta>0$ there exists a finite number of
points $t\in\left[  0,T\right]  $ such that $\left\vert x_{t}-x_{t-}%
\right\vert \geq\delta.\smallskip$

\noindent$3.$ If $x\in\mathbb{D}\left(  \left[  0,T\right]  ,\mathbb{R}%
^{d}\right)  ,$ then $\left\Vert x\right\Vert _{T}<\infty$ and the closure of
$x\left(  \left[  0,T\right]  \right)  $ is compact.$\smallskip$

\noindent$4.$ Let $x^{n},x:\left[  0,T\right]  \rightarrow\mathbb{R}^{d}$ be
such that $\left\Vert x^{n}-x\right\Vert _{T}\rightarrow0$ as $n\rightarrow
\infty.$ If $x^{n}\in\mathbb{D}\left(  \left[  0,T\right]  ;\mathbb{R}%
^{d}\right)  ,$ for all $n\in\mathbb{N}^{\ast},$ then $x\in\mathbb{D}\left(
\left[  0,T\right]  ;\mathbb{R}^{d}\right)  $.
\end{remark}

\subsection{Bounded variation functions\label{Bounded variation functions}}

Let $\left[  a,b\right]  $ be a closed interval from $\mathbb{R}_{+}$ and
$\pi=\left\{  t_{0},t_{1},\ldots,t_{n}\right\}  \in\mathcal{P}_{\left[
a,b\right]  }$ (a partition of $\left[  a,b\right]  $) of the form
$\pi:\left\{  a=t_{0}<t_{1}<\cdots<t_{n}=b\right\}  $. We define the variation
of a function $k:\left[  a,b\right]  \rightarrow\mathbb{R}^{d}$ corresponding
to the partition $\pi\in\mathcal{P}_{\left[  a,b\right]  }$ by $V_{\pi}\left(
k\right)  :=\sum_{i=0}^{n-1}\left\vert k_{t_{i+1}}-k_{t_{i}}\right\vert $ and
the total variation of $k$ on $\left[  a,b\right]  $ by%
\[
\left\updownarrow k\right\updownarrow _{\left[  a,b\right]  }:=\sup_{\pi
\in\mathcal{P}_{\left[  a,b\right]  }}V_{\pi}\left(  k\right)  =\sup
\Big\{\sum_{i=0}^{n_{\pi}-1}\left\vert k_{t_{i+1}}-k_{t_{i}}\right\vert
:\pi\in\mathcal{P}_{\left[  a,b\right]  }\Big\}.
\]
If $\left[  a,b\right]  =\left[  0,T\right]  $ then $\left\updownarrow
k\right\updownarrow _{T}:=\left\updownarrow k\right\updownarrow _{\left[
0,T\right]  }\,.$

\begin{remark}
\label{Remark 3_Annexes}We highlight that for all $t_{0}\in\left[  0,T\right]
$ we have $\left\Vert x\right\Vert _{T}\leq\left\vert x_{t_{0}}\right\vert
+\left\updownarrow x\right\updownarrow _{T}\;.$
\end{remark}

\begin{definition}
\label{Definition 1_Annexes}A function $k:\left[  0,T\right]  \rightarrow
\mathbb{R}^{d}$ has bounded variation on $\left[  0,T\right]  $ if
$\left\updownarrow k\right\updownarrow _{T}<\infty.$ The space of bounded
variation functions on $\left[  0,T\right]  $ will be denoted by \textrm{$BV$%
}$\left(  \left[  0,T\right]  ;\mathbb{R}^{d}\right)  .$ By \textrm{$BV$%
}$_{loc}\left(  \mathbb{R}_{+};\mathbb{R}^{d}\right)  $ we denote the space of
the functions $k:\mathbb{R}_{+}\rightarrow\mathbb{R}^{d},$ such that
$\left\updownarrow k\right\updownarrow _{T}<\infty$ for all $T>0.$
\end{definition}

\begin{proposition}
\label{Proposition 1_Annexes}Let $k\in\mathbb{D}\left(  \left[  0,T\right]
,\mathbb{R}^{d}\right)  $ and $D$ be dense subset of $\left[  0,T\right]  $.
Then, for every sequence of partitions $\overline{\pi}_{N}\in\mathcal{P}%
_{\left[  0,T\right]  }$ such that%
\[%
\begin{array}
[c]{ll}%
\left(  a\right)  \quad & \overline{\pi}_{N}=\big\{t_{0}^{\left(  N\right)
},t_{1}^{\left(  N\right)  },\ldots,t_{j_{N}}^{\left(  N\right)
}\big\}\subset\overline{\pi}_{N+1}\subset D,\medskip\\
\left(  b\right)  \quad & \overline{\pi}_{N}:0=t_{0}^{\left(  N\right)
}<t_{1}^{\left(  N\right)  }<\cdots<t_{j_{N}}^{\left(  N\right)  }%
=T,\medskip\\
\left(  c\right)  \quad & \left\Vert \overline{\pi}_{N}\right\Vert
\rightarrow0,\quad\text{as }N\rightarrow\infty,
\end{array}
\]
it holds%
\[
V_{\overline{\pi}_{N}}\left(  k\right)  \nearrow\left\updownarrow
k\right\updownarrow _{T}\quad\text{as }N\nearrow\infty.
\]

\end{proposition}

\begin{proof}
Clearly, $V_{\overline{\pi}_{N}}\left(  k\right)  $ is increasing with respect
to $N$ and $V_{\overline{\pi}_{N}}\left(  k\right)  \leq\left\updownarrow
k\right\updownarrow _{T}.$ Let $\pi\in\mathcal{P}_{\left[  0,T\right]  }$ be
arbitrary $\pi=\left\{  0=t_{0}<t_{1}<\cdots<t_{n_{\pi}}=T\right\}  $ and let
$\tau_{i}^{N}$ be the minimum of the finite set $\overline{\pi}_{N}\cap\left[
t_{i},T\right]  .$ Then $\lim_{N\rightarrow\infty}\tau_{i}^{N}=t_{i}\,$, since
$|\tau_{i}^{\left(  N\right)  }-t_{i}|\leq\left\Vert \overline{\pi}%
_{N}\right\Vert .$ We have%
\[%
\begin{array}
[c]{l}%
\displaystyle V_{\pi}\left(  k\right)  =\sum\limits_{i=0}^{n_{\pi}%
-1}\left\vert k_{t_{i+1}}-k_{t_{i}}\right\vert \leq\sum\limits_{i=0}^{n_{\pi
}-1}\left[  \big|k_{t_{i+1}}-k_{\tau_{i+1}^{N}}\big|+\big|k_{\tau_{i+1}^{N}%
}-k_{\tau_{i}^{N}}\big|+\big|k_{\tau_{i}^{N}}-k_{t_{i}}\big|\right]
\medskip\\
\displaystyle\leq2\sum\limits_{i=0}^{n_{\pi}}\big|k_{\tau_{i}^{N}}-k_{t_{i}%
}\big|+V_{\overline{\pi}_{N}}\left(  k\right)
\end{array}
\]
and passing to the limit for $N\nearrow\infty$ we obtain $V_{\pi}\left(
k\right)  \leq\lim_{N\nearrow\infty}V_{\overline{\pi}_{N}}\left(  k\right)
\leq\left\updownarrow k\right\updownarrow _{T}\,$, for all $\pi\in
\mathcal{P}_{\left[  a,b\right]  }$. Hence $\lim\limits_{N\nearrow\infty
}V_{\overline{\pi}_{N}}\left(  k\right)  =\left\updownarrow
k\right\updownarrow _{T}\,.$\hfill
\end{proof}

\begin{proposition}
\label{Proposition 2_Annexes}Let $k^{n},k:\left[  0,T\right]  \rightarrow
\mathbb{R}^{d}$, $n\in\mathbb{N}^{\ast}$.

\noindent$\left(  a\right)  $ If $\lim_{n\rightarrow\infty}k_{t}^{n}=k_{t}$
for all $t\in\left[  0,T\right]  ,$ then%
\[
\left\updownarrow k\right\updownarrow _{T}\leq\liminf_{n\rightarrow+\infty
}\left\updownarrow k^{n}\right\updownarrow _{T}\;.
\]
\noindent$\left(  b\right)  $ Let $D$ be a dense subset of $\left[
0,T\right]  $. If $k\in\mathbb{D}\left(  \left[  0,T\right]  ;\mathbb{R}%
^{d}\right)  $ and $\lim_{n\rightarrow\infty}k_{t}^{n}=k_{t}$ for all $t\in
D$, then%
\[
\left\updownarrow k\right\updownarrow _{T}\leq\liminf_{n\rightarrow+\infty
}\left\updownarrow k^{n}\right\updownarrow _{T}\;.
\]

\end{proposition}

\begin{proof}
In both cases, by Proposition \ref{Proposition 1_Annexes}, there exists an
increasing sequence of partitions $\overline{\pi}_{N}\in\mathcal{P}_{\left[
0,T\right]  }$ , $\overline{\pi}_{N}\subset D$ ($D=\left[  0,T\right]  $ in
the first case) such that $V_{\overline{\pi}_{N}}\left(  k\right)
\nearrow\left\updownarrow k\right\updownarrow _{T}$ as $N\nearrow\infty.$
Passing to $\liminf_{n\rightarrow+\infty}$ in $V_{\overline{\pi}_{N}}\left(
k^{n}\right)  \leq\left\updownarrow k^{n}\right\updownarrow _{T}$ we deduce
that%
\[
V_{\overline{\pi}_{N}}\left(  k\right)  \leq\liminf_{n\rightarrow+\infty
}\left\updownarrow k^{n}\right\updownarrow _{T}\;,
\]
for all $N\in\mathbb{N}^{\ast},$ that clearly yields, for $N\rightarrow
\infty,$ the conclusion.\hfill\medskip
\end{proof}

If $x\in\mathbb{D}\left(  \mathbb{R}_{+};\mathbb{R}^{d}\right)  $ and
$k\in\mathbb{D}\left(  \mathbb{R}_{+};\mathbb{R}^{d}\right)  \cap
\mathrm{BV}_{loc}\left(  \mathbb{R}_{+};\mathbb{R}^{d}\right)  $ (see
Definition \ref{Definition 1_Annexes}), we define $\Delta x_{s}=x_{s}-x_{s-}$
and%
\[
\left[  x,k\right]  _{t}:=\sum_{0\leq s\leq t}\left\langle \Delta x_{s},\Delta
k_{s}\right\rangle .
\]
The series is well defined since%
\[
\big|\left[  x,k\right]  _{t}\big|=\Big|\sum_{0\leq s\leq t}\left\langle
\Delta x_{s},\Delta k_{s}\right\rangle \Big|\leq\sum_{0\leq s\leq t}\left\vert
\Delta x_{s}\right\vert \left\vert \Delta k_{s}\right\vert \leq2\big(\sup
_{0\leq s\leq t}\left\vert x_{s}\right\vert \big)\left\updownarrow
k\right\updownarrow _{t}~.
\]
We recall now some results concerning the Lebesgue--Stieltjes integral in the
c\`{a}dl\`{a}g case (for other details we refer the reader to the Annexes from
\cite{ma-ra-sl/14}). If $k\in\mathbb{D}\left(  \mathbb{R}_{+};\mathbb{R}%
^{d}\right)  \cap\mathrm{BV}_{loc}\left(  \mathbb{R}_{+};\mathbb{R}%
^{d}\right)  $ then there exists a unique $\mathbb{R}^{d}$--valued, $\sigma
$--finite measure $\mu_{k}:\mathcal{B}_{\mathbb{R}_{+}}\rightarrow
\mathbb{R}^{d}$ such that $\mu_{k}\left(  (s,t]\right)  =k_{t}-k_{s}$, for all
$0\leq s<t$. The total variation measure is uniquely defined by $\left\vert
\mu_{k}\right\vert \left(  (s,t]\right)  =\left\updownarrow
k\right\updownarrow _{t}-\left\updownarrow k\right\updownarrow _{s}.$

The Lebesgue-Stieltjes integral on $(s,t]$ is given by%
\[
\int_{s}^{t}\left\langle x_{r},dk_{r}\right\rangle :=\int_{(s,t]}\left\langle
x_{r},\mu_{k}\left(  dr\right)  \right\rangle
\]
and is defined for all Borel measurable function such that $\int_{s}%
^{t}\left\vert x_{r}\right\vert d\left\updownarrow k\right\updownarrow
_{r}:=\int_{(s,t]}\left\vert x_{r}\right\vert \left\vert \mu_{k}\right\vert
\left(  dr\right)  <\infty$.

Let $\pi_{n}\in\mathcal{P}_{\mathbb{R}_{+}}$ be a sequence of partitions such
that $||\pi_{n}||\rightarrow0$ as $n\rightarrow\infty.$ Denote%
\[
\left\lfloor t\right\rfloor _{n}=\left\lfloor t\right\rfloor _{\pi_{n}}%
:=\max\left\{  r\in\pi_{n}:r<t\right\}  \quad\text{and}\quad\left\lceil
t\right\rceil _{n}=\left\lceil t\right\rceil _{\pi_{n}}:=\min\left\{  r\in
\pi_{n}:r\geq t\right\}  .
\]
Then, for all $x\in\mathbb{D}\left(  \mathbb{R}_{+};\mathbb{R}^{d}\right)  $
and $t\geq0,$ we have%
\[
x_{\left\lfloor t\right\rfloor _{n}}=\sum_{r\in\pi_{n}}x_{r}\mathbf{1}%
_{(r,r^{\prime}]}\left(  t\right)  \quad\text{and}\quad x_{\left\lceil
t\right\rceil _{n}}=\sum_{r\in\pi_{n}}x_{r^{\prime}}\mathbf{1}_{(r,r^{\prime
}]}\left(  t\right)
\]
and for $n\rightarrow\infty$
\[
x_{\left\lfloor t\right\rfloor _{n}}\rightarrow x_{t-}\quad\text{and}\quad
x_{\left\lceil t\right\rceil _{n}}\rightarrow x_{t}%
\]
By the Lebesgue dominated convergence theorem, as $n\rightarrow\infty,$%
\begin{equation}%
\begin{array}
[c]{cc}%
\left(  a\right)  \; & \sum\limits_{r\in\pi_{n}}\left\langle x_{r\wedge
t},k_{r^{\prime}\wedge t}-k_{r\wedge t}\right\rangle =\int_{0}^{t}\langle
x_{\left\lfloor r\right\rfloor _{n}},dk_{r}\rangle\longrightarrow\int_{0}%
^{t}\left\langle x_{r-},dk_{r}\right\rangle ,\medskip\\
\left(  b\right)  \; & \sum\limits_{r\in\pi_{n}}\left\langle x_{r^{\prime
}\wedge t},k_{r^{\prime}\wedge t}-k_{r\wedge t}\right\rangle =\int_{0}%
^{t}\langle x_{\left\lceil r\right\rceil _{n}\wedge t},dk_{r}\rangle
\longrightarrow\int_{0}^{t}\left\langle x_{r},dk_{r}\right\rangle .
\end{array}
\label{aprox-intr}%
\end{equation}
Let $x\in\mathbb{D}\left(  \mathbb{R}_{+};\mathbb{R}^{d}\right)  $ and
$k,\ell\in\mathbb{D}\left(  \mathbb{R}_{+};\mathbb{R}^{d}\right)
\cap\mathrm{BV}_{loc}\left(  \mathbb{R}_{+};\mathbb{R}^{d}\right)  .$ The
following properties hold:%
\[%
\begin{array}
[c]{ll}%
\left(  a\right)  & \displaystyle\Big|%
{\displaystyle\int_{s}^{t}}
\left\langle x_{r},dk_{r}\right\rangle \Big|\leq%
{\displaystyle\int_{s}^{t}}
\left\vert x_{r}\right\vert d\left\updownarrow k\right\updownarrow _{r}%
\leq\sup\limits_{s<r\leq t}\left\vert x_{r}\right\vert \;\left(
\left\updownarrow k\right\updownarrow _{t}-\left\updownarrow
k\right\updownarrow _{s}\right)  ,\medskip\\
\left(  b\right)  & \displaystyle%
{\displaystyle\int_{\left\{  t\right\}  }}
\left\langle x_{r},dk_{r}\right\rangle =\left\langle x_{t},\Delta
k_{t}\right\rangle \text{ and }%
{\displaystyle\int_{\left\{  0\right\}  }}
\left\langle x_{r},dk_{r}\right\rangle =0,\medskip\\
\left(  c\right)  & \displaystyle%
{\displaystyle\int_{0}^{t}}
\left\langle x_{r},dk_{r}\right\rangle =%
{\displaystyle\int_{0}^{t}}
\left\langle x_{r-},dk_{r}\right\rangle +\left[  x,k\right]  _{t}~,\medskip\\
\left(  d\right)  & \displaystyle%
{\displaystyle\int_{0}^{t}}
\left\langle \ell_{r},dk_{r}\right\rangle +%
{\displaystyle\int_{0}^{t}}
\left\langle k_{r},d\ell_{r}\right\rangle =\left\langle \ell_{t}%
,k_{t}\right\rangle -\left\langle \ell_{0},k_{0}\right\rangle +\left[
\ell,k\right]  _{t}\;.
\end{array}
\]

\begin{proposition}
\label{Proposition 3_Annexes}Let $x\in\mathbb{D}\left(  \mathbb{R}%
_{+};\mathbb{R}^{d}\right)  $, $k\in\mathbb{D}\left(  \mathbb{R}%
_{+};\mathbb{R}^{d}\right)  \cap\mathrm{BV}_{loc}\left(  \mathbb{R}%
_{+};\mathbb{R}^{d}\right)  $. Then $F:\mathbb{R}_{+}\rightarrow\mathbb{R}$
defined by $F\left(  0\right)  =0$ and $F\left(  t\right)  :=\displaystyle%
{\displaystyle\int_{0}^{t}}
\left\langle x_{r},dk_{r}\right\rangle $, if $t>0$, is a c\`{a}dl\`{a}g
function on $\mathbb{R}_{+}\;.$
\end{proposition}

\begin{proof}
Let $\varepsilon>0$ and $x^{\varepsilon}$ the $C^{1}$--function given by
Remark \ref{Remark 1_Annexes}. The function $F^{\varepsilon}:\mathbb{R}%
_{+}\rightarrow\mathbb{R},$ $F^{\varepsilon}\left(  0\right)  =0,$
\[
F^{\varepsilon}\left(  t\right)  :=%
{\displaystyle\int_{0}^{t}}
\left\langle x_{r}^{\varepsilon},dk_{r}\right\rangle =\tilde{x}_{t}%
^{\varepsilon}k_{t}-\tilde{x}_{0}^{\varepsilon}k_{0}-%
{\displaystyle\int_{0}^{t}}
\big\langle k_{r},\frac{d}{dr}\tilde{x}_{r}^{\varepsilon}\big\rangle dr,\text{
for }t>0
\]
is a c\`{a}dl\`{a}g function on $\mathbb{R}_{+}\,.$ Let $0\leq\tau<s\leq T$.
We have%
\[
\left\vert F\left(  s\right)  -F\left(  \tau\right)  \right\vert \leq\Big|%
{\displaystyle\int_{\tau}^{s}}
\left\langle x_{r}-x_{r}^{\varepsilon},dk_{r}\right\rangle \Big|+\left\vert
F^{\varepsilon}\left(  s\right)  -F^{\varepsilon}\left(  \tau\right)
\right\vert \leq%
{\displaystyle\int_{0}^{T}}
\left\vert x_{r}-x_{r}^{\varepsilon}\right\vert d\left\updownarrow
k\right\updownarrow _{r}+\left\vert F^{\varepsilon}\left(  s\right)
-F^{\varepsilon}\left(  \tau\right)  \right\vert .
\]
Then, for $0\leq t<t+\delta\leq T,$%
\[
\limsup_{\delta\searrow0}\left\vert F\left(  t+\delta\right)  -F\left(
t\right)  \right\vert \leq%
{\displaystyle\int_{0}^{T}}
\left\vert x_{r}-x_{r}^{\varepsilon}\right\vert d\left\updownarrow
k\right\updownarrow _{r},\;\text{for all }\varepsilon>0,
\]
which yields, thanks to the Lebesgue dominated convergence theorem as
$\varepsilon\searrow0$,%
\[
\lim\limits_{\delta\searrow0}\left\vert F\left(  t+\delta\right)  -F\left(
t\right)  \right\vert =0.
\]
Also if $t_{n}$ $\nearrow t$ and $s_{n}\nearrow t$ we have%
\[
\limsup_{n\rightarrow\infty}\left\vert F\left(  t_{n}\right)  -F\left(
s_{n}\right)  \right\vert \leq%
{\displaystyle\int_{0}^{T}}
\left\vert x_{r}-x_{r}^{\varepsilon}\right\vert d~\left\updownarrow
k\right\updownarrow _{r}\rightarrow0,\;\text{as }\varepsilon\rightarrow0.
\]
Consequently there exists $F\left(  t-\right)  .$\hfill
\end{proof}

\subsection{The \textrm{S}--topology}

The space of c\`{a}dl\`{a}g functions $\mathbb{D}\left(  \left[  0,T\right]
;\mathbb{R}^{d}\right)  $ is usually endowed with the Skorohod
topo{\footnotesize \-}lo{\footnotesize \-}gy induced by some metrics. We
present here the \textrm{S}--topology defined on $\mathbb{D}\left(  \left[
0,T\right]  ;\mathbb{R}^{d}\right)  $ and introduced by Jakubowski in
\cite{ja/97}. This topology is weaker than the Skorohod topology, but with
respect to this, we have the continuity of the application $\mathbb{D}\ni
y\longmapsto\int_{0}^{s}g\left(  r,y\left(  r\right)  \right)  dA_{r}$, where
$h$ is a continuous function and $A$ is a continuous non--decreasing function.

Moreover, we mention that this new \textrm{S}--topology, although it cannot be
metricized, shares many useful properties with the traditional Skorohod's
topology, e.g. both the direct and the converse Prohorov's theorems are valid,
the Skorohod representation for subsequences exists and finite dimensional
convergence outside a countable set holds.

Let $\mathcal{V}_{\left[  0,T\right]  }^{+}\subset\mathbb{D}\left(  \left[
0,T\right]  ;\mathbb{R}\right)  $ denote the space of non--negative and
non--decreasing functions $k\in\mathbb{D}\left(  \left[  0,T\right]
;\mathbb{R}\right)  $ and therefore, $\mathrm{BV}\left(  \left[  0,T\right]
;\mathbb{R}\right)  =\mathcal{V}_{\left[  0,T\right]  }^{+}-\mathcal{V}%
_{\left[  0,T\right]  }^{+}$ is the space of bounded variation functions.

The \textrm{S}--topology is a sequential topology defined by:

\begin{definition}
\label{Definition 2_Annexes}A sequence $\left\{  x^{n}:n\in\mathbb{N}^{\ast
}\right\}  \subset\mathbb{D}\left(  \left[  0,T\right]  ;\mathbb{R}%
^{d}\right)  $ is $S$--convergent to $x\in\mathbb{D}\left(  \left[
0,T\right]  ;\mathbb{R}^{d}\right)  $ (denoted by $x^{n}%
\xrightarrow[\mathrm{S}]{\;\;\;\; \;\;}x$) if for every $\varepsilon>0$, there
exists a sequence%
\[
\left\{  v^{\varepsilon},v^{n,\varepsilon}:n\in\mathbb{N}^{\ast}\right\}
\subset\mathrm{BV}(\left[  0,T\right]  ;\mathbb{R}^{d})\equiv\big(\mathcal{V}%
_{\left[  0,T\right]  }^{+}-\mathcal{V}_{\left[  0,T\right]  }^{+}\big)^{d}%
\]
such that:

\noindent$\left(  a\right)  $\ The elements $v^{\varepsilon},v^{n,\varepsilon
}$ are $\varepsilon$--uniformly close to $x$ and, respectively, to $x^{n}$,
i.e.%
\[
\left\Vert x-v^{\varepsilon}\right\Vert _{T}\leq\varepsilon\quad
\text{and}\quad\left\Vert x^{n}-v^{n,\varepsilon}\right\Vert _{T}%
\leq\varepsilon,\;\text{for all }n\in\mathbb{N}^{\ast};
\]
\noindent$\left(  b\right)  $\ $\lim\limits_{n\rightarrow\infty}$
$v_{T}^{n,\varepsilon}=v_{T}^{\varepsilon},$ for all $\varepsilon
>0;\smallskip$

\noindent$\left(  c\right)  $\ $v^{n,\varepsilon}$ is weakly--$\ast$
convergent to $v^{\varepsilon}$, i.e. for every continuous function $g:\left[
0,T\right]  \rightarrow\mathbb{R}^{d}$,%
\begin{equation}
\lim_{n\rightarrow\infty}\int_{0}^{T}\left\langle g_{t},dv_{t}^{n,\varepsilon
}\right\rangle =\int_{0}^{T}\left\langle g_{t},dv_{t}^{\varepsilon
}\right\rangle . \label{weak star conv}%
\end{equation}

\end{definition}

\begin{remark}
\label{Remark 2_Annexes}\ \

\noindent1. By the Banach--Steinhaus theorem (uniform boundedness principle)
and Remark \ref{Remark 3_Annexes} we have%
\[
\left\Vert v^{\varepsilon}\right\Vert _{T}+\left\updownarrow v^{\varepsilon
}\right\updownarrow _{T}+\sup_{n\in\mathbb{N}^{\ast}}\left\Vert
v^{n,\varepsilon}\right\Vert _{T}+\sup_{n\in\mathbb{N}^{\ast}}%
\left\updownarrow v^{n,\varepsilon}\right\updownarrow _{T}\leq M_{\varepsilon
}<\infty
\]
and, consequently, there exists $M\geq0$ such that%
\[
\left\Vert x\right\Vert _{T}+\sup_{n\in\mathbb{N}^{\ast}}\left\Vert
x^{n}\right\Vert _{T}\leq M<\infty.
\]

\noindent2. From (\ref{weak star conv}) it follows that for every continuous
function $G:\left[  0,T\right]  \rightarrow\mathbb{R}^{d^{\prime}\times d}:$%
\[
\lim_{n\rightarrow\infty}\int_{0}^{T}G_{t}dv_{t}^{n,\varepsilon}=\int_{0}%
^{T}G_{t}dv_{t}^{\varepsilon}.
\]
In particular for $G=I_{d\times d}$ we have (using the property $\left(
b\right)  $ from Definition \ref{Definition 2_Annexes}):%
\[
\lim_{n\rightarrow\infty}v_{0}^{n,\varepsilon}=v_{0}^{\varepsilon}\;;
\]
\noindent3. Since for all continuous functions $g,\alpha\in C\left(  \left[
0,T\right]  ;\mathbb{R}\right)  $, $g\geq0$, $\left\Vert \alpha\right\Vert
_{T}=1$ and for each coordinate $v^{n,\varepsilon;i}$ , $i=\overline{1,d}$ of
the vector $v^{n,\varepsilon}$,%
\[%
{\displaystyle\int_{0}^{T}}
g_{t}d\left\updownarrow v^{\varepsilon;i}\right\updownarrow _{t}=\sup_{\alpha
}\int_{0}^{T}g_{t}\alpha_{t}d\,v_{t}^{\varepsilon;i}=\sup_{\alpha}\left[
\lim_{n\rightarrow\infty}\int_{0}^{T}g_{t}\alpha_{t}d\,v_{t}^{n,\varepsilon
;i}\right]  \leq\liminf\limits_{n\rightarrow+\infty}%
{\displaystyle\int_{0}^{T}}
g_{t}d\left\updownarrow v^{n,\varepsilon;i}\right\updownarrow _{t}\,,
\]
we deduce that the increasing functions $\left\updownarrow v^{n,\varepsilon
;i}\right\updownarrow $ and $\ell^{n,\varepsilon;i}=v^{n,\varepsilon
}-\left\updownarrow v^{n,\varepsilon}\right\updownarrow -v_{0}^{n,\varepsilon
}$ are weakly--$\ast$ convergent. Therefore, by equivalence between the
weak--$\ast$ convergence and the convergence of cumulative distribution
functions, we infer the pointwise convergence outside a countable set
$Q_{\varepsilon;i}\subset(0,T)$ ($Q_{\varepsilon;i}$ is the set of the
discontinuity points of $v^{\varepsilon;i}$). Hence, if we denote
$Q_{\varepsilon}=\cup_{i=\overline{1,d}}Q_{\varepsilon;i}$ , then%
\begin{equation}
\lim_{n\rightarrow\infty}v_{t}^{n,\varepsilon}=v_{t}^{\varepsilon},\;\text{for
all }t\in\left[  0,T\right]  \setminus Q_{\varepsilon}. \label{point converg}%
\end{equation}
Additionally, we easily deduce%
\[
\lim_{n\rightarrow\infty}x_{t}^{n}=x_{t},\;\text{for all }t\in\left[
0,T\right]  \setminus\left(  \cup_{m=1}^{\infty}Q_{1/m}\right)  .
\]

\noindent4. Therefore, based on the Lebesgue dominated convergence theorem, we
deduce immediately the following result (see also \cite[Corollary 2.11]{ja/97}):

Let $\mu$ be an atomless positive measure on $\left[  0,T\right]  $ and
$f:\left[  0,T\right]  \times\mathbb{R}^{m}\times\mathbb{R}^{d}%
\mathbb{\rightarrow R}^{d^{\prime}}$ be a Carath\'{e}odory function (i.e.
$f\left(  \cdot,x,y\right)  $ is Lebesgue measurable for every $x,y$, and
$f\left(  t,\cdot,\cdot\right)  $ is a continuous function $dt$--a.e.
$t\in\left[  0,T\right]  $) such that for every $R>0$, $%
{\displaystyle\int_{0}^{T}}
f_{R}^{\#}\left(  t\right)  d\mu\left(  t\right)  <\infty$, where $f_{R}%
^{\#}\left(  t\right)  :=\sup_{\left\vert x\right\vert +\left\vert
y\right\vert \leq R}\left\vert f\left(  t,x,y\right)  \right\vert .$ If
$\left\{  x,x^{n}:n\in\mathbb{N}^{\ast}\right\}  \subset\mathcal{C}\left(
\left[  0,T\right]  ;\mathbb{R}^{m}\right)  $ and $\left\{  y,y^{n}%
:n\in\mathbb{N}^{\ast}\right\}  \subset\mathbb{D}\left(  \left[  0,T\right]
;\mathbb{R}^{d}\right)  $ such that $\left\Vert x^{n}-x\right\Vert
_{T}\rightarrow0$ as $n\rightarrow\infty$, and $y^{n}%
\xrightarrow[\mathrm{S}]{\;\;\;\; \;\;}y,$ then%
\[
\lim_{n\rightarrow\infty}%
{\displaystyle\int_{0}^{T}}
f\left(  t,x_{t}^{n},y_{t}^{n}\right)  d\mu\left(  t\right)  =%
{\displaystyle\int_{0}^{T}}
f\left(  t,x_{t},y_{t}\right)  d\mu\left(  t\right)  .
\]

\noindent5. Recall from \cite[Theorem 2.13]{ja/97} that the $\sigma$--algebra
generated by the \textrm{S}--topology is the $\sigma$--algebra generated by
the projections $\left\{  \pi_{t}\right\}  _{t\in\left[  0,T\right]  }\,$,
where $\pi_{t}\left(  x\right)  :=x_{t}\,.$
\end{remark}

At the end of this part we present a sufficient condition of $\mathrm{S}%
$--tightness together with the Prohorov and Skorohod representation theorems.
This result follows from \cite[Appendix A]{le/02} and \cite[Theorem 3.4 \&
Definition 3.3]{ja/97}.

\begin{theorem}
\label{Theorem 2_Annexes}Let $(\Omega,\mathcal{F},\mathbb{P})$ be a complete
probability space, $\left\{  \mathcal{F}_{t}\right\}  _{t\geq0}$ be a
filtration and $L^{n}:\Omega\rightarrow\mathbb{D}\left(  \left[  0,T\right]
;\mathbb{R}^{d}\right)  $, $n\in\mathbb{N}^{\ast}$ be a stochastic process
such that $\mathbb{E}\left\vert L_{t}^{n}\right\vert <\infty$ for all
$n\in\mathbb{N}^{\ast}$ and $t$. We define the conditional variation%
\begin{equation}
\mathrm{CV}_{T}(L^{n}):=\sup_{\pi}{\sum_{i=0}^{N-1}{\mathbb{E}}}%
\Big[\big|\mathbb{E}^{{{\mathcal{F}}_{t_{i}}}}{[L_{t_{i+1}}^{n}-L_{t_{i}}%
^{n}]{\big|}\Big]}\label{def cond var}%
\end{equation}
with the supremum taken over all partitions $\pi:t=t_{0}<t_{1}<\cdots
<t_{N}=T.$

If%
\[
\sup_{n\in\mathbb{N}^{\ast}}\Big(\mathrm{CV}_{T}\left(  L^{n}\right)
+\mathbb{E}\big(\sup_{s\in\lbrack0,T]}\left\vert L_{s}^{n}\right\vert
\big)\Big)<\infty,
\]
then the sequence $\left(  L^{n}\right)  _{n\in\mathbb{N}^{\ast}}$ is
$\mathrm{S}$--tight.

Therefore,

$\left(  i\right)  $ for all subsequence $\left(  L^{n_{k}}\right)
_{k\in\mathbb{N}^{\ast}}$, there exists a sub--subsequence $\left(
L^{n_{k_{l}}}\right)  _{l\in\mathbb{N}^{\ast}}$ and stochastic processes
$\left(  \bar{L}^{l}\right)  _{l\in\mathbb{N}^{\ast}}$ and $\bar{L}$ defined
on $\left(  \left[  0,1\right]  ,\mathcal{B}_{\left[  0,1\right]  }%
,d\lambda\right)  $, where $d\lambda$ denotes the Lebesgue measure, such that,%
\[
L^{n_{k_{l}}}\sim\bar{L}^{l},\;\text{for all }l\in\mathbb{N}^{\ast}%
\]
and%
\[
\text{for all }\omega\in\left[  0,1\right]  ,\;\bar{L}^{l}\left(
\omega\right)  \xrightarrow[\mathrm{S}]{\;\;\;\;\;\;\;\;\;}\bar{L}\left(
\omega\right)  \text{, as }n\rightarrow\infty,
\]
where $\sim$ denotes the equality in law of two stochastic processes;

$\left(  ii\right)  $ moreover, there exist a countable set $Q\subset
\lbrack0,T)$ such that, for any $\left\{  t_{1},t_{2},\ldots t_{m}\right\}
\subset\left[  0,T\right]  \setminus Q,$ as $\mathbb{R}^{d\times m}$--valued
random variable,%
\[
(L_{t_{1}}^{n_{k_{l}}},L_{t_{2}}^{n_{k_{l}}},\ldots,L_{t_{m}}^{n_{k_{l}}}%
)\sim(\bar{L}_{t_{1}}^{l},\bar{L}_{t_{2}}^{l},\ldots,\bar{L}_{t_{m}}%
^{l}),\quad\text{for all }l\in\mathbb{N}^{\ast}%
\]
and%
\[
(L_{t_{1}}^{n_{k_{l}}},L_{t_{2}}^{n_{k_{l}}},\ldots,L_{t_{m}}^{n_{k_{l}}%
})\xrightarrow[\lambda-\mathrm{a.s.}]{\;\;\;\;\;\;\;\;\;}(\bar{L}_{t_{1}}%
,\bar{L}_{t_{2}},\ldots,\bar{L}_{t_{m}}),\quad\text{as }l\rightarrow\infty.
\]

\end{theorem}

\subsection{The Helly--Bray theorem}

Very important ingredients for the proof of our main result are the following
Helly--Bray type theorems.

\begin{theorem}
[Helly-Bray]\label{Theorem 1_Annexes}Let $\left\{  k,k^{n}:n\in\mathbb{N}%
^{\ast}\right\}  \subset\mathrm{BV}\left(  \left[  0,T\right]  ;\mathbb{R}%
^{d}\right)  $ be such that%
\[
\left\updownarrow k\right\updownarrow _{T}+\sup\limits_{n\in\mathbb{N}^{\ast}%
}\left\updownarrow k^{n}\right\updownarrow _{T}=M<+\infty.
\]
\noindent$\mathrm{(I)}$ If $\left\{  x,x^{n}:n\in\mathbb{N}^{\ast}\right\}
\subset C\left(  \left[  0,T\right]  ;\mathbb{R}^{d}\right)  $ and $Q$ is a
subset of $[0,T)$, negligible with respect to the Lebesgue measure $dt$, such
that%
\begin{equation}%
\begin{array}
[c]{rl}%
\left(  i\right)  & \lim\limits_{n\rightarrow\infty}\left\Vert x^{n}%
-x\right\Vert _{T}=0,\medskip\\
\left(  ii\right)  & \lim\limits_{n\rightarrow\infty}k_{t}^{n}=k_{t},\text{
for all }t\in\left[  0,T\right]  \setminus Q,
\end{array}
\label{HB0}%
\end{equation}
then%
\begin{equation}
\lim\limits_{n\rightarrow\infty}%
{\displaystyle\int_{s}^{t}}
\left\langle x_{r}^{n},dk_{r}^{n}\right\rangle =%
{\displaystyle\int_{s}^{t}}
\left\langle x_{r},dk_{r}\right\rangle ,\;\text{for all }s,t\in\left[
0,T\right]  \setminus Q\text{, }s\leq t. \label{HB1}%
\end{equation}
\noindent$\mathrm{(II)}$ If $\left\{  x,x^{n}:n\in\mathbb{N}^{\ast}\right\}
\subset\mathbb{D}\left(  \left[  0,T\right]  ;\mathbb{R}^{d}\right)  $ and
$\left\{  k,k^{n}:n\in\mathbb{N}^{\ast}\right\}  \subset\mathbb{D}\left(
\left[  0,T\right]  ;\mathbb{R}^{d}\right)  $ such that%
\begin{equation}%
\begin{array}
[c]{rl}%
\left(  i\right)  & \lim\limits_{n\rightarrow\infty}\left\Vert x^{n}%
-x\right\Vert _{T}=0,\medskip\\
\left(  ii\right)  & \lim\limits_{n\rightarrow\infty}\left\Vert k^{n}%
-k\right\Vert _{T}=0,
\end{array}
\label{HB2}%
\end{equation}
then%
\begin{equation}
\lim\limits_{n\rightarrow\infty}%
{\displaystyle\int_{s}^{t}}
\left\langle x_{r}^{n},dk_{r}^{n}\right\rangle =%
{\displaystyle\int_{s}^{t}}
\left\langle x_{r},dk_{r}\right\rangle ,\;\text{for all }s,t\in\left[
0,T\right]  . \label{HB3}%
\end{equation}
\noindent$\mathrm{(III)}$ If $\left\{  x,x^{n}:n\in\mathbb{N}^{\ast}\right\}
\subset\mathbb{D}\left(  \left[  0,T\right]  ;\mathbb{R}^{d}\right)  $ and
$\left\{  k,k^{n}:n\in\mathbb{N}^{\ast}\right\}  \subset C\left(  \left[
0,T\right]  ,\mathbb{R}^{d}\right)  $ such that%
\begin{equation}%
\begin{array}
[c]{rl}%
\left(  i\right)  & x^{n}\xrightarrow[\mathrm{S}]{\;\;\;\; \;\;}x,\medskip\\
\left(  ii\right)  & \lim\limits_{n\rightarrow\infty}\left\Vert k^{n}%
-k\right\Vert _{T}=0,
\end{array}
\label{HB4}%
\end{equation}
then there exists a countable subset $Q\subset\lbrack0,T)$ such that%
\begin{equation}
\lim\limits_{n\rightarrow\infty}%
{\displaystyle\int_{s}^{t}}
\left\langle x_{r}^{n},dk_{r}^{n}\right\rangle =%
{\displaystyle\int_{s}^{t}}
\left\langle x_{r},dk_{r}\right\rangle ,\;\text{for all }s,t\in\left(  \left[
0,T\right]  \setminus Q\right)  \cup\left\{  r:k_{r}=0\right\}  \text{, }s\leq
t. \label{HB5}%
\end{equation}
In all cases%
\begin{equation}%
{\displaystyle\int_{0}^{T}}
\left\vert x_{r}\right\vert d\left\updownarrow k\right\updownarrow _{r}%
\leq\liminf\limits_{n\rightarrow+\infty}%
{\displaystyle\int_{0}^{T}}
\left\vert x_{r}^{n}\right\vert d\left\updownarrow k^{n}\right\updownarrow
_{r} \label{HB6}%
\end{equation}

\end{theorem}

\begin{proof}
\noindent$\mathrm{(I)}$ Let $\varepsilon>0$ and $\tilde{x}^{\varepsilon}$ be
the $C^{1}$--function as in Remark \ref{Remark 1_Annexes}. Since $x\in
C\left(  \left[  0,T\right]  ,\mathbb{R}^{d}\right)  $, $\left\Vert
x-\tilde{x}^{\varepsilon}\right\Vert _{T}<\boldsymbol{\upmu}_{x}\left(
2\varepsilon\right)  $, where $\boldsymbol{\upmu}_{x}$ is the modulus of
continuity of $x$ on $\left[  0,T\right]  .$ We have%
\[%
\begin{array}
[c]{l}%
\displaystyle\Big|\int_{s}^{t}\left\langle x_{r}^{n},dk_{r}^{n}\right\rangle
-\int_{s}^{t}\left\langle x_{r},dk_{r}\right\rangle \Big|\medskip\\
\displaystyle\leq\Big|\int_{s}^{t}\left\langle x_{r}^{n}-x_{r},dk_{r}%
^{n}\right\rangle \Big|+\Big|\int_{s}^{t}\left\langle x_{r}-\tilde{x}%
_{r}^{\varepsilon},dk_{r}^{n}-dk_{r}\right\rangle \Big|+\Big|\int_{s}%
^{t}\left\langle \tilde{x}_{r}^{\varepsilon},dk_{r}^{n}-dk_{r}\right\rangle
\Big|\medskip\\
\displaystyle\leq\left\Vert x^{n}-x\right\Vert _{T}~M+\boldsymbol{\upmu}%
_{x}\left(  2\varepsilon\right)  M+\Big|\left\langle \tilde{x}_{t}%
^{\varepsilon},k_{t}^{n}-k_{t}\right\rangle -\left\langle \tilde{x}%
_{s}^{\varepsilon},k_{s}^{n}-k_{s}\right\rangle -%
{\displaystyle\int_{s}^{t}}
\big\langle k_{r}^{n}-k_{r},\frac{d}{dr}\tilde{x}_{r}^{\varepsilon
}\big\rangle dr\Big|.
\end{array}
\]
Since (see Remark \ref{Remark 3_Annexes}) $\sup_{n\in\mathbb{N}^{\ast}%
}\left\Vert k^{n}-k\right\Vert _{T}<\infty$, by Lebesgue dominated convergence
theorem, we have%
\[
\underset{n\rightarrow\infty}{\overline{\lim}}\left[  \Big|\int_{s}%
^{t}\left\langle x_{r}^{n},dk_{r}^{n}\right\rangle -\int_{s}^{t}\left\langle
x_{r},dk_{r}\right\rangle \Big|\right]  \leq\boldsymbol{\upmu}_{x}\left(
\varepsilon\right)  M,\quad\text{for all }\varepsilon>0,
\]
and (\ref{HB1}) clearly follows.$\smallskip$

\noindent$\mathrm{(II)}$ Let $\varepsilon>0$ and $x^{\varepsilon}$ be the
simple function as in Remark \ref{Remark 1_Annexes}. We have%
\begin{equation}%
\begin{array}
[c]{l}%
\displaystyle\Big|\int_{s}^{t}\left\langle x_{r}^{n},dk_{r}^{n}\right\rangle
-\int_{s}^{t}\left\langle x_{r},dk_{r}\right\rangle \Big|\medskip\\
\displaystyle\leq\Big|\int_{s}^{t}\left\langle x_{r}^{n}-x_{r},dk_{r}%
^{n}\right\rangle \Big|+\Big|\int_{s}^{t}\left\langle x_{r}-x_{r}%
^{\varepsilon},dk_{r}^{n}-dk_{r}\right\rangle \Big|+\Big|\int_{s}%
^{t}\left\langle x_{r}^{\varepsilon},dk_{r}^{n}-dk_{r}\right\rangle
\Big|\medskip\\
\displaystyle\leq\left\Vert x^{n}-x\right\Vert _{T}~M+\varepsilon
M+5\left\Vert x\right\Vert _{T}\left\Vert k^{n}-k\right\Vert _{T}%
\,\mathrm{card}\left\{  r\in\pi_{\varepsilon}:s\leq r\leq t\right\}  .
\end{array}
\label{HB7}%
\end{equation}
and (\ref{HB3}) follows.$\smallskip$

\noindent$\mathrm{(III)}$ Let $\varepsilon>0$ and $\left\{  v^{n,\varepsilon
},v^{\varepsilon}:n\in\mathbb{N}^{\ast}\right\}  $ be the sequence of bounded
variation functions associated to the sequence $\left\{  x^{n},x:n\in
\mathbb{N}^{\ast}\right\}  $ by Definition \ref{Definition 2_Annexes}. We
have, for $n\in\mathbb{N}^{\ast},$%
\[%
\begin{array}
[c]{l}%
\Big|%
{\displaystyle\int_{s}^{t}}
\left\langle x_{r}^{n},dk_{r}^{n}\right\rangle -%
{\displaystyle\int_{s}^{t}}
\left\langle x_{r},dk_{r}\right\rangle \Big|\leq\Big|%
{\displaystyle\int_{s}^{t}}
\left\langle x_{r}^{n}-v_{r}^{n,\varepsilon},dk_{r}^{n}\right\rangle
\Big|+\Big|%
{\displaystyle\int_{s}^{t}}
\left\langle v_{r}^{n,\varepsilon}-v_{r}^{\varepsilon},dk_{r}^{n}\right\rangle
\Big|\medskip\\
\quad+\Big|%
{\displaystyle\int_{s}^{t}}
\left\langle v_{r}^{\varepsilon}-x_{r},dk_{r}^{n}\right\rangle \Big|+\Big|%
{\displaystyle\int_{s}^{t}}
\left\langle x_{r},dk_{r}^{n}\right\rangle -%
{\displaystyle\int_{s}^{t}}
\left\langle x_{r},dk_{r}\right\rangle \Big|\medskip\\
\leq2\varepsilon M+\Big|\left\langle v_{t}^{n,\varepsilon}-v_{t}^{\varepsilon
},k_{t}^{n}\right\rangle -\left\langle v_{s}^{n,\varepsilon}-v_{s}%
^{\varepsilon},k_{s}^{n}\right\rangle -%
{\displaystyle\int_{s}^{t}}
\left\langle k_{r}^{n},d\left(  v_{r}^{n,\varepsilon}-v_{r}^{\varepsilon
}\right)  \right\rangle \Big|\medskip\\
\quad+\Big|%
{\displaystyle\int_{s}^{t}}
\left\langle x_{r},dk_{r}^{n}\right\rangle -%
{\displaystyle\int_{s}^{t}}
\left\langle x_{r},dk_{r}\right\rangle \Big|.
\end{array}
\]
By (\ref{HB1}) and (\ref{point converg}) we have%
\[
\limsup_{n\rightarrow\infty}\Big|%
{\displaystyle\int_{s}^{t}}
\left\langle x_{r}^{n},dk_{r}^{n}\right\rangle -%
{\displaystyle\int_{s}^{t}}
\left\langle x_{r},dk_{r}\right\rangle \Big|\leq2\varepsilon M,
\]
for all $s,t\in\left(  \left[  0,T\right]  \setminus Q_{\varepsilon}\right)
\cup\left\{  r:k_{r}=0\right\}  .$ Setting $\varepsilon=\frac{1}{m}%
\rightarrow0$ and $Q=\bigcup_{m\in\mathbb{N}^{\ast}}Q_{1/m}~,$ (\ref{HB5}) follows.

Let $\alpha\in C\left(  \left[  0,T\right]  ;\mathbb{R}^{d}\right)  ,$
$\left\Vert \alpha\right\Vert _{T}\leq1.$ Then by (\ref{HB1})%
\[%
{\displaystyle\int_{0}^{T}}
\left\vert x_{r}\right\vert \left\langle \alpha_{r},dk_{r}\right\rangle
=\lim_{n\rightarrow\infty}%
{\displaystyle\int_{0}^{T}}
\left\vert x_{r}^{n}\right\vert \left\langle \alpha_{r},dk_{r}^{n}%
\right\rangle \leq\liminf_{n\rightarrow+\infty}%
{\displaystyle\int_{0}^{T}}
\left\vert x_{r}^{n}\right\vert d\left\updownarrow k^{n}\right\updownarrow
_{r}%
\]
and passing to $\sup_{\left\Vert \alpha\right\Vert _{T}\leq1}$ we get
(\ref{HB6}).\hfill
\end{proof}

\begin{remark}
The part $\mathrm{(III)}$ of the previous theorem represents a version of
Lemma 3.3 from \cite{bo-ca/04}.
\end{remark}

Finally, we give

\begin{proposition}
\label{Proposition 4_Annexes}Let $f:\left[  0,T\right]  \times\mathbb{R}%
^{d}\mathbb{\rightarrow R}^{d^{\prime}}$ be a locally Lipschitz function. If
$\left\{  x,x^{n}:n\in\mathbb{N}^{\ast}\right\}  \subset\mathbb{D}\left(
\left[  0,T\right]  ;\mathbb{R}^{d}\right)  $ and $x^{n}%
\xrightarrow[\mathrm{S}]{\;\;\;\;
\;\;}x,$ then%
\[
f\left(  \cdot,x_{\cdot}^{n}\right)
\xrightarrow[\mathrm{S}]{\;\;\;\; \;\;}f\left(  \cdot,x_{\cdot}\right)  .
\]

\end{proposition}

\begin{proof}
Let $\varepsilon>0$ and $\left\{  v^{\varepsilon},v^{n,\varepsilon}%
:n\in\mathbb{N}^{\ast}\right\}  $ be the sequence of bounded variation
functions associated to the sequence $\left\{  x^{n}:n\in\mathbb{N}^{\ast
}\right\}  $ by Definition \ref{Definition 2_Annexes}. Let $R\geq1+\left\Vert
x\right\Vert _{T}+\sup_{n\in\mathbb{N}^{\ast}}\left\Vert x^{n}\right\Vert
_{T}$ and $\alpha_{R}\in C^{\infty}\left(  \mathbb{R}^{d}\right)  $ such that
$0\leq\alpha_{r}\leq1$, $\alpha_{R}\left(  x\right)  =1$, if $\left\vert
x\right\vert \leq R$ and $\alpha_{R}\left(  x\right)  =0$, if $\left\vert
x\right\vert \geq R+1.$

We define $f_{R}\left(  t,x\right)  =\alpha_{R}\left(  x\right)  f\left(
t,x\right)  $ which is Lipschitz continuous on $\left[  0,T\right]
\times\mathbb{R}^{d}$ with Lipschitz constant $L_{R}.$ Then $\left\{
f_{R}\left(  \cdot,v_{\cdot}^{n,\varepsilon}\right)  :n\in\mathbb{N}\right\}
$ clearly satisfies:$\smallskip$

\noindent$\bullet\quad$by the Lipschitz property of the function $f_{R}$, we
have $\sup_{n\in\mathbb{N}}\left\updownarrow f_{R}\left(  \cdot,v_{\cdot
}^{n,\varepsilon}\right)  \right\updownarrow _{T}<\infty;\smallskip$

\noindent$\bullet\quad\left\Vert f_{R}\left(  \cdot,x_{\cdot}^{n}\right)
-f_{R}\left(  \cdot,v_{\cdot}^{n,\varepsilon}\right)  \right\Vert _{T}\leq
L_{R}\left\Vert x_{\cdot}^{n}-v_{\cdot}^{n,\varepsilon}\right\Vert _{T}%
\leq\varepsilon L_{R}\,;\smallskip$

\noindent$\bullet\quad\lim_{n\rightarrow\infty}$ $f_{R}\left(  t,v_{t}%
^{n,\varepsilon}\right)  =f_{R}\left(  t,v_{t}^{\varepsilon}\right)  ,$ for
all $t\in\left[  0,T\right]  \setminus Q_{\varepsilon}$, $Q_{\varepsilon}$ is
a countable set from $[0,T);\smallskip$

\noindent$\bullet\quad$from (\ref{HB1}) we see that $\displaystyle\lim
_{n\rightarrow\infty}\int_{0}^{T}\left\langle g_{t},df_{R}\left(
t,v_{t}^{n,\varepsilon}\right)  \right\rangle \longrightarrow\int_{0}%
^{T}\left\langle g_{t},df_{R}\left(  t,v_{t}^{\varepsilon}\right)
\right\rangle $ for all $g\in\mathbb{D(}\left[  0,T\right]  ;\mathbb{R}^{d})$.

Hence $f\left(  \cdot,x_{\cdot}^{n}\right)  =f_{R}\left(  \cdot,x_{\cdot}%
^{n}\right)  \xrightarrow[\mathrm{S}]{\;\;\;\;
\;\;}f_{R}\left(  \cdot,x_{\cdot}\right)  =f\left(  \cdot,x_{\cdot}\right)
.$\hfill
\end{proof}

\subsection{Convex functions\label{Convex functions}}

Let $\varphi:\mathbb{R}^{d}\rightarrow(-\infty,+\infty]$ be a proper convex
lower semicontinuous (l.s.c.) function. Define

\begin{itemize}
\item the domain of $\varphi:$%
\[
\mathrm{Dom}\left(  \varphi\right)  :=\{x\in\mathbb{R}^{d}:\varphi
(x)<+\infty\},
\]

\item the subdifferential of the function $\varphi$ at $x:$%
\[
\partial\varphi(x):=\{\widehat{x}\in\mathbb{R}^{d}:\left\langle \widehat
{x},z-x\right\rangle +\varphi(x)\leq\varphi(z),\;\forall\,z\in\mathbb{R}%
^{d}\},
\]

\item the domain of $\partial\varphi:$%
\[
\mathrm{Dom}\left(  \partial\varphi\right)  :=\left\{  x\in\mathbb{R}%
^{d}:\partial\varphi(x)\neq\emptyset\right\}  .
\]

\end{itemize}

We say that $\varphi$ is proper if $\mathrm{Dom}(\varphi)\neq\emptyset.$

We remark that under the assumptions on $\varphi$ (see for example
\cite{pa-ra/12}-Annex B):

\begin{itemize}
\item $\mathrm{Dom}\left(  \partial\varphi\right)  \neq\emptyset,$

\item $\partial\varphi:\mathbb{R}^{d}\rightrightarrows\mathbb{R}^{d}$ is a
maximal monotone operator,

\item $\mathrm{int}\left(  \mathrm{Dom}\left(  \varphi\right)  \right)
=\mathrm{int}\left(  \mathrm{Dom}\left(  \partial\varphi\right)  \right)  $
and $\overline{\mathrm{Dom}\left(  \partial\varphi\right)  }=\overline
{\mathrm{Dom}\left(  \varphi\right)  },$

\item $\varphi$ is continuous on $\mathrm{int}\left(  \mathrm{Dom}\left(
\varphi\right)  \right)  ,$

\item there exists a set $\left\{  \left(  v_{i},\alpha_{i}\right)
\in\mathbb{R}^{d}\times\mathbb{R}:i\in I\right\}  $ such that%
\[
\varphi(x)=\sup\left\{  \left\langle v_{i},x\right\rangle +\alpha_{i}:i\in
I\right\}  ,
\]

\item (Jensen inequality) if $a,b\in\mathbb{R}$, $a<b$, $y\in L^{\infty
}\left(  a,b;\mathbb{R}^{d}\right)  $ and $\rho\in L^{1}\left(  a,b;\mathbb{R}%
_{+}\right)  $ such that $\int_{a}^{b}\rho\left(  r\right)  dr=1,$ then%
\[
\varphi\left(  \int_{a}^{b}\rho\left(  r\right)  y\left(  r\right)  dr\right)
\leq\int_{a}^{b}\rho\left(  r\right)  \varphi\left(  y\left(  r\right)
\right)  dr.
\]

\end{itemize}

The Moreau regularization $\varphi_{\varepsilon}$ of the convex l.s.c.
function $\varphi$ is defined by%
\[
\varphi_{\varepsilon}(x)=\inf\Big\{\frac{1}{2\varepsilon}|z-x|^{2}%
+\varphi(z);z\in\mathbb{R}^{d}\Big\},\quad\mathbb{\,\varepsilon}>0.
\]
The function $\varphi_{\varepsilon}$ is a convex function of class $C^{1}$ on
$\mathbb{R}^{d}$; the gradient $\nabla\varphi_{\varepsilon\text{ }}$ is a
Lipschitz function on $\mathbb{R}^{d}$ with the Lipschitz constant equals to
$\varepsilon^{-1}$.

If we denote $J_{\varepsilon}\left(  x\right)  =x-\varepsilon\nabla
\varphi_{\varepsilon}(x)$ then on can easily prove (see for example
\cite{pa-ra/12}-Annex B) that for all $x\in\mathbb{R}^{d}$ and $\varepsilon
>0:$%
\[%
\begin{array}
[c]{l}%
\displaystyle\varphi_{\varepsilon}(x)=\dfrac{\varepsilon}{2}|\nabla
\varphi_{\varepsilon}(x)|^{2}+\varphi(J_{\varepsilon}\left(  x\right)
),\medskip\\
\displaystyle\nabla\varphi_{\varepsilon}\left(  x\right)  =\partial
\varphi_{\varepsilon}\left(  x\right)  ,\medskip\\
\displaystyle\nabla\varphi_{\varepsilon}(x)\in\partial\varphi(J_{\varepsilon
}\left(  x\right)  ).
\end{array}
\]
Hence $J_{\varepsilon}\left(  x\right)  =\left(  I+\varepsilon\partial
\varphi\right)  ^{-1}\left(  x\right)  $ and $\nabla\varphi_{\varepsilon
}(x)=\dfrac{x-J_{\varepsilon}\left(  x\right)  }{\varepsilon};$ $\nabla
\varphi_{\varepsilon}$ is called Moreau-Yosida approximation of $\partial
\varphi.$

If $K$ is a nonempty closed convex subset of $\mathbb{R}^{d},$ then the
\textit{convexity indicator function} $I_{K}:\mathbb{R}^{d}\rightarrow
(-\infty,+\infty]$ defined by%
\[
I_{K}\left(  x\right)  =\left\{
\begin{array}
[c]{ll}%
0, & \text{if }x\in K,\smallskip\\
+\infty, & \text{if }x\in\mathbb{R}^{d}\setminus K,
\end{array}
\right.
\]
is a proper convex l.s.c. function and%
\[
\partial I_{K}(x)=\left\{
\begin{array}
[c]{ll}%
0, & \text{if \ }x\in\mathcal{O},\smallskip\\
\mathcal{N}_{\overline{\mathcal{O}}}\left(  x\right)  , & \text{if \ }%
x\in\mathrm{Bd}\left(  \overline{\mathcal{O}}\right)  ,\smallskip\\
\emptyset, & \text{if \ }x\in\mathbb{R}^{d}\setminus\overline{\mathcal{O}},
\end{array}
\right.
\]
where $\mathrm{Bd}\left(  K\right)  $ is the boundary of $K$ and
$\mathcal{N}_{K}\left(  x\right)  $ is the outward normal cone to $K$ at
$x\in\mathrm{Bd}\left(  K\right)  .$

We remark that in the case $d=1$:

\begin{itemize}
\item $\varphi:\mathrm{Dom}(\varphi)\rightarrow\mathbb{R}$ is continuous and
$\varphi:\mathrm{int}\left(  \mathrm{Dom}(\varphi)\right)  \rightarrow
\mathbb{R}$ is locally Lipschitz continuous,

\item the left derivative $\varphi_{-}^{\prime}:\mathrm{Dom}(\varphi
)\rightarrow\left[  -\infty,+\infty\right]  $ and the right derivative
$\varphi_{+}^{\prime}:\mathrm{Dom}(\varphi)\rightarrow\left[  -\infty
,+\infty\right]  $ are well defined increasing functions and%
\[
u\in\left[  \varphi_{-}^{\prime}\left(  r\right)  ,\varphi_{+}^{\prime}\left(
r\right)  \right]  \cap\mathbb{R}\quad\Leftrightarrow\quad u\in\partial
\varphi\left(  r\right)  .
\]

\item if $a,b\in\mathbb{R},$ $a<b$ and $\varphi(x)=I_{[a,b]}(x)$, then
\[
\partial I_{[a,b]}(x)=\left\{
\begin{array}
[c]{ll}%
0, & \text{if\ \thinspace}x\in(a,b),\smallskip\\
(-\infty,0], & \text{if\ \thinspace}x=a,\smallskip\\
\lbrack0,+\infty), & \text{if\ \thinspace}x=b,\smallskip\\
\emptyset, & \text{if\ \thinspace}x\in\mathbb{R}\setminus\lbrack a,b].
\end{array}
\right.
\]

\end{itemize}

We recall a well known result.

\begin{lemma}
\label{anex-l1}Let $B$ be a convex subset of $\mathbb{R}^{d}.$ If
$\mathrm{int}\left(  B\right)  \neq\emptyset$, then%
\[%
\begin{array}
[c]{rl}%
\left(  j\right)  & \displaystyle\left(  1-\varepsilon\right)  \overline
{B}+\varepsilon\,\mathrm{int}\left(  B\right)  \subset\mathrm{int}\left(
B\right)  ,\quad\text{for all }\varepsilon\in(0,1],\medskip\\
\left(  jj\right)  & \displaystyle\mathrm{int}\left(  B\right)  \text{ is a
convex set and }\overline{\mathrm{int}\left(  B\right)  }=\overline{B}\,.
\end{array}
\]

\end{lemma}

\begin{proposition}
\label{equiv-subdiff}If $\varphi:\mathbb{R}^{d}\rightarrow]-\infty,+\infty
]$\ is a proper convex l.s.c. function with $\mathrm{int}\left(
\mathrm{Dom}\left(  \varphi\right)  \right)  \neq\emptyset$, $x\in
{\mathcal{{\mathbb{D}}}}([0,T];\mathbb{R}^{d})$, $k\in{\mathcal{{\mathbb{D}}}%
}([0,T];\mathbb{R}^{d})\cap BV([0,T];\mathbb{R}^{d})$ and \ $\ell
\in{\mathcal{{\mathbb{D}}}}([0,T];\mathbb{R})$ is a nondecreasing function,
then the following assertions are equivalent (with convention $\infty\cdot
0=0$):%
\begin{equation}%
\begin{array}
[c]{ll}%
\left(  a_{1}\right)  &
{\displaystyle\int\nolimits_{s}^{t}}
\left\langle z-x(r),dk(r)\right\rangle +%
{\displaystyle\int\nolimits_{s}^{t}}
\varphi(x(r))d\ell\left(  r\right)  \leq\varphi(z)(\ell\left(  t\right)
-\ell\left(  s\right)  ),\medskip\\
& \quad\quad\quad\quad\quad\quad\quad\quad\quad\quad\quad\quad\quad\quad
\quad\forall\,z\in\mathbb{R}^{d},\;\forall\,0\leq s\leq t\leq T,\medskip\\
\left(  a_{2}\right)  &
{\displaystyle\int\nolimits_{s}^{t}}
\left\langle y(r)-x(r),dk(r)\right\rangle +%
{\displaystyle\int\nolimits_{s}^{t}}
\varphi(x(r))d\ell\left(  r\right)  \leq%
{\displaystyle\int\nolimits_{s}^{t}}
\varphi(y(r))d\ell\left(  r\right)  ,\medskip\\
& \quad\quad\forall\,y\in\mathbb{D}\left(  \left[  0,T\right]  ;\mathbb{R}%
^{d}\right)  \;(\text{or }\forall\,y\in C([0,T];\mathbb{R}^{d})),\;\forall
\,0\leq s\leq t\leq T,\medskip\\
\left(  a_{3}\right)  &
{\displaystyle\int\nolimits_{0}^{T}}
\left\langle y(r)-x(r),dk(r)\right\rangle +%
{\displaystyle\int\nolimits_{0}^{T}}
\varphi(x(r))d\ell\left(  r\right)  \leq%
{\displaystyle\int\nolimits_{0}^{T}}
\varphi(y(r))d\ell\left(  r\right)  ,\medskip\\
& \quad\quad\quad\quad\quad\quad\quad\quad\quad\quad\forall\,y\in
\mathbb{D}\left(  \left[  0,T\right]  ;\mathbb{R}^{d}\right)  \;(\text{or
}\forall\,y\in C([0,T];\mathbb{R}^{d})).
\end{array}
\label{an-ineq-fi0}%
\end{equation}
and in all these equivalent cases we say that $dk(r)\in\partial\varphi\left(
x\left(  r\right)  \right)  d\ell\left(  r\right)  $ as measure on $\left[
0,T\right]  $ .
\end{proposition}

\begin{proof}
There exist $v\in\mathbb{R}^{d}$ and $a\in\mathbb{R}$ such that $\varphi
\left(  x\right)  +\left\langle v,x\right\rangle +a\geq0$, hence we can assume
$\varphi\geq0$, since the inequalities from (\ref{an-ineq-fi0}) are equivalent
to the corresponding inequalities with $\varphi$ replaced by $\varphi\left(
\cdot\right)  +\left\langle v,\cdot\right\rangle +a$ and $dk(r)$ replaced by
$dk(r)+vd\ell(r).$

We extend $y\left(  t\right)  =y\left(  0\right)  $ for $t\leq0$ and $y\left(
t\right)  =y\left(  T\right)  $ for $t\geq T.$ The same extension will be
considered for the functions $x$ and $k$ and $\ell.$

In all cases we can assume that
\[%
{\displaystyle\int\nolimits_{0}^{T}}
\varphi(y(r))d\ell\left(  r\right)  <\infty,
\]
and consequently
\[%
{\displaystyle\int\nolimits_{0}^{T}}
\varphi(x(r))d\ell\left(  r\right)  <\infty.
\]
Therefore we can assume $y\left(  t\right)  \in\overline{\mathrm{Dom}\left(
\varphi\right)  }$ and $x\left(  t\right)  \in\overline{\mathrm{Dom}\left(
\varphi\right)  }$ for all $t\in\left[  0,T\right]  .$ Also we can consider
only the case $z\in\mathrm{Dom}\left(  \varphi\right)  .$\medskip

We will show that $\;\left(  a_{1}\right)  \Leftrightarrow\left(
a_{3}\right)  $ and $\left(  a_{1}\right)  \Leftrightarrow\left(
a_{2}\right)  $.\medskip

$\left(  a_{1}\right)  \Rightarrow\left(  a_{3}\right)  :$

Let $y\in{\mathcal{{\mathbb{D}}}}\left(  \left[  0,T\right]  ;\mathbb{R}%
^{d}\right)  $ (respectively $y\in C\left(  \left[  0,T\right]  ;\mathbb{R}%
^{d}\right)  $).

Since $\mathrm{int}\left(  \mathrm{Dom}\left(  \varphi\right)  \right)
\neq\emptyset,$%
\[
\left(  1-\varepsilon\right)  \overline{\mathrm{Dom}\left(  \varphi\right)
}+\varepsilon~\mathrm{int}\left(  \mathrm{Dom}\left(  \varphi\right)  \right)
\subset\mathrm{int}\left(  \mathrm{Dom}\left(  \varphi\right)  \right)
,\quad\text{for all }\varepsilon\in(0,1].
\]
Consequently for any $x_{0}\in\mathrm{int}\left(  \mathrm{Dom}\left(
\varphi\right)  \right)  ,$ fixed, $t\longmapsto y_{\varepsilon}\left(
t\right)  :=\left(  1-\varepsilon\right)  y\left(  t\right)  +\varepsilon
~x_{0}:\left[  0,T\right]  \rightarrow\mathrm{int}\left(  \mathrm{Dom}\left(
\varphi\right)  \right)  $ is a c\`{a}dl\`{a}g (respectively continuous)
function with $y_{\varepsilon}\left(  t-\right)  \in\mathrm{int}\left(
\mathrm{Dom}\left(  \varphi\right)  \right)  $ for all $t\in\left[
0,T\right]  .$ Therefore $t\longmapsto\varphi\left(  y_{\varepsilon}\left(
t\right)  \right)  :\left[  0,T\right]  \rightarrow\lbrack0,\infty)$ is a
c\`{a}dl\`{a}g (respectively continuous) function.

Let $0=t_{0}<t_{1}<\cdots<t_{n}=T$ a partition of $\left[  0,T\right]  $ such
that $t_{i}-t_{i-1}=\frac{T}{n}$ for all $i\in\overline{1,n}.$

We write $\left(  a_{1}\right)  $ for $z=y_{\varepsilon}\left(  t_{i}\right)
$ and $s=t_{i-1},$ $t=t_{i}$ and we add the inequalities member by member for
$i\in\overline{1,n}$. It follows%
\[%
\begin{array}
[c]{l}%
\displaystyle\sum_{i=1}^{n}\left\langle y_{\varepsilon}\left(  t_{i}\right)
,k\left(  t_{i}\right)  -k\left(  t_{i-1}\right)  \right\rangle -\int_{0}%
^{T}\left\langle x(r),dk(r)\right\rangle +\int_{0}^{T}\varphi\left(  x\left(
r\right)  \right)  dr\medskip\\
\displaystyle\leq\sum_{i=1}^{n}\varphi\left(  y_{\varepsilon}\left(
t_{i}\right)  \right)  \left(  \ell\left(  t_{i}\right)  -\ell\left(
t_{i-1}\right)  \right)  .
\end{array}
\]
Passing to the limit as $n\rightarrow\infty$, we obtain using
(\ref{aprox-intr}-b),%
\[%
\begin{array}
[c]{l}%
\displaystyle\int_{0}^{T}\left\langle y_{\varepsilon}\left(  r\right)
,dk\left(  r\right)  \right\rangle -\int_{0}^{T}\left\langle
x(r),dk(r)\right\rangle +\int_{0}^{T}\varphi\left(  x\left(  r\right)
\right)  dr\medskip\\
\displaystyle\leq\int_{0}^{T}\varphi\left(  y_{\varepsilon}\left(  r\right)
\right)  d\ell\left(  r\right)  .
\end{array}
\]
Therefore%
\[%
\begin{array}
[c]{l}%
\displaystyle\left(  1-\varepsilon\right)  \int_{0}^{T}\left\langle y\left(
r\right)  ,dk\left(  r\right)  \right\rangle +\varepsilon\left\langle
x_{0},k\left(  T\right)  -k\left(  0\right)  \right\rangle -\int_{0}%
^{T}\left\langle x(r),dk(r)\right\rangle +\int_{0}^{T}\varphi\left(  x\left(
r\right)  \right)  dr\medskip\\
\displaystyle\leq\left(  1-\varepsilon\right)  \int_{0}^{T}\varphi\left(
y\left(  r\right)  \right)  d\ell\left(  r\right)  +\varepsilon\varphi\left(
x_{0}\right)  \left[  \ell\left(  T\right)  -\ell\left(  0\right)  \right]
\end{array}
\]
(in the second member we also used the convexity property of $\varphi).$

Passing to the limit as $\varepsilon\rightarrow0$, $\left(  a_{3}\right)  $
follows.\medskip

$\left(  a_{3}\right)  \Rightarrow\left(  a_{1}\right)  :$

Let $x_{0}\in\mathrm{int}\left(  \mathrm{Dom}\left(  \varphi\right)  \right)
.$ Let$\,\;\alpha_{n}\in$\ $C([0,T];\ \mathbb{R})$, $0\leq\alpha_{n}\leq1$,
and $\lim_{n\rightarrow\infty}\alpha_{n}\left(  r\right)  =\mathbb{1}%
_{(s,t]}\left(  r\right)  $ for all $r\in\left[  0,T\right]  .$ Let
$\varepsilon,\lambda\in\left(  0,1\right)  ,$ $x_{\lambda}\left(  r\right)
=\left(  1-\lambda\right)  x\left(  r\right)  +\lambda x_{0}$ and
$x_{\varepsilon,\lambda}\left(  r\right)  =\frac{1}{\varepsilon}\int
_{r}^{r+\varepsilon}x_{\lambda}\left(  u\right)  du,$ $r\in\left[  0,T\right]
.$ We have for all $r\in\left[  0,T\right]  :$%
\[
x_{\lambda}\left(  r\right)  ,x_{\varepsilon,\lambda}\left(  r\right)
\in\mathrm{int}\left(  \mathrm{Dom}\left(  \varphi\right)  \right)
\quad\text{and}\quad\lim_{\varepsilon\rightarrow0}x_{\varepsilon,\lambda
}\left(  r\right)  =x_{\lambda}\left(  r\right)
\]
So, since $\varphi:\mathrm{int}\left(  \mathrm{Dom}\left(  \varphi\right)
\right)  \rightarrow\lbrack0,\infty)$ is continuous, we have for all
$r\in\left[  0,T\right]  :$%
\[
\lim_{\varepsilon\rightarrow0}\varphi\left(  x_{\varepsilon,\lambda}\left(
r\right)  \right)  =\varphi\left(  x_{\lambda}\left(  r\right)  \right)
\]
and%
\[
0\leq\varphi\left(  x_{\varepsilon,\lambda}\left(  r\right)  \right)
\leq\frac{1}{\varepsilon}\int_{r}^{r+\varepsilon}\varphi\left(  x_{\lambda
}\left(  r\right)  \right)  du\leq\sup_{r\in\left[  0,T\right]  }%
\varphi\left(  x_{\lambda}\left(  r\right)  \right)  <\infty,
\]
because $\left\{  x_{\lambda}\left(  r\right)  :r\in\left[  0,T\right]
\right\}  $ is a bounded subset of $\mathrm{int}\left(  \mathrm{Dom}\left(
\varphi\right)  \right)  $.

In $\left(  a_{3}\right)  $ we replace $y$ by the continuous function
$r\mapsto(1-\alpha_{n}\left(  r\right)  )x_{\varepsilon,\lambda}\left(
r\right)  +\alpha_{n}\left(  r\right)  z.$ It follows
\begin{align*}
&  \int\nolimits_{0}^{T}\left\langle (1-\alpha_{n}\left(  r\right)
)x_{\varepsilon,\lambda}\left(  r\right)  +\alpha_{n}\left(  r\right)
z-x\left(  r\right)  ,dk(r)\right\rangle +\int\nolimits_{0}^{T}\varphi\left(
x\left(  r\right)  \right)  d\ell\left(  r\right) \\
&  \leq\int\nolimits_{0}^{T}\left(  \left(  1-\alpha_{n}\left(  r\right)
\right)  \varphi\left(  x_{\varepsilon,\lambda}\left(  r\right)  \right)
+\alpha_{n}\left(  r\right)  \varphi\left(  z\right)  \right)  d\ell\left(
r\right)
\end{align*}
By the Lebesgue dominated convergence theorem as $\varepsilon\rightarrow0$ and
$n\rightarrow\infty$ we infer%
\begin{align*}
&  \int_{0}^{T}(1-\mathbb{1}_{(s,t]}\left(  r\right)  )\left\langle
x_{\lambda}\left(  r\right)  -x\left(  r\right)  ,dk(r)\right\rangle +\int
_{0}^{T}\left\langle z-x\left(  r\right)  ,dk(r)\right\rangle \\
&  \leq\int\nolimits_{0}^{T}\left(  1-\mathbb{1}_{(s,t]}\left(  r\right)
\right)  \varphi\left(  x_{\lambda}\left(  r\right)  \right)  d\ell\left(
r\right)  +\int_{0}^{T}\varphi\left(  z\right)  d\ell\left(  r\right) \\
&  \leq%
{\displaystyle\int_{0}^{T}}
\left(  1-\mathbb{1}_{(s,t]}\left(  r\right)  \right)  \left[  \left(
1-\lambda\right)  \varphi\left(  x\left(  r\right)  \right)  +\lambda
\varphi\left(  x_{0}\right)  \right]  d\ell\left(  r\right)  +\varphi\left(
z\right)  \left[  \ell\left(  t\right)  -\ell\left(  s\right)  \right]
\end{align*}
Now, once again by the Lebesgue dominated convergence theorem, passing to the
limit as $\lambda\rightarrow0$ we get $\left(  a_{1}\right)  .$\medskip

$\left(  a_{2}\right)  \Rightarrow\left(  a_{1}\right)  :$ evidently
holds.\medskip

$\left(  a_{1}\right)  \Rightarrow\left(  a_{2}\right)  :$ clearly it can be
shown in the same manner as $\left(  a_{1}\right)  \Rightarrow\left(
a_{3}\right)  .$\hfill
\end{proof}

\section{Erratum}

Below we present briefly the updated statement of Proposition 13 from
\cite{ma-ra/10}. The only difference is due to the presence of the integral
with respect to the measure generated by the total variation $\updownarrow
{\hspace{-0.15cm}A^{t,x}-A^{t^{\prime},x^{\prime}}\hspace{-0.15cm}%
}\updownarrow$. The proof uses the same techniques and inequalities as
declared in \cite[Proposition 13]{ma-ra/10}.

\begin{proposition}
Under assumptions (\ref{Assumpt. 1})--(\ref{Assumpt. 6}), we have%
\begin{equation}
\mathbb{E}\underset{s\in\left[  0,T\right]  }{\sup}e^{\mu A_{s}^{t,x}}%
|Y_{s}^{t,x}|^{2}\leq C\left(  T\right)  \label{mg Y}%
\end{equation}
and%
\begin{equation}%
\begin{array}
[c]{l}%
\displaystyle\mathbb{E}\underset{s\in\left[  0,T\right]  }{\sup}e^{\mu
A_{s}^{t,x}}|Y_{s}^{t,x}-Y_{s}^{t^{\prime},x^{\prime}}|^{2}\leq\mathbb{E}%
\Big[e^{\mu A_{T}^{t,x}}\big |h(X_{T}^{t,x})-h(X_{T}^{t^{\prime},x^{\prime}%
})\big |^{2}\medskip\\
\displaystyle+\int_{0}^{T}e^{\mu A_{r}^{t,x}}\big |{\mathbb{1}_{[t,T]}%
}(r)f(r,X_{r}^{t,x},Y_{r}^{t,x},Z_{r}^{t,x})-{\mathbb{1}_{[t^{\prime},T]}%
}(r)f(r,X_{r}^{t^{\prime},x^{\prime}},Y_{r}^{t,x},Z_{r}^{t,x})\big |^{2}%
dr\medskip\\
\displaystyle+\int_{0}^{T}e^{\mu A_{r}^{t,x}}\big |{\mathbb{1}_{[t,T]}%
}(r)g(r,X_{r}^{t,x},Y_{r}^{t,x})-{\mathbb{1}_{[t^{\prime},T]}}(r)g(r,X_{r}%
^{t^{\prime},x^{\prime}},Y_{r}^{t,x})\big |^{2}dA_{r}^{t,x}\medskip\\
\displaystyle+\int_{0}^{T}e^{\mu A_{r}^{t,x}}{\mathbb{1}_{[t,T]}%
}(r)\big |g(r,X_{r}^{t,x},Y_{r}^{t,x})\big |^{2}d\updownarrow{\hspace
{-0.15cm}A^{t,x}-A^{t^{\prime},x^{\prime}}\hspace{-0.15cm}}\updownarrow
_{r}\Big].
\end{array}
\label{cont Y}%
\end{equation}

\end{proposition}

Since we do not know how to estimate the last integral, we cannot obtain the
continuity of the process $Y^{t,x}$ with respect to the initial data, and
therefore, the continuity of application $\left(  t,x\right)  \mapsto
Y_{t}^{t,x}=:u\left(  t,x\right)  $ does not follow anymore directly from
inequality (\ref{cont Y}). Our main result, Theorem \ref{cont of Y and u},
constitutes the correct proof of point $\left(  c\right)  $ of Corollary 14
from \cite{ma-ra/10} which states the continuity of the function $u$.

\begin{acknowledgement}
The authors would like to thank the referees for their valuable remarks and
comments which have led to a significant improvement of the paper.\medskip

The work of the authors was supported by the projects ERC-Like no.
1ERC/02.07.2012 and IDEAS no. 241/05.10.2011.
\end{acknowledgement}

\end{document}